\documentclass[11pt,reqno]{amsart}

\usepackage{natbib}
\bibpunct[, ]{(}{)}{,}{a}{}{,}%
\usepackage{setspace}
\usepackage{amsmath,amssymb,amsfonts, mathrsfs,  array, stmaryrd,  indentfirst, amsthm,  hyperref, comment, color, mathtools,bbm}
\usepackage{caption}
\usepackage{subcaption}
\usepackage[ruled,vlined]{algorithm2e}
\usepackage{graphicx, enumitem, tabularx}
\usepackage{textcomp}
\usepackage{xcolor}
\usepackage{verbatim}
\usepackage[normalem]{ulem}
\usepackage[margin=1.0in]{geometry}
\usepackage{todonotes}
\usepackage{fancyhdr}
\numberwithin{equation}{section}
\newtheorem{lemma}{Lemma}[section]
\newtheorem{thm}{Theorem}[section]
\newtheorem{defn}{Definition}[section]
\newtheorem{prop}{Proposition}[section]
\newtheorem{cor}{Corollary}[section]

\newtheorem{remark}{Remark}[section]

\newcommand{\comm}[1]{}
\makeatletter
\newcommand*{\rom}[1]{\expandafter\@slowromancap\romannumeral #1@}
\makeatother

\DeclarePairedDelimiter\ceil{\lceil}{\rceil}
\DeclarePairedDelimiter\floor{\lfloor}{\rfloor}

\definecolor{aoE}{rgb}{0.0, 0.5, 0.0}
\definecolor{dblue}{rgb}{0.0, 0.18, 0.39}
\newcommand{\Added}[1]{{{{\color{black}{#1}}}}}
\allowdisplaybreaks[4]

\begin{document}
	%{\color{red} red : to be deleted, old ways of grouping the regret. }\\
	%\Added{blue : edits and new content}\\
	%{\color{aoE} green: the improved algorithms and analysis.}

\pagestyle{plain}
\pagenumbering{arabic}

\title{Learning-based Optimal Admission Control in a Single Server Queuing System
}

\footnote{This is the final version of the paper. To appear in {\it Stochastic Systems.}}

\author[A. Cohen]{Asaf Cohen }
\address{Department of Mathematics\\
University of Michigan\\
%2854 East Hall, 530 Church Street,\\
Ann Arbor, MI 48109\\
United States}
\email{shloshim@gmail.com }

\author[V.G. Subramanian]{Vijay G. Subramanian}
\address{Department of Electrical Engineering and Computer Science\\
	University of Michigan\\
	Ann Arbor, MI 48109\\
	United States}
\email{vgsubram@umich.edu }
\author[Y. Zhang]{Yili Zhang}
\address{Department of Mathematics\\
	University of Michigan\\
	Ann Arbor, MI 48109\\
	United States}
\email{zhyili@umich.edu }

\thanks{A.C. is partially supported by the NSF grant DMS-2006305; V.S. is supported in part by NSF grants CCF-2008130, ECCS-2038416, CNS-1955777 and CMMI-2240981.}

\begin{abstract}
	We consider a long-term average profit maximizing admission control problem in an M/M/1 queuing system with unknown service and arrival rates. With a fixed reward collected upon service completion and a cost per unit of time enforced on customers waiting in the queue, a dispatcher decides upon arrivals whether to admit the arriving customer or not based on the full history of observations of the queue-length of the system. %The paper [Naor, \emph{Econometrica}, \citeyear{Naor}] 
    \cite[Econometrica]{Naor} showed that if all the parameters of the model are known, then it is optimal to use a static threshold policy---admit if the queue-length is less than a predetermined threshold and otherwise not. We propose a learning-based dispatching algorithm and characterize its regret with respect to optimal dispatch policies for the full information model of \cite{Naor}. %\cite[Naor, {\it Econometrica}, 1969]{Naor}.
    We show that the algorithm achieves an $O(1)$ regret when all optimal thresholds with full information are non-zero, and achieves an \Added{$O(\ln^{1+\epsilon}(N))$ regret for any specified $\epsilon>0$,} in the case that an optimal threshold with full information is $0$ (i.e., an optimal policy is to reject all arrivals), where $N$ is the number of arrivals.%and $\epsilon>0$.
\end{abstract}
\maketitle

\section{Introduction}
	We consider admission control for a first-in-first-out (FIFO) single-class single-server queuing model with Poisson arrivals and exponential service times. Specifically, there is a dispatcher that decides on admitting arrivals with the goal to maximize the long-term average profit -- each admitted arrival yields a positive reward $R$ (obtained after a customer finishes service), which is balanced by a holding cost for the (homogeneous) customers waiting in the queue. The buffer capacity of this queue is infinite and the dispatcher may decide upon arrivals to reject any customers joining the queue with the profit objective in mind. When the service and arrival rates are known, this model was studied %by Naor 
    in \cite{Naor}. In our investigation, we will consider the situation where the dispatcher does not have knowledge of either the arrival rate or the service rate. %knows the arrival rate but not the service rate. 
    One potential application is the job dispatching problem for online computing demands, especially when the computing servers are provided by a third-party cloud computing platform: the dispatcher may negotiate the reward and cost with the customers, and thus, have information (via market research) on the arrival rate of the jobs, but since the servers are provided by a third-party platform, the dispatcher may not know the service rate. Despite prior market research, it is, however, plausible that the dispatcher doesn't know the arrival rate accurately.
 % \VS{Give some justification as in talk about applications.}\YZ{Done.} 
	
	 %In %Naor's original paper~
  \cite{Naor} studied two problems: 1) the optimal policy for the self-optimization problem where customers are maximizing their own net (expected) profit so that a selfish Wardrop equilibrium is of interest; as well as 2) the optimal policy for the social welfare maximization problem where a dispatcher is aiming at maximizing the long-term average profit so that a social Wardrop equilibrium is of interest. In both problems, a threshold policy was shown to be optimal: 1) in the self-optimization problem, arrivals do not join the queue if the queue-length upon arrival is high enough; and 2) in the social-welfare maximization problem, the dispatcher doesn't admit arrivals whenever a threshold level is reached.  %In the paper, it was shown
        \cite{Naor} showed that the threshold for the social welfare maximization problem is not greater than the threshold for the self-optimization problem. Our investigation and the accompanying algorithm are primarily designed for the social welfare optimization problem where the dispatcher is interested in learning how to perform at the same level of efficiency as if knowing the actual arrival and service rate. Any learning-based algorithm will necessarily need exploration which could violate incentive-compatibility constraints (even \emph{ex-ante} and not only \emph{ex-post}) of individual utility maximizing agents. Hence, we do not consider the self-optimization version of the problem in this manuscript.
	 
	In our analysis, we will couple two queuing systems: a {\it learning system}, whose dispatcher does not know the arrival and service rate {\it apriori}, and a {\it genie-aided} system, whose dispatcher has full information of the model parameters. We refer to the corresponding algorithm and dispatcher of the two systems as the {\it learning algorithm}, {\it learning dispatcher} and {\it genie-aided algorithm}, {\it genie-aided dispatcher}, respectively. Our figure of metric at a given time $t$ will be the difference between the net expected profits of a genie-aided algorithm and the learning algorithm, i.e., the expected regret. 
	
	%\blue{AC: How about replacing the previous paragraph with: To estimate the regret, we couple between two systems: a {\it learning system}, whose dispacher does not know the service rate {\it apriori}, and a {\it genie-aided} system, whose dispacher has a full information of the rates. Our figure of metric at given time $t$ will be the difference between the net expected profits of the genie-aided algorithm and the learning algorithm, i.e., the expected regret. }
 
	 \paragraph{\bfseries Contributions:} We propose a learning-based dispatching algorithm that achieves an $O(1)$ regret when (genie-aided) optimal algorithms use a non-zero threshold, and achieves an \Added{$O(\ln^{1+\epsilon}(N))$ regret for any specified $\epsilon>0$} when it is optimal to use threshold $0$, where $N$ denotes the number of arrivals~\footnote{We show how to translate the regret from the number of arrivals to a time horizon.}; \Added{see Remark~\ref{rem:order} for a refinement on the achievable regret}. Our learning-based algorithm consists of batches with each batch being composed of an optional forced exploration phase (phase $1$) and an exploitation phase (phase $2$) whose length increases with batch index. The exploration phase is omitted if 
  % \comm{the logic determines that there will be an exploitation phases} 
  there are new samples collected from the exploitation phase that just ended. Our learning algorithm uses samples collected from all the exploitation phases as well as from any exploration phases; the former is important if the exploration phase is omitted. 
	 
	 For the system studied in \cite{Naor}, not all values of the unknown model parameters result in a unique optimal static threshold policy. For some specific choices of the model parameters, there exist two optimal static thresholds, and therefore all the policies that stochastically alternate between the two static optimal thresholds also achieve the optimal long-term average profit. As mentioned earlier, we are interested in analyzing the regret -- defined to be the difference between the expected profit of the learning and genie-aided systems. When the optimal policy is unique, there is no ambiguity in the definition of the regret as there is a fixed optimal policy to compare against. However, when there are multiple policies that are optimal, we need to specify a particular optimal policy that we are comparing against. Among the multiple optimal policies, we compare against a policy with a specific way of randomizing between the two static optimal thresholds, and then we prove that we can achieve similar regret as when there exists a unique optimal policy, which is of order $O(1)$ when both thresholds are positive, and of order \Added{$O(\ln^{1+\epsilon}(N))$ for any specified $\epsilon>0$} when $0$ is an optimal threshold and $N$ is the number of customers that have arrived; %\VS{and $N$ is ...}\YZ{added.}; 
     \Added{Remark~\ref{rem:order} applies with non-unique thresholds too}. 
  % \VS{Correct the regret for optimal threshold $0$ to the proper value that finally results.}\YZ{edited.}
	 
	 In our setting, we do not exclude the case where the genie-aided dispatcher uses a static threshold $0$, 
  % \comm{a $0$ threshold value}, 
  and hence, rejects all customers. This leads to a balancing act for the dispatcher: quickly transitioning to reject all customers if the true threshold is $0$ versus admitting customers infinitely often otherwise (based on the optimal threshold), and all of this while not being aware of the true optimal admission policy. With this in mind, for learning to not stall, the existence of the exploration phase is crucial when the true threshold is positive. A naive learning scheme that only uses the empirical average service time as an estimate of the unknown parameter may perform poorly: a few extremely long service times at the beginning may mislead the learning dispatcher to think that the service rate is low, and hence, result in it not accepting customers into the queue even when the genie-aided dispatcher uses a non-zero threshold; \Added{see plots in Section~\ref{section:Numerical}}. 
	 
	\paragraph{\bfseries Related work:} %\magenta{Topic 1 is before topic 2. No longer starting sentences by references. Updated description to reference \cite{Knudsen}, \cite{lippman76},\cite{Johansen80},\cite{Zhong2022}, \cite{Agrawal2019}. All references do not have author names mentioned, except \cite{Naor}. }
	On the topic of finding optimal controls vis-a-vis individual and social welfare maximization, there are many models that have studied generalizations of the model introduced in \cite{Naor}. %The paper 
    \cite{Knudsen} generalized the model in \cite{Naor} to multiple servers with a non-linear cost for customers waiting in the system. The reward for customers served is constant and customers arrive according to a Poisson process. The service times of the customers are exponentially distributed and are independent of the identity of the currently active server. %which server the customer is served at. 
    %The paper 
    \cite{Lippman76} studied a single queue model with Poisson arrivals and non-decreasing, concave service rate with respect to the number of customers in the system. The holding cost per unit of time for each customer is constant and the rewards for the customers entering the system are \emph{i.i.d.} random variables with finite mean. The authors first considered the discounted net profit in the finite horizon case (in terms of the total number of admissions and service completions), and then extended the analysis to the non-discounted and infinite horizon case. %\magenta{Done}\VS{Reword previous sentence}
	%The paper 
    \cite{Johansen80} studied the problem of finding the optimal admission policy of a system with general service and arrival processes. In the problem's setting, the net profit is discounted and the authors considered the finite horizon (in terms of the number of arriving customers) case. The rewards of the customers are \emph{i.i.d.} random variables with finite mean and the non-negative waiting cost is a function of the number of customers in the system as well as the total number of past arrivals. All the works \cite{Knudsen,Lippman76,Johansen80} compared the optimal policy for the individual and social welfare maximization problems and showed that the optimal policies for both optimization problems are threshold policies that depend on the rewards of customers. Moreover, they also showed that the optimal threshold for the social welfare maximization  problem is no greater than the individual maximization problem.\Added{ Assuming a  random arrival rate, \cite{unobqueue} showed that the optimal thresholds for the social welfare maximization problem are no larger than the individual maximization problem when either the queue length is observable or unobservable. They also showed that the optimal threshold for the revenue maximization problem may not coincide with the social welfare maximization problem when the queue is unobservable.}
	
	 Learning unknown parameters to operate optimally in queuing systems, and analyzing queuing systems with model uncertainly have both been studied under various settings -- see the tutorial \cite{walton2021} for a recent overview.  Our paper focuses on regret analysis in comparison with an optimal algorithm when the parameters are known. Under this framework, there is growing literature considering different models and various types of regret. %The paper
  \cite{subramanian2022} considered an Erlang-B blocking system with unknown arrival and service rates, where a customer is either blocked or receives service immediately. The authors proposed an algorithm that observes the system upon arrivals and converges to the optimal policy that either admits all customers when there is a free server, or blocks all customers.  In our setting, the queue has infinite capacity, customers may wait in the queue, and the dispatcher observes the whole history of the queue-length when making a decision. The reward of admitting a customer in both our paper and \cite{subramanian2022} is only realized in the future as it involves knowledge of service times and (in our case also) waiting times, and the expected net profit requires knowledge of the arrival and service rates; this precludes the direct use of Reinforcement Learning based methods discussed in \cite{sutton2018reinforcement} and \cite{bertsekas2019reinforcement}. Stability is always assured in \cite{subramanian2022} since the maximum system occupancy is bounded (finite number of servers with no queuing). %In our problem, however, whereas the system is stable under the optimal policy, the maximum queue-length over the unknown parameter regime under an arbitrary learning dispatcher may be unbounded 
The queuing system is stable under any optimal policy for the problem we consider. However, under an arbitrary learning dispatcher, the supremum of the queue-lengths may be unbounded when the service rate is unknown. We will discuss the impact of this on our analysis in Section~\ref{subsection:LearningAlg}. %The paper
  \cite{cmu1} first considered a discrete-time single-server queuing system with multi-class customers and unknown service rates, and then modified and extended their algorithms to parallel multi-server queuing systems, again with multi-class customers. In the model customers of class $i$ have (per unit-time) waiting cost $c_i$ when waiting in the queue and Bernoulli services with the service success probability at server $j$ being $\mu_{i,j}$ for class $i$ (i.e., geometrically distributed service-times). They proposed a $ c\mu$-rule-based algorithm that achieves constant regret compared to using the $c\mu$ rule with the true service rates. The $c\mu$ rule prioritizes the service of customers of type $i$ at server $j$ when $c_i\mu_{i,j}$ is higher. Optimality of the $c\mu$ rule  has been proved in various settings, especially in the single server case; see \cite{Smith_cmu,Shwartzcmu,Buyukkoc}  and \cite[Chapter 3]{Cox}. 
  \cite{Zhong2022} considered the problem of learning the optimal static scheduling policy in a multi-class many-server queuing system with time-varying Poisson arrivals. Customers of type $i$ have exponentially distributed patience with rate $\theta_i$ and exponentially  distributed service requirements with rate $\mu_i$. Unlike in \cite{cmu1}, where stability is not guaranteed for arbitrary scheduling policies, the impatience of the customers helps to stabilize the queue without any extra requirements on the scheduling policy. The authors compared their Learn-Then-Schedule learning algorithm with the $c\mu/\theta$-rule %\VS{What is $\theta$?}\YZ{Done.}
  and showed that their learning algorithm achieves a $\Theta(\log(T))$ regret where $T$ is the (finite) time-horizon. For a discrete-time multi-class parallel-server system, when compared to the algorithm which matches a queue to a server for which the success service probability is the highest among all possible matches of this queue to any other server, %the paper 
  \cite{Krishnasamy2021} used a multi-armed bandit viewpoint and proposed Q-UCB and Q-Thompson sampling algorithms that achieve  $O (\mathrm{poly}(\log (T))/T)$ queue-regret as the time horizon $T$ goes to infinity. %The paper 
  \cite{stahlbuhk2020} focused on a single-server discrete-time queue, and showed the existence of queue-length-based policies that can achieve an $O(1)$ regret. When each server has its own queue, \cite{choudhury2021} studied the discrete-time routing problem when service rate and queue-length are not known. Taking a Markov Decision Process (MDP) viewpoint, \cite{Agrawal2019} considered a discrete-time inventory control problem where orders to be made arrive with delay and the decision-maker observes solely the sales and not the demands. Thereafter, a holding cost is collected for each unit of the good that is in storage. At each time step, the decision-maker needs to make new orders and aims to minimize the total expected holding cost. The authors studied the problem of learning the proper units of orders to be made at each time step when the distribution of the demand is unknown. The algorithm they proposed achieves an $O(\sqrt{T})$ regret (for horizon $T$) when compared to the best base-stock policy. 
  
  \Added{With the goal of stabilizing the queues and also minimizing penalties enforced in a discrete-time system, \cite{maxweightNeely} proposed an algorithm that learns a set of Max-Weight functionals that depend on the unknown underlying distribution, and make two-stage decisions (which are shown to correspond to scheduling choices in illustrated examples). The proposed algorithm stabilizes the system considered and achieves at most linear regret in the accumulated penalties when compared to the optimal controller. Considering a scheduling problem with unknown arrival and channel statistics, \cite{Krishnasamy_2018MW} studied a wireless scheduling problem with switching costs. Under their proposed explore-exploit policy with the exploration probability going to $0$ slowly and together with a Max-Weight scheduling policy using learned statistics, the network is shown to be stable and the algorithm achieves at most linear regret in the accumulated switching and activating  cost when comparing to the optimal scheduler with the knowledge of the model statistics. The error bound on the long-term average in both works can be made arbitrarily small (when compared to the optimal cost) by changing algorithm parameters. Instead of having explicit exploration, \cite{LeiYing23} studied a discrete-time multi-server queuing system, and proposed a Max-Weight with discounted Upper Confidence Bound (UCB) scheduling algorithm. Their main result shows the stability of the queuing system under the proposed algorithm. }%The paper 

  \Added{There is a growing literature that studies online dynamic pricing in service systems using queuing models. We discuss some relevant recent work next. The authors of \cite{chen2022online} considered optimal pricing with congestion in a $GI/GI/1$ queue where there is unit cost depends on the service rate, the arrival rate depends on the service fee, and where customers experience congestion given by the average queue-length of the system. As the cost as a function of the service rate and the dependence of the arrival rate in chosen price is unknown, the authors proposed a gradient-based online learning algorithm that achieves a sub-linear regret when compared with the accumulated profit obtained with the optimal service rate and fee (using steady-state quantities). Also considering an online learning version of finding a proper price amongst a finite set of prices, \cite{JiaPricing} considered a multi-server queuing model with Poisson arrivals and exponential services where the dependence of arrival and service rates price chosen is unknown (with the values unknown as well but such that the load for each choice is strictly less than $1$). Two online batch processing algorithms based on UCB and Thompson sampling are proposed in \cite{JiaPricing}. Both algorithms achieve sub-linear regret (optimal up to logarithmic factors) when compared with the accumulated profit achieved by the optimal price choice.}
  
	 In our work, we consider a paradigm where there's uncertainty in the model parameters. A different type of uncertainty, often called  Knightian uncertainty, was studied in \cite{atar2022scheduling}, \cite{cohen2019}, \cite{cohen20192}, and \cite{cohen2021} for multi-class queuing systems in the heavy traffic regime. In these models, the decision-maker is looking for robust control for a class of models. The uncertainty is modeled by including an adversarial player who chooses a worst-case scenario. Hence, the robust control problem is formulated via a stochastic game between the decision maker and the adverse player. Optimality is then characterized by studying Stackelberg equilibria.
	
	\paragraph{\bfseries Outline of paper:}  In Section \ref{section:ProblemSetUp} we introduce the model, propose our learning algorithm and state our main results. In Section \ref{section:Preliminary}, we state some preliminary results, including the properties of the coupling introduced in Section \ref{section:ProblemSetUp}. Section \ref{section:UniqueThresh} and \ref{section:MultiThresh} are devoted to the analysis of our learning algorithm and include the proof of our main results. Section \ref{section:Numerical} provides the finite-time performance of our algorithm via simulations. In section \ref{section:Conclusion} we summarize our result. 
	
	\section{The learning problem and the main results}\label{section:ProblemSetUp} 
	
	In this section, we introduce the stochastic model and the learning algorithm. Specifically, in Section \ref{subsection:KnownParam} we introduce the optimal admission control problem for the queuing system studied %by Naor 
    in~\cite{Naor}. In this model, all the parameters are known. The same model but with unknown service and arrival rate %\YZ{updated}
    is introduced in Section \ref{subsection:LearningSystem}. We couple the models with known and unknown parameters so that we can characterize the regret of our learning dispatcher. Our learning algorithm is provided in Section \ref{subsection:LearningAlg}. Finally, in Section \ref{subsection:MainResults} we state the main results.
    % \comm{\VS{Consistently use dispatcher instead of scheduler.}}
	
	\subsection{The stochastic model with known parameters}\label{subsection:KnownParam}
	% The following model was first introduced by Naor \cite{Naor}.
	%Naor 
    \cite{Naor} studied the self-optimization and social welfare maximization problems for the following model. Homogeneous customers arrive at a single server queue according to a Poisson process with a rate $0<\lambda<\infty$. When a customer arrives, and only then, the dispatcher decides whether to admit this customer to the queue or not.  A customer that is not admitted (i.e., rejected) leaves and does not return. An admitted customer remains in the queue until being served. Upon service completion, the dispatcher receives a reward $R>0$. Once the service is completed, the customer leaves the queue. The dispatcher suffers from a waiting/holding cost at the rate of $C>0$ per time unit for each customer in the queue until service completion. The service requirements for the customers are \emph{i.i.d.}~EXP($\mu$) (i.e., exponentially distributed random variables with the rate $0<\mu<\infty$). The dispatcher's goal is to maximize the social welfare, i.e., to maximize the long-term average profit accrued by serving customers -- the ergodic-reward maximization problem.  Let $Q(t)$ denote the queue-length of the system at time t, $N_A(t)$ denote the number of customers that arrived at the system until and including time $t$, then for an admission policy $\rho$ the long-term average profit can be expressed as:
	\begin{align}
		\liminf_{T\rightarrow \infty}\frac{1}{T}\left(\sum_{i = 1}^{N_A(T)} R\mathbbm{1}_{\{\text{Policy $\rho$ admits customer $i$}\}} - \int_{0}^{T} CQ(t)dt\right), \label{def:Long-termAverageProfit}
	\end{align}
where throughout the paper, $\mathbbm{1}_{ A}$ is the indicator function of event $A$: namely, $\mathbbm{1}_{ A} = 1$ if $A$ happens and $0$ otherwise.
	
	The optimal admission policy of the dispatcher in \cite{Naor} is a static \emph{threshold policy}. That is, there is a threshold that depends on the parameters of the model, such that the dispatcher admits an arriving customer if and only if the queue-length upon arrival is strictly below this threshold. %Naor~
    \cite{Naor} studied optimal admission control for the ergodic cost minimization problem by choosing the best threshold value among all possible thresholds. %This is valid because it is easy to show that the optimal policy must be a threshold policy. 
	 When the dispatcher uses a static threshold policy with a threshold $K$, the result is an $M/M/1/K$ queueing system. The queue-length process of such a system has a stationary distribution and is also ergodic. Note that the optimal threshold can then be determined by computing the expected reward using the stationary distribution of the $M/M/1/K$ queueing system for all possible values of $K$. Using this logic %Naor~
    \cite{Naor} characterized the optimal threshold via the function $V: \mathbb{N} \times (0,\infty)^2 \rightarrow [0,\infty)$, \Added{given by:
	%{\small	
	\begin{equation}\label{fun:V}
			V(K,y,z)= 
			\begin{cases}
				\frac{K(y-z) - z(1-(z/y)^K)}{(y-z)^2}, & \text{if } y\neq z,\\
				\frac{K(K+1)}{2y}, & \text{if } y = z.
			\end{cases}
		\end{equation}
The following proposition states a few properties of this function $V(\cdot,\cdot,\cdot)$. 
\begin{prop}\label{prop:FuncVProp}The following hold:%$\quad$
	\begin{itemize}[leftmargin=*]
		\item[1.] For all fixed $K$, the function $V(K,\cdot,\cdot)$ is continuous in its domain.
		\item[2.] For all fixed $(y,z)$, $V(K, y,z)$ is strictly increasing in $K$. 
	\end{itemize}
\end{prop}
\begin{proof}
	Note that when $K = 0$, $V(0, y ,z ) = 0 $  for all $(y,z) \in (0,\infty)^2$. Consider any  point $(K,y,z) \in \mathbb{N}^{+} \times (0,\infty)^2 $. In order to prove the continuity of $V$, it will be easier to rely on an alternative formulation of $V$ based on the stationary distribution which we now provide. Let $p^{K}_{i}$ denote the stationary probability of having the queue-length equal to $i$ and let $E_{K}$ denote the stationary expected queue-length when using the threshold policy with a threshold $K$. One can show that:
	\begin{align*}
 V(K,y,z)  = \frac{E_{K-1}- E_{K}}{p^{K}_{K} - p^{K-1}_{K-1}}\frac{1}{z},\qquad
	\text{ where } \qquad
 	    p^{K}_i  = \frac{(z/y)^i}{\sum_{i = 0}^{K}(z/y)^i}\quad \text{ and }\quad
		E_{K} = \sum_{i = 0}^{K}ip_{i}^K.
	\end{align*}
	%\VS{Double-check $V(K,y,z)$ above. It doesn't agree with \eqref{fun:V}. \eqref{fun:V} is not a function of just $z/y$ but the expression above is.}\YZ{it should be 1/z instead of y/z I think.} \VS{That may be the correct. Can you confirm? \eqref{fun:V} for $K=1$ yields $1/y$. Maybe you can use that as a check. Try other values of $K$ as well.}\YZ{finally figured that I messed up the index. the previous expression is actually $V(K+1,y,z)$.} 
    Clearly, when $(y,z)\in (0,\infty)^2$, $1/z$, $E_{K}$, $E_{K-1}$, $p^{K}_{K}$ and $p^{K-1}_{K-1}$ are all continuous in $(y,z)$. Moreover, $p^{K-1}_{K-1} \neq p^{K}_{K}$ for all $(y,z)\in (0,\infty)^2$.
	%Consider any point $(K,y^*,y^*) \in \mathbb{N} \times (0,\infty)^2 $. We would like to show that the function $V$ is continuous at this point, i.e. we would need to find:
	%\begin{align*}
	%	&\lim_{(K,y,z)\rightarrow (K,y^*,y^*)}  V(K,y,z)\\
	%	=&\lim_{(K,y,z)\rightarrow (K,y^*,y^*)}  \frac{K(y-z)-z(1-(z/y)^K)}{(y-z)^2}\\
	%	=& \lim_{(K,y,z)\rightarrow (K,y^*,y^*)}
	%	\frac{1}{y} \frac{K(1-(z/y)) -z/y(1-(z/y)^{K})}{(1-(z/y))^2}.\\
	%\end{align*}
	%Note that $(y,z) = (0,0)$ is not in the domain of the function, and for any $y^*\neq 0$, $(y,z) \rightarrow (y^*,y^*)$ is equivalent as $y\rightarrow y^*$ and $z/y\rightarrow1$. 
	%Let $\rho = z/y$. 
	%\begin{align*}
	%	& \lim_{(K,y,z)\rightarrow (K,y^*,y^*)} \frac{K(1-(z/y)) -z/y(1-(z/y)^{K})}{(1-(z/y))^2}\\
	%	 =& \lim_{(K,\rho) \rightarrow (K,1)} \frac{K(1-\rho) -\rho(1-\rho^{K})}{(1-\rho)^2}.
	%\end{align*}
%When $K = 0$, clearly $(K(1-\rho) -\rho(1-\rho^{K}))/((1-\rho)^2) = 0$ for all $\rho\neq 1$, and the limit is $0 = K(K+1)/2$.  When $K \neq 0$, we can use L'Hopital's rule twice and compute that:
%\begin{align*}
%	\lim_{(K,\rho) \rightarrow (K,1)} \frac{K(1-\rho) -\rho(1-\rho^{K})}{(1-\rho)^2} = \frac{K(K+1)}{2}.
%\end{align*}
%We have showed the function $V$ is continuous. 
% \VS{This only shows continuity at $V(K,y,y)$ which is the $y=z$ line but this doesn't show that $V(K,y,z)$ is continuous in both arguments $(y,z)$. However. the other part might be easy given that things are well-defined. It maybe better to use the long expression using the stationary distribution as you will not have to do any cancellations using L'Hopital's rule.}\YZ{edited.}

Now, let us consider the function $V(K, y,z)$ for any fixed $(y,z) \in (0,\infty)^2$. To show the monotonic increasing property, we consider the function $f: [0,\infty) \rightarrow[0,\infty)$, $f(K) = V(K,y,z)$ by extending the definition of $V(\cdot,\cdot,\cdot)$ to real-valued $K$. From \eqref{fun:V}, it follows that  when $y = z$, $f(K)$ is strictly increasing. Now, we focus on the case $y \neq z$. Computing the derivative of $f(K)$, we get:
\begin{align*}
	f'(K) = \frac{(y - z) + z(z/y)^{K} \ln(z/y)}{(y-z)^2}.
\end{align*}
Using the inequality $\ln(x) > 1 - 1/x$ for all $x >0, x\neq 1$, we get:
\begin{align*}
	 (y - z) + z(z/y)^{K} \ln(z/y) >  (y - z) + z(z/y)^{K} (1- y/z)%\\
	 = (y-z)(1 - (z/y)^K)%\\
	 > 0,
\end{align*}
for all $y\neq z$. This shows that $f(K)$ is strictly increasing, which implies that $V(K, y,z)$ is strictly increasing in $K$ for all fixed $(y,z) \in (0,\infty)^2$.
\end{proof}

Using these properties %Naor 
\cite{Naor} showed that for every service rate $\mu$ and arrival rate $\lambda$ the following inequalities for integer $x$
%{\small
	\begin{align}\label{eq:ThresholdIneq}
		V\left(x,\mu,\lambda \right)\leq \frac{R}{C} < V\left(x+1,\mu,\lambda \right)
	\end{align} 
	have a unique solution $x = \bar{K}$, and this $\bar{K}$ is an optimal admittance threshold for the problem considered. Moreover, when $V\left(\bar{K},\mu,\lambda \right)< R/C$, the optimal threshold is unique. However, when $V\left(\bar{K},\mu ,\lambda \right) =R/C$, both $\bar{K}$ and $\bar{K}-1$ are optimal thresholds; hence, any policy that randomizes between the two thresholds at each arrival %\YZ{edited.} 
    is also optimal\footnote{We discuss what we mean by ``optimal" in Remark~\ref{rem:optimal} after we specify the strategy to which we compare our learning algorithm in the case that there are multiple optimal thresholds.}.
 
	Let $m:=1/\mu$ and $\nu:=1/\lambda$ denote the average service time and the average inter-arrival times respectively.  Consider a pair of the true service  and arrival rates $(\mu, \lambda)$ for which there exists a unique optimal threshold and the corresponding $\bar{K}$ satisfying \eqref{eq:ThresholdIneq} with strict inequalities. %\YZ{edited} 
    Proposition \ref{prop:FuncVProp} implies that there exist $\delta_1>0$ and $\delta_2 >0$, both depending on $\mu$ and $\lambda$, such that for all pairs of points $(\hat{m},\hat{\nu})$, where \begin{align}\label{eq:RegionOfParam}
		m-\delta_1<\hat{m}<m+ \delta_1 \quad \text{ and } \quad \nu-\delta_2<\hat{\nu}<\nu+ \delta_2, 
	\end{align}
	we have: 
    \begin{align}\label{eq:CorrectThreshold}
		V(\bar{K}, 1/\hat{m},1/\hat{\nu}) < \frac{R}{C} < V(\bar{K}+1, 1/\hat{m},1/\hat{\nu}). 
	\end{align}
	That is, if one can estimate the average service time and the average inter-arrival time accurately so the inequality \eqref{eq:RegionOfParam} is satisfied, one can obtain the corresponding $\bar{K}$ by solving \eqref{eq:ThresholdIneq} using $1/\hat{m}$ and $1/\hat{\nu}$ instead of $\mu$ and $\lambda$. %\VS{All parts here that are new should be in blue for revision.}\YZ{Edited.}
	
	When equality holds in \eqref{eq:ThresholdIneq}, for pairs of the true service and arrival rates $(\mu, \lambda)$  and the corresponding $\bar{K}$ that satisfies $V(\bar{K}, \mu,\lambda) = R/C$, there exist $\tilde{\delta}_1>0$ and $\tilde{\delta}_2>0$, both depending on $\mu$ and $\lambda$, such that for all pairs of points $(\hat{m}, \hat{v})$ where
	\begin{align}\label{eq:ICRegionOfParam}
		m-\tilde{\delta}_1<\hat{m}<m+ \tilde{\delta}_1 \quad \text{ and }\quad  \nu-\tilde{\delta}_2<\hat{\nu}<\nu+ \tilde{\delta}_2, 
	\end{align}
	we have:
	\begin{align}\label{eq:ICCorrectThreshold}
		V(\bar{K}-1, 1/\hat{m},1/\hat{\nu}) <  \frac{R}{C} < V(\bar{K}+1,  1/\hat{m},1/\hat{\nu}).
	\end{align}
	That is, as long as the estimated average service time and average inter-arrival time are accurate enough to satisfy inequality \eqref{eq:ICRegionOfParam}, the integer solved from inequality \eqref{eq:ThresholdIneq} using $1/\hat{m}$ and $1/\hat{\nu}$ in place of $\mu$ and $\lambda$ will be in the set of optimal thresholds, that is, $\{ \bar{K}-1, \bar{K}\}$.
}

	\subsection{The learning system and the genie-aided system}\label{subsection:LearningSystem}
	We assume that the reward $R$ and the cost per time unit $C$ are known to the learning dispatcher, but neither the service rate $\mu$ nor the arrival rate $\lambda$. Consider again the potential application of job dispatch for online computing demands. When the computation clusters are provided by a third-party cloud computing platform, the dispatcher of the online computing jobs may not have knowledge about the configuration of the servers and their service rate. \Added{The dispatcher may also be unfamiliar with the customer type that demands services, and therefore may only possess limited knowledge of the arrival rate. }%\VS{Explain why this is justified using some examples as in the talk.} \YZ{Done} 
 In our model, the dispatcher continuously observes the queue-length \Added{and past admission control decisions}. Hence, we restrict the dispatcher to admission controls that at the time of a new arrival, admit or reject based on the entire history of the queue-length until the arrival time, \Added{and also the past admission control decisions}. We call such controls {\it admissible}. Note that based on the FIFO serving discipline that's used, we can infer the time to enter service for all customers entering service by time $t$, and also the departure epochs for all the customers departing (after completing service) by $t$. Therefore, when a new customer arrives, the dispatcher can estimate the  mean service time (also the service rate) using the service times of the customers that have departed before the new arrival, and use it for admission control. \Added{Further, knowledge of all past admission control decisions enables the dispatcher to obtain information on all past inter-arrival times, which will then be used to compute the statistics for the arrival process, i.e., the arrival rate.}
	
	We measure the performance of a policy chosen by the learning dispatcher by the regret it incurs in comparison to an optimal policy. Specifically, we use the difference between the expected net profit under the given learning-based control/policy and the best expected net profit the dispatcher could have obtained had it known the \Added{parameters $\mu$ and $\lambda$}. To rigorously define the regret, we introduce some relevant processes for both the genie-aided and the learning systems.
	
	We will use the marker $\bar{\quad}$ to denote processes associated with the {\it genie-aided system} (dispatcher knows $\mu$ and $\lambda$). The processes without a marker are associated with the {\it learning system} (dispatcher does not know $\mu$ and $\lambda$). %Our learning dispatcher also uses a threshold policy. 
	We let 
	\begin{itemize}[leftmargin=*]
		\item $\bar{Q}(t)$ and $Q(t)$ denote the queue-length at time $t$; 
		\item $\bar{Q}_i$ and $Q_i$ denote the queue-length right before the arrival of the $i^{th}$ customer; 
		\item $\bar{N}_A(t)$ and $N_A(t)$ denote the number of customers that have arrived at the system until and including time $t$; 
		\item $\bar{N}_{\mathrm{join}}(t)$ and $N_{\mathrm{join}}(t)$ denote the number of customers that have joined the queue until and including time $t$; % \VS{Use  mathrm join everywhere these variables are used}\YZ{Done.}
		\item$\bar{T}^A_i$ and $T^A_i$ denote the arrival time of the $i^{th}$ customer to the system (i.e., $\bar{T}^A_i = \inf\{ t: \bar{N}_A(t)\geq i\}$ and $T^A_i = \inf\{ t: N_A(t)\geq i\}$, respectively); 
		%\item $\bar{T}^J_i$ and $T^J_i$ denote the arrival time of the $i^{th}$ customer that joined the queue (i.e. $T^J_i = \inf \{t: N_J(t) \geq i  \}$); 
		%\item $\bar{D}(t)$ and $D(t)$ denote the total number of customers that were being served until time $t$;  
		%\item$\bar{S}_i$ and $S_i$ denote the service time of the $i^{th}$ customer that joined the queue; and 
		%{\item $\bar{W}_i$ and $W_i$ denote the waiting time of the $i^{th}$ customer that joined the queue,}
		\item $\bar{K}_i$ and $K_i$ denote the threshold policy used by the respective dispatchers at the arrival of the $i^{th}$ customer.
	\end{itemize}
	%{At present we omit symbols of the controls used since in the sequel we will work with a fixed control for each system \textcolor{red}{except in Section \ref{section:MultiThresh}.}  In order to compare the two systems next we present a coupling. }

	\subsubsection{A coupling between the two systems}\label{coupling}
	%{\blue{AC: this coupling is only relevant for the strict inequality case, unless we allow for a general optimal control for the genie-aided system, which I guess we can keep it here general and only in the proof split to cases.}}
	Consider a probability space $(\Omega,\mathcal{F},\mathbb{P})$ rich enough to support two independent Poisson processes $(P(t))_{ t\geq 0}$ and $(N_A(t))_{t\geq0}$ with rates $\mu$ and $\lambda$, respectively. Set  $ \bar{N}_A = N_A $ so the arrival processes to both systems are the same. %\VS{Use a different symbol than $L$ as it is used for something else later on.}\YZ{Done. No longer using L. } %We will use the process $P$ to couple the services in the two systems.
	 Let $T^{PD}_i$ denote the $i^{th}$ jump time of $P$. The service requirements of the customers that are being served at time $t$ by all systems to be analyzed are determined as follows: the head of the line customer of each system (assuming not empty) completes her service at the time of the next jump of $P(t)$. Note that it may be the case that the services of the currently in-service customers are initiated at different times for the learning and genie-aided systems. Nevertheless, because the exponential distribution is memoryless, this does not change the distribution of the random process corresponding to the two systems, and in particular the distribution of the customer's service times. In other words, the time between the beginning of a service of a customer and the next jump of $P$ is  EXP($\mu$) distributed. Hence, we refer to $P(t)$ as the {\it potential departure process}, and to $\{T^{PD}_i\}_{i\geq1}$ as the  {\it potential departure times}, i.e., when there is a jump in $P$, and the queue-length is larger than 0, there will be a departure of a customer, but when the queue-length is 0, i.e., no customer is being served, this potential departure is \emph{wasted}. Therefore, $\{P(T^A_i)-P(T^A_{i-1})\}_{i \geq 1}$ is the number of potential services between two consecutive arrivals for both systems.%{\blue{AC: do we need the following sentence?} For the rest of the paper, we assume that the genie-aided system and the learning system are coupled as described above \blue{AC: below?}. }

	Now, we will use the underlying processes $\bar{N}_A = N_A$ and $P$ to couple the queue-length processes of both systems assuming that a threshold policy is used in each system. 
	Consider a sequence of random variables $\{ K_i\}_{i\geq 0}$ taking vales in $\mathbb{N}$, such that each $K_i$ is measurable with respect to the filtration generated by the queue-length until time $T^A_i$: since $T^A_i$ is a stopping time for the filtration being used, we can define the $\sigma$-algebra $\mathcal{F}_{T^A_i}:=\mathcal{F}_i$ (for short) using the original filtration $\mathcal{F}_T=\sigma(Q(t): t\leq T)$ in the usual way (See \cite{Durrett}). We use $\{ K_i\}_{i\geq 0}$ as a sequence of thresholds. Similarly, we use $\{ \bar{K}_i\}_{i\geq0}$ to denote the sequence of thresholds used by the genie-aided dispatcher. We refer to any such $ \{ K_i\}_{i\geq 0}$ as a threshold policy. For the coupled genie-aided and learning systems, we have the following: for any $i \geq 1$,
	%{\small
	\begin{align*}
		Q_i &= \left( Q_{i-1}+ \mathbbm{1}_{\{ Q_{i-1} < K_{i-1} \}} -(P(T^A_{i})- P(T^A_{i-1}))\right)^{+},\\
		\text{and }\bar{Q}_i &= \left( \bar{Q}_{i-1}+ \mathbbm{1}_{\{ \bar{Q}_{i-1} < \bar{K}_{i-1}\}} -(P(T^A_{i})- P(T^A_{i-1}))\right)^{+},
	\end{align*}
	%}%
	where for $x\in \mathbb{R}$, $(x)^+:=\max(x,0)$. Similarly, we have: 
	%{\small
	\begin{align}
		Q(t) &= \left(Q_n + \mathbbm{1}_{\{Q_n < K_n\}} - (P(t) - P(T^A_n))\right) ^{+} ,\label{def:q(t)1}\\
		\text{and }\bar{Q}(t) &= \left(\bar{Q}_n + \mathbbm{1}_{\{\bar{Q}_n < \bar{K}_n\}} - (P(t) - P(T^A_n))\right) ^{+} ,\label{def:q(t)2}
	\end{align}
	%}%
	where $n := \max\{m: T^A_m < t \}$. Once the initial queue-lengths $Q_0$ and $\bar{Q}_0$ are specified in $\mathbb{Z}_+$, by induction one can show that the processes $\{Q_i\}_{i\geq 0}$ and  $\{\bar{Q}_i\}_{i\geq 0}$ are well-defined, and using these $\{Q_t\}_{t\geq 0}$ and $\{\bar{Q}_t\}_{t\geq 0}$ are also well-defined. 
% 	We refer to any such $ \{ K_i\}_{i\geq 0}$ as a {\it threshold control}.
	
	\subsubsection{The regret.} Let $\mathbb{E}[\cdot]$ be expectation associated with $(\Omega,\mathcal{F}, \mathbb{P})$. Then, the regret is given by  $$G(t) := \mathbb{E}\left[ R\, \bar{N}_\mathrm{join}(t) - C\int_{0}^{t} \bar{Q}(u)du - \left( R\, N_\mathrm{join}(t) -  C\int_{0}^{t} Q(u) du\right)  \right] .$$ 
	This definition of the regret compares the net reward processes of the learning and genie-aided systems: if the learning-based admission control algorithm achieves the same long-term average profit, then this will allow us to estimate the sub-linear offset. The genie-aided dispatcher uses a static threshold policy that maximizes the long-term average profit described in \eqref{def:Long-termAverageProfit}. Note that when equality does not hold in \eqref{eq:ThresholdIneq}, the genie-aided policy is unique so there is no ambiguity in the definition of the regret. In this case, $\bar{K}_i  \equiv \bar{K}$, where $\bar{K}$ uniquely satisfies inequality \eqref{eq:ThresholdIneq}. However, when equality holds in \eqref{eq:ThresholdIneq}, the genie-aided policy is not unique.  We will compare our learning algorithm with a particular optimal genie-aided system that will be specified in Section \ref{section:MultiThresh}. %{ \blue{AC: all over the paper. It looks like that the proper way to approach it is to consider an ergodic cost also in Naor, which is equivalent to his rate per unit time. Does he prove it? Then, the terminology is not asymptotically optimal. I don't really understand what is optimal and what is asymptotically optimal in our model. This may require some definitions, unless we somehow ignore this point and towards the end mention that they are equivalent without a proof (assuming that Naor did not provide one).} \VS{A clarification. All of these are considering the ergodic cost problem. Naor characterizes the optimal control although he doesn't exactly use DP and HJB to prove it, but this is likely an easy thing. Asymptotically optimal is then used for other control policies, including learning based ones, where the ergodic cost is the same as the optimal value but there can be an offset, which is the regret. This is standard usage in the literature.}}
	
	Consider a threshold policy for the learning system, $\{K_i\}_{i\geq 0}$%\comm{Recall the definition of $\bar{K}$ from \eqref{eq:ThresholdIneq}. When equality does not hold in \eqref{eq:ThresholdIneq},}
    , and a threshold policy for the genie-aided system,  $\{\bar{K}_i\}_{i\geq 0}$, the regret can be estimated as:
%{\tiny
		\begin{align}
		G(t)& =  \mathbb{E}\left[ R\sum\limits_{i = 1}^{N_A(t)} \left(\mathbbm{1}_{\{\bar{Q}_i<\bar{K}_i\}}-  \mathbbm{1}_{\{Q_i<K_i\}}\right)\right]- \mathbb{E}\left[C\int_{0}^{t} \Big( \bar{Q}(u)-Q(u) \Big) du \right] \nonumber \\
		& \leq  \mathbb{E}\left[  R\sum\limits_{i = 1}^{N_A(t)} \left\rvert\mathbbm{1}_{\{\bar{Q}_i<\bar{K}_i\}} -  \mathbbm{1}_{\{Q_i<K_i\}}\right\rvert\right] + \mathbb{E}\left[C\int_{0}^{t}\Big\rvert\bar{Q}(u)-Q(u) \Big\rvert du \right].   \label{eq:RegretUpperBound}
	\end{align} 
% 	\VS{BTW, you need to argue that the regret will be non-negative - not too hard in the single optimum policy setting by taking supremum over different policy classes but it is not that easy in the non-unique thresholds case. Otherwise showing a positive bound on a negative term is not that useful.}\magenta{If the regret is negative then I think it implies that our learning algorithm performs better than the genie-aided algorithm and it is a good thing?  I thought the regret is to usually act as a measurement  how much worse the learning algorithm performs comparing to the chosen base line and an upper bounds shows that it can only perform better than the showed bound.  If we really want to avoid the case of showing a positive bound on a negative term, I wonder if we can consider the absolute value of the difference? }
	From \eqref{def:q(t)1} and \eqref{def:q(t)2}, we note that $$\left\rvert \bar{Q}(t) - Q(t) \right\rvert \leq \left\rvert \bar{Q}_n + \mathbbm{1}_{\{\bar{Q}_n < \bar{K}_n \}}  - \left( Q_n + \mathbbm{1}_{\{Q_n < K_n\}} \right) \right\rvert. $$This expression helps us to get an upper bound for the integral
	$\int_{0}^{t}\rvert \bar{Q}(u) - Q(u)\rvert du$ in \eqref{eq:RegretUpperBound} as follows: 
	%{\footnotesize
	\begin{align*}
		\int\limits_{0}^{t}\Big\rvert \bar{Q}(u) - Q(u)\Big\rvert du \leq \sum_{i = 0}^{N_A(t)} \left(T^A_{i+1}-T^A_i\right)\left( \rvert \bar{Q}_i   - Q_i \rvert +\rvert \mathbbm{1}_{\{ \bar{Q}_i < \bar{K}_i\}}- \mathbbm{1}_{\{ Q_i < K_i\}}  \rvert \right)  .
	\end{align*}
	%}%
	 Substituting the above bound in  \eqref{eq:RegretUpperBound}, we get: 
	%{\footnotesize
	\begin{align}
		G(t) &\leq \mathbb{E}\left[  R\sum\limits_{i = 1}^{N_A(t)} \Big\rvert \mathbbm{1}_{\{\bar{Q}_i<\bar{K}_i\}} -  \mathbbm{1}_{\{Q_i<K_i\}} \Big\rvert \right] %\nonumber \displaybreak[0]\\
		%&\quad
		+ \mathbb{E}\left[C \sum_{i = 0}^{N_A(t)} (T^A_{i+1}-T^A_i)\Big\rvert \bar{Q}_i   - Q_i \Big\rvert \right]\nonumber\displaybreak[0]\\
		&\quad+ \mathbb{E}\left[C \sum_{i = 0}^{N_A(t)} \left(T^A_{i+1}-T^A_i\right)\Big\rvert \mathbbm{1}_{\{ \bar{Q}_i < \bar{K}_i\}}- \mathbbm{1}_{\{ Q_i < K_i\}}  \Big\rvert  \right].\label{eq:RegretUpperBound2}
	\end{align}
	%}%
	 Note that the (future) inter-arrival time $T^A_{i+1} - T^A_i$ is independent of the queue-length of the learning and genie-aided systems $Q_i$ and $\bar{Q}_i$, respectively, as well as the threshold used at the arrival of the $i^{th}$ customer $K_i$ and $\bar{K}_i$. In particular,  $T^A_{i+1} - T^A_i$ is independent of $\rvert \bar{Q}_i - Q_i\rvert$ and $\rvert \mathbbm{1}_{\{\bar{Q}_i < \bar{K}_i\} } -  \mathbbm{1}_{\{Q_i < K_i\} }\rvert$. Then as the increments of the Poisson process are independent, we have:
 	\begin{align*}
 	   \mathbb{E}\left[C \sum_{i = 0}^{N_A(t)} (T^A_{i+1}-T^A_i) \Big\rvert \bar{Q}_i   - Q_i \Big\rvert \right] 
 	  & = \mathbb{E}\left[ C \sum_{i=0}^\infty (T^A_{i+1}-T^A_i) \Big\rvert \bar{Q}_i   - Q_i \Big\rvert \mathbbm{1}_{\{T^A_i \leq t\}} \right] \displaybreak[0]\\
 	  & = C \sum_{i=0}^\infty \mathbb{E}\left[ (T^A_{i+1}-T^A_i) \Big\rvert \bar{Q}_i   - Q_i \Big\rvert  \mathbbm{1}_{\{T^A_i \leq t\}} \right] \quad \text{(MCT)}\displaybreak[0]\\
 	  %& = C \sum_{i=0}^\infty \mathbb{E}\left[ \mathbb{E}\left[ (T^A_{i+1}-T^A_i)\Big\rvert \bar{Q}_i   - Q_i \Big\rvert  \mathbbm{1}_{\{T^A_i \leq t\}} \Big| \mathcal{F}_i\right] \right] \displaybreak[0]\\
 	  %& = C \sum_{i=0}^\infty \mathbb{E}\left[ \mathbb{E}\left[ (T^A_{i+1}-T^A_i)\Big| \mathcal{F}_i\right]\Big\rvert \bar{Q}_i   - Q_i \Big\rvert \mathbbm{1}_{\{T^A_i \leq t\}}  \right] \displaybreak[0]\\
 	  & = C \sum_{i=0}^\infty \mathbb{E}\left[ \frac{1}{\lambda} \Big\rvert \bar{Q}_i   - Q_i \Big\rvert \mathbbm{1}_{\{T^A_i \leq t\}}  \right] \quad \text{(By independence)}\displaybreak[0]\\
 	  & = \mathbb{E}\left[ \frac{C}{\lambda} \sum_{i=0}^\infty  \Big\rvert \bar{Q}_i   - Q_i \Big\rvert \mathbbm{1}_{\{T^A_i \leq t\}}  \right]  \quad \text{(MCT)}\displaybreak[0]\\
 	  & = \frac{C}{\lambda} \mathbb{E}\left[ \sum_{i=0}^{N_A(t)}   \Big\rvert \bar{Q}_i   - Q_i \Big\rvert  \right],
 	\end{align*}
 	where MCT stands for the Monotone Convergence Theorem. Similarly, we can also simplify $ \mathbb{E}\left[C \sum_{i = 0}^{N_A(t)} (T^A_{i+1}-T^A_i)\Big\rvert \mathbbm{1}_{\{ \bar{Q}_i < \bar{K}_i\}}- \mathbbm{1}_{\{ Q_i < K_i\}}  \Big\rvert  \right]$ to get
 \begin{align}
 	%\text{RHS of \eqref{eq:RegretUpperBound2} }  &= 
  G(t)&\leq\mathbb{E}\left[  \bigg(R+\frac{C}{\lambda}\bigg) \sum\limits_{i = 1}^{N_A(t)} \left\rvert \mathbbm{1}_{\{\bar{Q}_i<\bar{K}_i\}} -  \mathbbm{1}_{\{Q_i<K_i\}} \right\rvert  \right]  + \mathbb{E}\left[\frac{C}{\lambda} \sum_{i = 0}^{N_A(t)}\Big\rvert \bar{Q}_i   - Q_i \Big\rvert \right] \nonumber \\
 	& \leq \bigg(R+\frac{C}{\lambda}\bigg)\mathbb{E}\left[ \sum\limits_{i = 1}^{N_A(t)} \left\rvert \mathbbm{1}_{\{\bar{Q}_i<\bar{K}_i\}} -  \mathbbm{1}_{\{Q_i<K_i\}} \right\rvert + \Big\rvert \bar{Q}_i   - Q_i \Big\rvert\right] .\label{eq:RegretUpperBound3}
 \end{align} 
	Following this bound, from now on, we analyze the systems at the arrival epochs $\{ T_i^A\}_{i\geq 1}$. 
% 	\magenta{Done}\VS{These upper bounds are more or less re-derived in each proof in the document - in Lemma 4.1 and in Prop. 4.5. You can simplify the expressions and analysis considerably using this bound. It would be good to reflect this into the proofs, but do save the originals in a commented out segment.}
	
	With the shift to analyzing the systems at arrival epochs, we will characterize the regret in terms of the total number of arrivals $N$.  We use $\tilde{G}(N): = G(T^A_N)$ to denote the total regret accumulated up to the arrival of the $N^{th}$ customer. %We assume that the arrival sequence ends  after the arrival of the $N^{th}$ customer. 
	Recall that $m = 1/\mu$ denote the average service time and $\nu = 1/\lambda$ denote the inter-arrival time. We assume that $0<m<\infty$ and $0<\nu<\infty$: we allow for the average service time to be large, and it is possible to have $\bar{K} = 0$ where the optimal policy for the genie-aided system is to reject any arriving customer. Note that when $\bar{K} = 0$, equality in \eqref{eq:ThresholdIneq} is not possible for $R,C>0$, therefore the optimal policy is unique, and $\bar{K}_i  = \bar{K} = 0$ for all $i\geq0$. If the genie-aided dispatcher always admits customers when the queue is empty and the learning dispatcher knows this, then the algorithm design would be simpler: %{less interesting \blue{AC: consider much simpler?}}
	%there is no need to explore explicitly as always using the larger value between $1$ and the thresholds computed by solving the inequalities \eqref{eq:ThresholdIneq} using the empirical service rate estimates would give constant regret. 
	 there is no need to balance exploration and exploitation explicitly. With this knowledge, a learning dispatcher can achieve constant regret using a policy that always accepts customers when the queue is empty and uses a threshold computed by solving the inequalities \eqref{eq:ThresholdIneq} using the empirical service rate otherwise. The conflicting requirements for a learning algorithm in the two different regimes -- $\bar{K}=0$ (stop admitting customers soon) versus $\bar{K}>0$ (admit customers infinitely often but at the correct rate via the right choice of the threshold) -- are critical to the difficulty of our problem and its analysis. %The assumption of $\underbar{M}$ is needed for our analysis to resolve a technical issue: see the proofs of Propositions~\ref{prop:Phase2Reg} and \ref{prop:ICPhase2Reg} where this bound allows us to counteract events where all measured service times are extremely small. Owing to this issue, knowledge of this parameter is included in the learning algorithm too. Thus, the algorithm has knowledge of $\underbar{M}$, and therefore, it also has knowledge of a uniform upper bound on the maximum queue-length for any system from the assumed service rate range. %\VS{Later on you use $U'$ or $U$ in proofs - use consistent definitions.}\YZ{Here we keep word description only.} 
 %Knowing $\underbar{M}$, the queue-length process under our learning algorithm remains bounded, and therefore, the queue is always stable, and importantly, returns to $0$ from any state within a random time with finite expectation that is bounded uniformly over all the parameters. 
% 	\VS{Where is $\bar{M}$ used in the analysis? Does the learning algorithm need to be aware of it?} \magenta{I did not find where we used $\bar{M}$ explicitly. Yet we do need $m$ to be finite.} \VS{Should we assume a $\underbar{M}$ (lower bound on mean service time) -- only for analysis or for both analysis and learning algorithm?} \magenta{Now we assume a lower bound on the mean service time, and we would need this when designing the algorithm. }

 \begin{algorithm}[bthp!]
	\caption{Learning-based customer dispatch, with unknown service and arrival rate.}
	\label{alg:Alg3}
	\SetAlgoLined
	$i = 0$; $j = 0$; \Added{$\alpha_j $ grows at polynomial rate in $j$ }; $s = 0$; \Added{$K^*(j) = \max\{\floor{\ln(j)},0\}+l_1+Q_0.$}\\ 
	% \VS{Correct the choice of $\alpha_j$.}\YZ{edited.}\\
	\While{$i \leq N$}{
		$j = j+1$\;
        \% \textsf{If the phase $1$ of the $j^{th}$ batch happens, it sees $l_1$ customers.}\\
		\If{$j==1\text{ or }  (K(j-1) ==0 \text{ and }  B^j ==1 )$ }{
			\For{the next $l_1$ customers}{ 
				$i = i+1$\;
				\Added{
				\% \textsf{we update the belief of the average arrival time when there is a new arrival.} \\ 
				$\hat{\nu} = \hat{\nu} + \frac{ \text{inter-arrival time observed } - \hat{\nu}}{i}$\;
			}
				Exploration phase: customers always join the queue, $K_i=l_1$. \\%$K_i = \infty$. \\
				
			} 
                \If{there are $\mathrm{S_{cnt}}>0$ new services completed during this phase 1}{
					\For{$\mathrm{cnt}=1$ to $\mathrm{S_{cnt}}$}{
						$s = s +1$\;
						$\hat{m} = \hat{m} + \frac{ \text{ service time of the $\mathrm{s}^{\mathrm{th}}$ customer that completed service} - \hat{m}}{s}$\;
					}
				}
		}
  		\Added{Compute integer $K$, which satisfies $V(K,1/\hat{m}, 1/\hat{\nu}) \leq R/C<V(K+1,1/\hat{m},1/\hat{\nu})$\; 
		Set $K(j) = \min\{K^{*}(j),K\}$\;
	}
		count = 0 \;
		\% \textsf{The phase $2$ of the $j^{th}$ batch sees at least $\alpha_jl_2$ customers. The queue-length is $0$ when phase $2$ ends.}\\ 
		% 			\magenta{Done} \VS{ Correct this}\\
		\While{ $ count < \alpha_j l_2 \text{ or } Q_i > 0 $ }{
			count = count +1\;
			$i = i+1$\;
			\Added{$\hat{\nu} = \hat{\nu} + \frac{ \text{inter-arrival time observed  } - \hat{\nu}}{i}$\;
			}
			Customers join the queue if and only if the queue-length is smaller than $K(j)$, and so $K_i = K(j)$.\\
			
		}
		\If{there are $\mathrm{S_{cnt}}>0$ new services completed during this phase 2}{
			\For{$\mathrm{cnt}=1$ to $\mathrm{S_{cnt}}$}{
				$s = s +1$\;
				$\hat{m} = \hat{m} + \frac{ \text{ service time of the $\mathrm{s}^{\mathrm{th}}$ customer that completed service} - \hat{m}}{s}$\;
			}
		}
		
	}
	% \VS{Check variable $k$ used in algorithm doesn't stand for anything else.}\YZ{k was used as indexes in the proof of \ref{subsection:IntCaseAnotherGenie}.Renamed the k here to scnt.}
\end{algorithm}
	\subsection{ The learning algorithm}\label{subsection:LearningAlg}  
	We propose (and study) Algorithm~\ref{alg:Alg3} for learning-based social-welfare maximizing dispatch that consists of a sequence of \textit{batches}, where each batch has two phases: phase $1$ for exploration and phase $2$ for exploitation. %Specifically, every arriving customer will join the queue during phase $1$ (assuming that a phase $1$ is used); hence, the exploration title.  During phase $2$, the algorithm will use a threshold policy that is computed at the beginning of this exploitation phase. 
	For customer $i$ who arrives during phase 1 (assuming that a phase $1$ is used), we can assume that $K_i = \infty$ as this customer is admitted in the queue no matter the queue-length at this arrival. However, in our algorithm, we will fix any exploration phase (if used) for all batches to last for exactly $l_1$ arrivals, and so, the threshold $K_i$ is effectively $K_i=l_1$ for all arrivals in any phase $1$. %At the beginning of phase 2 of the $j^{th}$ batch, threshold $K(j)$ is computed by solving the inequalities $V(x,1/\hat{m},1/\hat{\nu}) \leq R/C<V(x+1,1/\hat{m},1/\hat{\nu})$ for integer $x$ where $\hat{m}$ is the empirical average service time \Added{and $\hat{\nu}$ is the empirical inter-arrival time; note that we will also use $K^*(j)$ to clip the computed integer where $K^*(j)$ increases (slowly) to infinity}.  
 \Added{At the beginning of phase 2 of the $j^{th}$ batch $K(j)$ is computed by finding the minimum between $K^*(j)$ and the integer that solves inequalities $V(x,1/\hat{m},1/\hat{\nu}) \leq R/C<V(x+1,1/\hat{m},1/\hat{\nu})$.} The computed $K(j)$ will be used for the entire exploitation phase of batch $j$. That is, for customers $i_1$ and $i_2$ who arrive during phase $2$ of the $j^{th}$ batch, $K_{i_1}= K_{i_2} = K(j)$ and these customers are admitted to the queue when the queue-length seen at their arrival is strictly less than $K(j)$. For technical reasons, we will insist that at the termination of phase $2$, the queue is empty. As the batch number increases, our algorithm will extend the length of the exploitation phase and reduce the occurrences of the exploration phases.

	Here is some notation that we use in the algorithm: 
 % \VS{$\alpha_j$ is not discussed anywhere but presented in the algorithm where we set $\alpha_j=j$. Is this the only choice for the results?}
	\begin{itemize}[leftmargin=*]
	    \item $l_1$: A positive integer representing the length of phase $1$, $l_1>1$;
	    \item $l_2$: A positive integer representing the initial minimum length of phase $2$, $l_2\geq l_1$;
     \Added{\item  $i$: A positive integer which is the  index of the arriving customer from the very beginning. It is used to update the belief of the average arrival rate;}
	    \item  $j$: A positive integer that indices the batch number;
        \item $\alpha_j\geq 1$: Growth factor for the length of phase $2$ in the $j^{\mathrm{th}}$ batch which ensures that the phase $2$ duration lasts for at least $\lceil\alpha_j l_2\rceil$ arrivals;
        % \VS{$\alpha_j\geq 1$: Growth factor for the length of phase $2$ in the $j^{\mathrm{th}}$ batch which ensures that the phase $2$ duration lasts for at least $\lceil\alpha_j l_2\rceil$ arrivals;}\YZ{Added}
	    \item $B^j$: A Bernoulli random variable that is independent of everything else, where $\mathbb{P}\left[B^j= 1\right] = 1 $ for $j = 1$, and \Added{$\mathbb{P}\left[B^j = 1\right] = \ln^{\epsilon}(j)/j$ for $j>1$ and fixed $\epsilon >0 $.} If the threshold used in the previous batch (the $(j-1)^{\mathrm{th}}$ batch) is $0$, the random variable $B^j$ will be used to determine if phase $1$ will happen; 
	   	\item $K(j)$: the threshold used by the learning dispatcher during phase $2$ of the $j^{th}$ batch;
	   	\Added{ \item $K^*(j)$: the upper bound of the threshold used by the learning algorithm. This parameter slowly increases to infinity, and is chosen to be larger than the initial queue-length, $Q_0$, and the length of phase $1$, i.e., $l_1$;}
     \Added{\item $S_{cnt}$: A counter which counts for the number of completed services in each phase. This counter is used to update the belief of the average service rate after each phase.  }
	    %\item $\bar{K}^*$: the uniform upper bound of the threshold that should be used. Using the knowledge of $\underbar{M}$, $\bar{K}^*$ satisfies $V(\bar{K}^*,1/\underbar{M}) \leq R/C<V(\bar{K}^*+1,1/\underbar{M})$.
	    %\comm{\item $L_2^j: $ $ L_2^j = \min\{n \left\rvert n>\alpha_jl_2, Q_{n}  = 0 \right. \} - \alpha_jl_2$ is a random variable that denotes the number of arrivals that the learning algorithm sees in the $j^{th}$ phase 2 after seeing $\alpha_jl_2$ customers. }
	    
	   % \magenta{I removed the description of $L_2$ here because I think we just want to emphasize the length of phase 2 will be at lease $\alpha_jl_2$ and when phase 2 ends the queue-length is 0 at this point. Introducing $L_2^j$ here seems unnecessary.} \VS{This $L_2^j$ seems to be incomplete in ensuring that the system is empty at the end of phase 2 particularly if the queue empties before $\alpha_j l_2$ arrivals occur and the threshold is positive as the new customers will still be in the system. Also, what is $i$ here? Use variable for arrival count at the start of phase 2 in the $j^{th}$ batch - I think you use $n^j$ in the proof of Lemma 4.1. Then, I think the correct formula is $L_2^j=\min\{n: n \geq n^j + \alpha_j l_2, Q_n=0\}-n^j$ when there is no phase 1 and $L_2^j=\min\{n: n \geq n^j + l_1+ \alpha_j l_2, Q_n=0\}-n^j-l_1$ if there is a phase 1. In the proof of Prop. 4.5 you use the correct expression for $L_2^j$ but a definition override is used!}
	\end{itemize}
 
	Note that Algorithm \ref{alg:Alg3} enforces an exploration phase only for the first batch, and then utilizes one in a probabilistic manner when the learned threshold in the previous batch is $0$. When the genie-aided system uses a non-zero threshold, as the number of services experienced by the customers admitted by the dispatcher increases, the threshold learned by the algorithm will quickly become non-zero for phase $2$. In this scenario, the exploration phase can potentially be eschewed, and, in fact, should be used more and more infrequently as time progresses so that the regret is not large. In fact, in our algorithm we completely eliminate a phase $1$ for a batch if in the previous batch the threshold of its phase $2$ is positive: some customers will be admitted in a phase $2$ with a positive threshold so new service time estimates will obtain, and on the contrary, a phase $2$ with a $0$ threshold will not admit any customers. However, allowing for an exploration phase is necessary. When the genie-aided system uses a non-zero threshold, it is possible that the learning system sees the first few service times being long enough so that the learned threshold is $0$. Then, without the exploration phase, the learning system will stop admitting any customers to the queue, and therefore, will not get any more samples to update its false belief. Although this is a low-probability event, the probability of this happening is non-negligible for any fixed length $l_1$ of the exploration. 
	
	The frequency of the exploration phase in our algorithm is controlled by the distribution of $B^j$. Our theoretical regret analysis uses $\mathbb{P}\left[B^j = 1\right] = \ln(j)/j$. %\comm{blue{AC: missing =1?}If one has a prior belief of the service rate and arrival rate, the distribution of $B^j$ can be adjusted based on the prior \blue{AC: I don't understand it}.} 
	When the genie-aided system uses the threshold $0$, the exploration phase should not happen too often. This is because every time the learning system admits a customer into the queue, the regret increases. Hence, this regime demands that phase $1$ be eschewed as fast as possible. However, as the algorithm is unaware of the parameter regime (even whether the optimal threshold is zero or non-zero), we necessarily need enough phase $1$s when the threshold from the previous batch is $0$. Hence, to combat the regret accumulation from phase $1$s when the optimal policy is not to admit any arrivals, we increase the length of phase $2$ (the exploitation phase) as the batch count increases. \Added{The control of the length of phase $2$ of the $j^{\mathrm{th}}$ batch is achieved using parameter $\alpha_j$: phase $2$ of the $j^{\mathrm{th}}$ batch will last for at least $\lceil\alpha_j l_2\rceil$ arrivals. Whereas we do require that $\alpha_j$ grows to infinity, we do not want it to grow too fast as this could lead to poor performance: when the thresholds used by the learning and genie-aided systems do not match in a batch, there may be too much regret accumulated during that batch if there is a large value of $\alpha_j$ for small $j$ (when the probability of an error is higher).}
	
	\Added{Note  that $K^*(j) = \max\{\floor{\ln(j)},0\}+l_1+Q_0$ is a deterministic function, with $K^*(j)$ no smaller than $l_1$ and the initial queue-length of the learning system $Q_0$ (when $Q_0$ is chosen in a deterministic manner). %\VS{Understanding here is that $Q_0$ is chosen in a deterministic manner.} \YZ{added.} 
    We also note that $\lim_{j\rightarrow \infty} K^*(j) = \infty$. This ensures that as the number of batches increases, eventually, the (true) optimal thresholds will be smaller than this upper bound. Note that for all $j\geq \ceil{e^{\bar{K}}}$ batches, $K^*(j)\geq\bar{K}$. Therefore, if the estimations on the service and arrival rates are accurate during batch $j$ for $j\geq \ceil{e^{\bar{K}}}$, then the learning dispatcher will be using $\bar{K}$ during phase $2$. Although $\ceil{e^{\bar{K}}}$ can be a large number, it is a fixed constant (fixing $\mu$ and $\lambda$), and the total expected regret accumulated during the first $\floor{e^{\bar{K}}}$ batches will also be a constant (see Remark \ref{rem:RegretBeforeReachBound}). Therefore, in our analysis we focus our analysis on the regret accumulated when $j\geq\ceil{e^{\bar{K}}}$.}
 % \VS{We will see in our results that there is also a trade-off in the choice of the algorithmic parameter $\alpha_j$, which controls the length of phase $2$ of the $j^{\mathrm{th}}$ batch: phase $2$ will last for at least $\lceil\alpha_j l_2\rceil$ arrivals. For parameter settings with all optimal thresholds being positive, we will need $\alpha_j$ to increase without bound fast enough for low regret, but this growth rate will impact the regret achieved for parameter settings with $0$ being an optimal threshold.} \YZ{Added.}%\VS{Here would be a good place to discuss the use of $\alpha_j$.}

	\subsection{Main results: Regret bounds for Algorithm \ref{alg:Alg3}}\label{subsection:MainResults}

	\begin{thm}\label{thm:Thm1}
		Assume that the initial queue-length for the learning and genie-aided systems are the same, and $0$ is not in the set of optimal thresholds used by the genie-aided system. %the threshold used in the genie-aided system is not $0$. 
		Then, Algorithm \ref{alg:Alg3} achieves $O(1)$ regret as $N \rightarrow \infty$, where $N$ is the total number of arrivals.
	\end{thm}
\Added{
	\begin{thm}\label{thm:Thm2}
		Assume that the initial queue-length for the learning and genie-aided systems are the same, and $0$ is in the set of optimal thresholds used by the genie-aided system. %the threshold used in the genie-aided system can be $0$. 
		Then, Algorithm \ref{alg:Alg3} achieves $O(\ln^{1+\epsilon}(N))$ regret for any specified $\epsilon>0$ as $N \rightarrow \infty$, where $N$ is the total number of arriving customers. 
		
	\end{thm}
}

	When the learning and genie-aided systems have different initial queue-lengths, as stated in Remark \ref{rem:DiffQL} below, the regret characterization still holds. This is done by introducing another genie-aided system that has the same initial queue-length as the learning system. Thereafter, we will use Proposition \ref{prop:OrderedSystems} (discussed in the following section), which shows that if two coupled systems use the same threshold policy, then the ordering of their queue-lengths is preserved. \Added{We end this section by pointing out that the regret characterization in Theorem~\ref{thm:Thm2} can be changed to $O(\log^{1+\epsilon}(N))$ for all $\epsilon>0$ as $N\rightarrow \infty$; see the discussion in Remark~\ref{rem:order}.}
	
	\section{Preliminary results}\label{section:Preliminary}
	
	We will use a few coupled systems to prove the main results. Besides the coupling between the learning  and the genie-aided systems mentioned before, we will also compare the queue-length process of the learning system with systems using the same threshold policy but with different initial queue-lengths. The following results are proved for systems coupled by having the same arrival process and with the service time of the customers in the queue of both systems begin determined by the same Poisson process from $t = 0$.

	The next proposition states that the order of the queue-lengths of two coupled systems is preserved over time if their threshold policies satisfy certain conditions. This is a core preliminary result that is used in different ways, and helps us establish our main results in considerable generality. Consider two systems $G$ and $L$ coupled through process $\{ N_A(t)\}_{t \geq 0}$ and $\{P(t)\}_{t\geq 0 }$ as described in Section~\ref{coupling}, but with possibly different initial queue-lengths and (threshold) admission policies. %\VS{Your arrival process is not $N_A(t)$ in Section 2.2.1 but rather, you set $N_A(t)$ to equal another Poisson process. Do correct this.} \YZ{ the description in previous sections has been changed. } 
 Let $Q^G(t)$ and $Q^L(t)$ denote the queue-length at time $t$ of the two systems, respectively. Let $\{K^G_i\}_{i\geq0}$ and $\{K^L_i\}_{i\geq0}$ denote the threshold policies of the two systems, respectively.

	\begin{prop}\label{prop:OrderedSystems}$\quad$
		\begin{itemize}[leftmargin=*]
			\item[1.]  If the dispatchers for the two coupled systems $G$ and $L$ use the same threshold admission policy for all arrivals, i.e., $K^G_i = K^L_i$ for all $i$, then with probability $1$, the order of their queue-lengths is preserved for all time, that is, %\comm{\VS{Extends to equality too?}}
            %, then,
			\begin{align}
				Q^G(0) \geq Q^L(0)\ \Longrightarrow\ Q^G(t) \geq Q^L(t), \quad \quad \quad \forall t \geq 0. \label{ineq1}
			\end{align}
			\item[2.] Assume that both systems have the same initial queue-length $q:=Q^G(0) = Q^L(0)$. Let $D^G(t)$ and $D^L(t)$ denote the number of departures up to time $t$ for the systems $G$ and $L$, respectively. If %systems $G$ and $L$ use threshold strategies
            $K^G_i \geq K^L_i$ for all $i$, then with probability $1$, %$Q^G(t) \geq Q^L(t)$ and $D^G(t) \geq D^L(t)$ for all $t\geq 0$. 
            \begin{align}
            Q^G(t) \geq Q^L(t) \text{ and }D^G(t) \geq D^L(t), \quad \quad \quad \forall t \geq 0.
            \end{align}
            Moreover, every customer that joins the queue in the system $L$ necessarily joins the queue in the system $G$ when static thresholds  $K^G \geq K^L$ are used in the two systems, respectively, and $q\leq K^L$.
		\end{itemize}
	\end{prop}
% 	\magenta{Done. Since the proof is by induction, and we compare the threshold ``point-wisely" only at arrivals, the proof does not need to be modified for the first part and the first half of the second part. However, I think we cannot show the second half statement on customers joining the queue in one system must also join the queue in the other system for non-static thresholds in general. I have edited the statement of the Proposition. }\VS{This result and its proof is for a fixed static threshold. How does this result apply when the thresholds are time-varying, i.e., a sequence? You'll need to define ``=" and ``$>$" and adjust the proof to account for it, and in 2 you'll need to adjust the comparison statement between $q$ and $K^L$ - is this relationship initially sufficient OR do you need some property across the sequence?} 
	Before proving the proposition, we state a useful corollary. 
	\begin{cor} \label{cor:OrderedSystems}
		 Assume that phase $1$ of the $j^{th}$ batch did not happen and the queue-length processes of the learning and genie-aided systems are coupled. If the two systems use the same threshold during the phase $2$ of the $j^{th}$ batch and if the queue-length of the genie-aided system hits $0$ during this phase $2$, then the queue-lengths of both systems are $0$ at the end of this phase $2$.
	\end{cor}
	\begin{proof}[Proof of Corollary \ref{cor:OrderedSystems}]
		Recall that under the proposed algorithm, the queue-length of the learning system is $0$ at the end of each phase $2$. Hence, the result follows immediately by Proposition \ref{prop:OrderedSystems}. 
% 		\VS{If $j$ is used to count batches, then use that everywhere.} \magenta{There were no batch number appearing in this corollary. I re-worded this corollary to describe the $j^{th}$ batch( assuming that this is what you are referring to.)}
	\end{proof}
	\begin{proof} [Proof of Proposition \ref{prop:OrderedSystems}]
		Let us start by proving the first part of Proposition \ref{prop:OrderedSystems}. %\comm{\blue{AC: why is this red?}}
		Since the queue-length process is a jump process, it is sufficient to show that after each jump, the queue-lengths of the two systems satisfy \eqref{ineq1}. Note that the set of potential jump times is the union of the arrival times (jumps times in the arrival process) and the jump times in the Poisson process that determines the service process. Let $\{t_l\}_{l\geq 0}= \{T^A_i\}_{i \geq 0} \cup \{ T^{PD}_i\}_{i \geq 0}$ denote the ordered countable set of potential jump times of the queue-length process, where $t_{l-1} < t_l$. By the superposition property of independent Poisson processes, 
% 		(the memory-less property of the exponential distribution, and the independence of the two processes, 
		with probability $1$, $\{T^A_i\}_i \cap \{ T^{PD}_i\}_i  = \emptyset$, so that at any time instant $t_l$, either there is an arrival, or there is a potential departure. Let $Q^G_l$ and $Q^L_l$ denote the queue-lengths immediately before the $l^{th}$ potential jump of the system $G$ and $L$, respectively. Also, let $Q^G_0$ and $Q^L_0$, respectively, denote the initial queue-length of the two systems.
		
		The proof follows by induction. Fix $n >0 $ and assume $Q^G_l \geq Q^L_l$ holds for all $l\leq n$. Immediately  after time $t_{n}$, one of the following can happen:
		\begin{itemize}[leftmargin=*]
			\item If $Q^G_n = Q^L_n$: In case the jump at time $t_n$ is due to a service completion or a service wasted, $Q^G_{n+1} = Q^L_{n+1}$. If the jump is due to a new arriving customer, the dispatcher will make the same choice in both systems, and $Q^G_{n+1} = Q^L_{n+1}$ holds. 
			\item If $Q^G_n > Q^L_n\geq 0$: In case the jump at time $t_n$ is due to a service completion or a service wasted, $Q^G_{n+1} \geq Q^L_{n+1}$. Otherwise, the jump is due to an arriving customer. We have $Q^G_{n+1} \geq Q^G_{n}\geq Q^L_{n}+1 \geq Q^L_{n+1}$. 
		\end{itemize}
		
		Now, let us consider the second part of Proposition \ref{prop:OrderedSystems}. First, we show that $Q^G(t)\geq Q^L(t)$ holds for all $t$. Again, it is sufficient to show $Q^G_l\geq Q^L_l$ for every $l >0$, the proof of which follows by induction. Fix $n>0$ and assume that $Q^G_l \geq Q^L_l $ for all $l\leq n$. Immediately  after $t_n$, one of the following can happen:
		\begin{itemize}[leftmargin=*]
			\item If $Q^G_n = Q^L_n$: In case the jump at time $t_n$ is due to a service completion or a service wasted, then $ Q^G_{n+1} = Q^L_{n+1}$. Otherwise, the jump is due to an arriving customer. Since $K^G_{i}\geq K^L_{i}$  for all $i$, this customer is admitted in system L only if also admitted in system G and we have $Q^G_{n+1} \geq Q^L_{n+1}$.
			\item If $Q^G_n > Q^L_n\geq 0$: As before, either both processes jump in the same direction at time $t_n$ or only one of them jumps (which would be the L system). In either case, $Q^G_{n+1} \geq Q^L_{n+1}$.
		\end{itemize}
		Since $Q^G(t)\geq Q^L(t)$ holds for all $t$ it follows that whenever there is a service completion in system L then there is one also in G. Therefore, $D^G(t) \geq D^L(t)$. 
		
		Now assume that the static thresholds $K^G$ and $K^L$ are used in the systems $G$ and $L$, respectively. To show that every customer who joins the queue in system $L$ also joins the queue in system G, we will show first that $Q^G(t)- Q^L(t) \leq K^G-K^L$. Fix a $n >0$ and assume that $Q^G_l - Q^L_l \leq K^G-K^L$ holds for all $l\leq n$. One of the following can happen immediately after time $t_n$: 
		\begin{itemize}[leftmargin=*]
			\item If $Q^G_n - Q^L_n = K^G-K^L$: Under this case, either we have \{$Q^G_n = K^G$, $Q^L_n = K^L$\}, or \{$Q^L_n \leq Q^G_n < K^G$, $Q^L_n < K^L$\}. Then, only when $Q^G_n = K^G-K^L$, $Q^L_n = 0$, and the jump is due to a service completion or service being wasted, the queue-length processes of the two systems evolve differently: system $G$ has a service completion but not $L$. However, $Q^G_{n+1} - Q^L_{n+1} \leq K^G-K^L$ still holds.
			\item If $Q^G_n - Q^L_n < K^G-K^L$: Either we have \{$ Q^L_n \leq Q^G_n < K^G$, $Q^L_n = K^L$\}, or \{$Q^L_n \leq Q^G_n < K^G$, $Q^L_n < K^L$\}. When \{$Q^L_n \leq Q^G_n < K^G$, $Q^L_n = K^L$\}, if the jump is due to an arriving customer, the dispatcher in the system $G$ will assign this customer to the queue but not the dispatcher in the system $L$. Otherwise, both systems have a service completion. Then, $Q^G_{n+1} - Q^L_{n+1} \leq K^G-K^L$  holds in either case. When \{$Q^L_n \leq Q^G_n < K^G$, $Q^L_n < K^L$\}, if the jump is due to a new arriving customer, the dispatchers in both systems admit the customers to the queue. Otherwise, the jump is due to a service completion or service being wasted, where  it is possible that only in system $G$ there is a service completion. Again, $Q^G_{n+1} - Q^L_{n+1} \leq K^G-K^L$ holds in either cases. 
		\end{itemize}
		At the time $T^A_l$, which  corresponds to the arrival of the $l^{th}$ customer, assume that this customer is admitted to the queue in the system $L$ but not in $G$. %\VS{Use the variables you've already defined, for example arrival times.}\YZ{Rephrased.} 
  We must have  $Q^L_{l} < K^L$ and $Q^G_{l} = K^G$, i.e., $Q^G_{l} - Q^L_{l} > K^G-K^L$. This is a contradiction. Therefore, for any arriving customer, either the dispatchers in both systems $G$ and $L$ make the same admission decision, or only the dispatcher in the system $G$ admits this customer. As a result, any customer who joins the queue in the system $L$ necessarily joins the queue in the system $G$.
	\end{proof}

	\begin{remark}\label{rem:DiffQL}
		In case the genie-aided system and the learning system have different initial queue-lengths, we can introduce a second genie-aided system that has the same initial queue-length as the learning system and is also coupled with the two systems using the procedure from Section~\ref{coupling}. Let $Q'_i$ denote the queue-length of this new system right before the $i^{th}$ arrival customer, $G'(N)$ denote the regret of the learning algorithm with respect to the second genie-aided system.
		Using the triangle inequality and equation \eqref{eq:RegretUpperBound3}, we get:
		\begin{align*}
			\tilde{G}(N)& \leq \bigg(R+\frac{C}{\lambda}\bigg)\mathbb{E}\left[ \sum\limits_{i = 1}^{N} \left\rvert \mathbbm{1}_{\{\bar{Q}_i<\bar{K}_i\}} -  \mathbbm{1}_{\{Q'_i<\bar{K}_i\}} \right\rvert + \Big\rvert \bar{Q}_i   - Q'_i \Big\rvert\right]  + G'(N).
		\end{align*}
		Theorems \ref{thm:Thm1} and \ref{thm:Thm2}  provide regret bounds for $G'(N)$. By Proposition \ref{prop:OrderedSystems},  the orders of $Q'_i$ and $\bar{Q}_i$ are preserved, thus after both queue-length processes hit 0, $Q'_i$ and $\bar{Q}_i$ will evolve together. Since the expected time of both queue-length processes to hit $0$ simultaneously is finite, the regret characterization in Theorems \ref{thm:Thm1} and \ref{thm:Thm2} still holds. 
	\end{remark}
	
	\section{Unique admittance threshold case }\label{section:UniqueThresh}
	
	In this section, we analyze the case where \eqref{eq:ThresholdIneq} holds with strict inequality. In this case, the genie-aided dispatcher uses a unique optimal threshold $\bar{K}$, and the resulting queue-length process has a stationary distribution. 
	
	In section \ref{subsection:SASampleEst}, we start by providing an estimate for the number of samples \Added{of completed service times} that the learning algorithm uses in order to \Added{estimate the average service time, and then to} update the threshold policy for each phase $2$: see Proposition \ref{prop:SampleEst}. We use it to estimate the probability that the learning system \Added{can obtain an accurate estimate of the average service time: see Proposition \ref{prop:AccuServ}. Combining the above estimate with the probability that the learning system can obtain an accurate estimation on the arrival rate, see Proposition \ref{prop:AccuArri}, we can bound the probability of the learning system using the same threshold as the genie-aided system; see Corollary \ref{cor:SameThreshProb}. } In section \ref{subsection:PhaseReg}, we estimate the regret of the learning algorithm because of  having  phase $1$ (if used) and using incorrect thresholds in phase $2$ separately. Proposition \ref{prop:BadEventProb} we consider ``bad" events where there will be regret accumulated during phase $2$ \Added{because of using the wrong threshold}. In addition, we will use an upper bound on the difference between the queue-length processes of the learning and genie-aided system to bound the regret accumulated \Added{because of the existence of  phase $1$ (if used) in Lemma \ref{lem:Phase1Reg} and because of using the wrong threshold during phase $2$ in Lemma \ref{lem:Phase2Reg}. The proof of Theorem \ref{thm:Thm1} and \ref{thm:Thm2} are stated in  section \ref{subsection:pfthm1} and \ref{subsection:pfthm2} respectively.}

    \Added{\subsection{Sample estimation}\label{subsection:SASampleEst} First, we state and prove some results on the number of  samples the learning dispatcher gets on the inter-arrival times and completed service times, and the resulting implications on the estimates of the arrival and service rates.}
	
	In the following proposition, we show that with high probability, the number of  samples \Added{of completed service times} that the learning algorithm can observe is sufficiently large at the beginning of the phase $2$ of the $j^{th}$ batch. For this, we use the fact that (by design) each phase $2$ is longer than phase $1$.
 
	\begin{prop}\label{prop:SampleEst}
		Let $D_j$ denote the number of observed service times up to the beginning of phase $2$ of the $j^{th}$ batch. Then,
		%{\footnotesize
		\Added{\begin{align*}
				\mathbb{P}\left[ D_j \leq \frac{l_1\ln^{1+\epsilon}(j)\mu}{4(1+\epsilon)(\lambda+\mu)}\right] &\leq  \exp\left( -\frac{l_1\ln^{1+\epsilon}(j)\mu}{16(1+\epsilon)(\lambda+\mu)}\right) + \exp\left( -\frac{C_0(\epsilon)}{8}-\frac{\ln^{1+\epsilon}(j)}{8(1+\epsilon)}\right),
		\end{align*} 
		%}%
  where $C_0(\epsilon) := 1 + \sum_{i = 2}^{\floor{e^{\epsilon}}} \frac{\ln^{\epsilon}(i)}{i} - \frac{\ln^{1+\epsilon}(\ceil{e^{\epsilon}})}{1+\epsilon}$ is a constant depending on the choice of $\epsilon$.}
	\end{prop}
	\begin{proof}%\comm{[Outline of proof of Proposition \ref{prop:SampleEst}]}
		Consider the epoch which is the beginning of phase $2$ of the $j^{th}$ batch. Let $\hat{X}^j$ denote the total number of arrivals that the learning dispatcher sees during the past batches and the potential phase 1 of the $j^{th}$ batch%\comm{, where consider phase $2$s of past batches for which the admission policy does not reject every arrival}
        . Note that $\hat{X}^j$ counts for the arrivals in phase 1's (when they occur), and all past phase 2's using a threshold $\geq 1$.
		
		The following inequality holds when $\alpha_jl_2\geq l_1$ for all $j$:
    %{\small
	\begin{align*}
		\hat{X}^j& \geq l_1 +\sum_{i = 1}^{j-1} \left(\mathbbm{1}_{\{K(i)>0\}} \alpha_i l_2 + \mathbbm{1}_{\{K(i)= 0\}} B^{i+1} l_1 \right)\geq l_1\sum_{i = 1}^{j} B^i .
	\end{align*} 
    %}%
	\Added{Observing the function $\ln^{\epsilon}(x)/x$ is decreasing when $x\geq e^{\epsilon}$, when $j\geq \ceil{e^{\epsilon}}$, we have:
 
 \begin{align*}
    &\frac{\ln^{1+\epsilon}(j)}{1+ \epsilon} - \frac{\ln^{1+\epsilon}(\ceil{e^{\epsilon}})}{1+\epsilon} = \int_{\ceil{e^{\epsilon}}}^{j} \frac{\ln^{\epsilon}(x)}{x}dx \leq \sum_{i = \ceil{e^{\epsilon}}}^{j}\frac{\ln^{\epsilon}(i)}{i} ,\\
    &\sum_{i = \ceil{e^{\epsilon}}}^{j}\frac{\ln^{\epsilon}(i)}{i}  \leq \frac{\ln^{\epsilon}(\ceil{e^{\epsilon}})}{\ceil{e^{\epsilon}}} + \int_{e^{\epsilon}}^{j} \frac{\ln^{\epsilon}(x)}{x}dx  = \frac{\ln^{\epsilon}(\ceil{e^{\epsilon}})}{\ceil{e^{\epsilon}}} + \frac{\ln^{1+\epsilon}(j)}{1+\epsilon} - \frac{\ln^{1+\epsilon}(e^{\epsilon})}{1+\epsilon}.
 \end{align*}

 %Combining with $1/2 <1 + \ln(2)/2 - \ln^2(3)/2$, and $1/2 + \ln^2(2)/2<1 + \ln(2)/2 $ we get: 
 %$$ \frac{1}{2}+\frac{\ln^2(j)}{2}\leq  \mathbb{E}\left[ \sum_{i = 1}^{j}\hat{B}^i \right] \leq 3+ \frac{\ln^2(j)}{2}.$$
 Set  
 $$C_0(\epsilon) := 1 + \sum_{i = 2}^{\floor{e^{\epsilon}}} \frac{\ln^{\epsilon}(i)}{i} - \frac{\ln^{1+\epsilon}(\ceil{e^{\epsilon}})}{1+\epsilon}\qquad \text{and}\qquad \tilde{C}_0(\epsilon) := 1 + \sum_{i = 2}^{\ceil{e^{\epsilon}}} \frac{\ln^{\epsilon}(i)}{i} - \frac{\ln^{1+\epsilon}(e^{\epsilon})}{1+\epsilon},$$ 
 we get: 
 $$ C_0(\epsilon)+\frac{\ln^{1+\epsilon}(j)}{1+\epsilon}\leq  \mathbb{E}\left[ \sum_{i = 1}^{j} B^i \right] \leq \tilde{C}_0(\epsilon) + \frac{\ln^{1 + \epsilon}(j)}{1 + \epsilon}.$$
 Using the multiplicative Chernoff bound for independent Bernoulli random variables,  the inequalities above, and $\tilde{C}_0(\epsilon) \geq 0 $ for all $\epsilon >0 $, %$$\frac{1}{2}+\frac{\ln^2(j)}{2}\leq  \mathbb{E}\left[ \sum_{i = 1}^{j}\hat{B}^i \right] \leq 3+ \frac{\ln^2(j)}{2},$$ 
 we get the following upper bound on the probability of $\hat{X}^j$ being small: 
	%{\tiny
	\begin{align*}
		\mathbb{P}\left[\hat{X}^j <\frac{ l_1\ln^{1+\epsilon}(j)}{2(1+\epsilon)} \right]&\leq \mathbb{P}\left[ l_1\sum_{i = 1}^{j} B^j<\frac{ l_1\ln^{1+\epsilon}(j)}{2(1+\epsilon)}\right]  \leq \exp\left( -\frac{C_0(\epsilon)}{8}-\frac{\ln^{1+ \epsilon}(j)}{8(1+\epsilon)}\right).
	\end{align*}
 }
 %\VS{Can you push this result further to have $\mathbb{E}[\hat{B}^j]=\tfrac{\ln^\epsilon(j)}{j}$ for a fixed $\epsilon>0$ instead of $\epsilon=1$ that's been used thus far? Perhaps a good choice here can push the regret when $0$ is an optimal threshold to $\ln^{1+\delta}(N)$ for some $\delta(\epsilon)>0$ that goes to $0$ as $\epsilon$ goes to $0$, whilst preserving constant regret otherwise. Seems like you need $x> \exp(\epsilon)$ for $\tfrac{\ln^\epsilon(x)}{x}$ to be a decreasing function of $x$ which will work well for the above proof argument. Maybe even something like $\mathbb{E}[\hat{B}^j]=\tfrac{\ln(\ln(j))}{j}$ for $j$ large enough will give something extremely close to $\ln(N)$ regret.} \YZ{edited. Now having have $\mathbb{E}[\hat{B}^j]=\tfrac{\ln^\epsilon(j)}{j}$ for a fixed $\epsilon \in (0,1]$. The choice $\mathbb{E}[\hat{B}^j]=\tfrac{\ln(\ln(j))}{j}$ for $j$ may work as well, but I am afraid that numerical experiments in this case is going to be hard to perform. } \VS{The case of $\mathbb{E}[\hat{B}^j]=\tfrac{\ln(\ln(j))}{j}$ wouldn't be used for numerical evaluations but to argue that the regret scaling will be better than $\ln^{1+\epsilon}(N)$ for every $\epsilon>0$ without having to choose $\epsilon$. This would come closer to answering the question that we posed for how low the regret can go if we have constant regret in positive thresholds case.}
	%}%
 
    \Added{Recall that $i$ is the index of the customers arriving from the very beginning.} Let $\zeta_i$ be a Bernoulli random variable such that $\zeta_i = 1$ when there is at least one potential service completion between the arrival time of the $i^{th}$ and $(i+1)^{th}$ customer.  The random variables $\{\zeta_i\}_{i}$  are \emph{i.i.d}. and $\mathbb{P}[\zeta_i = 1] = \mu/(\lambda+\mu)$. When the threshold used is at least 1, if the $i^{th}$ customer is rejected, the queue-length at the arrival of this customer is non-zero%\comm{, and there is at least one completed service during the inter-arrival time between the $i^{th}$ and $(i+1)^{th}$ customers}
    ; obviously, when the $i^{th}$ customer is admitted to the queue, the queue-length right after the arrival of this customer is non-zero. %\comm{, there is at least one completed service during the inter-arrival time between the $i^{th}$ and $(i+1)^{th}$ customers}.
    In either case, if there are any potential services during the inter-arrival times between the $i^{th}$ and $(i+1)^{th}$ customers, at least one of the completed services is observed by the learning dispatcher. \Added{This implies that   $\sum_{i \text{ counted in } \hat{X}_j } \zeta_i 
 = \sum_{n = 0}^{\hat{X}_j} \zeta_{cnt_n} \leq D_j$, where $cnt_n$ is a sub-sequence of $i$ and $cnt_n$ is the index from the beginning of the $n^{th}$ arrival customer that is counted in $\hat{X}_j$.  }  Then we have: 
    % \VS{Ceiling needed below so check bound.}\YZ{adding this ceiling seems to be fine since $\hat{X}_j$ only takes integer values? fixed a typo here.}
	%{\scriptsize
 \Added{
	\begin{align*}
		%&
  \mathbb{P}\left[ \left. D_j \leq \frac{l_1\ln^{1+\epsilon}(j)\mu}{4(1+\epsilon)(\lambda+\mu)}\right\rvert \hat{X}^j \geq \frac{l_1\ln^{1+ \epsilon}(j)}{2(1+\epsilon)} \right] %\\%\text{\VS{Fix $\hat{X}^j$} \YZ{fixed.}}\\
  &\leq%& 
  \mathbb{P}\left[ \sum_{n = 1}^{\lceil l_1\ln^{1+\epsilon}(j)/2(1+\epsilon) \rceil}\zeta_{cnt_n} \leq \frac{l_1\ln^{1+\epsilon}(j)\mu}{4(1+\epsilon)(\lambda+\mu)}\right]\\ 
  &\leq%& 
  \exp\left( -\frac{l_1\ln^{1+\epsilon}(j)\mu}{16(1+\epsilon)(\lambda+\mu)}\right).
	\end{align*}
    %}%
    }
	We dropped the conditioning in the first inequality using $\sum_{i = 1}^{\hat{X}^j} \zeta_i \leq D_j$, and $\mathbb{P}[\sum_{i = 1}^{n+1} \zeta_i \leq c] \leq \mathbb{P}[\sum_{i = 1}^{n} \zeta_i\leq c]$  for all $n, c \in \mathbb{Z}^+$, and the second inequality follows from multiplicative Chernoff bound for independent Bernoulli random variables. Combining the results above, we obtain:
 \Added{
 \small{
	\begin{align*}
		\mathbb{P}\left[ D_j \leq \frac{l_1\ln^{1+\epsilon}(j)\mu}{4(1+\epsilon)(\lambda+\mu)}\right]%\\
		& = \mathbb{P} \left[ \left. D_n \leq \frac{l_1\ln^{1+\epsilon}(j)\mu}{4(1+\epsilon)(\lambda+\mu)}\right\rvert \hat{X}^j \geq \frac{l_1\ln^{1+\epsilon}(j)}{2(1+\epsilon)}\right] \mathbb{P}\left[\hat{X}^j \geq \frac{l_1\ln^{1+\epsilon}(j)}{2(1+\epsilon)}\right]\\
		&\quad\quad + \mathbb{P}\left[\left. D_j \leq \frac{l_1\ln^{1+\epsilon}(j)\mu}{4(1+\epsilon)(\lambda+\mu)}\right\rvert \hat{X}^j < \frac{l_1\ln^{1+\epsilon}(j)}{2(1+\epsilon)}\right] \mathbb{P}\left[\hat{X}^j< \frac{l_1\ln^{1+\epsilon}(j)}{2(1+\epsilon)}\right]\\
		&\leq  \exp\left( -\frac{l_1\ln^{1+\epsilon}(j)\mu}{16(1+\epsilon)(\lambda+\mu)}\right) + \exp\left( -\frac{C_0(\epsilon)}{8}-\frac{\ln^{1+\epsilon}(j)}{8(1+\epsilon)}\right).
	\end{align*}
}
 }
This completes the proof.
 	\end{proof}

 Using Proposition~\ref{prop:SampleEst} above, in the next proposition we will establish that with high probability, \Added{the learning dispatcher will have an accurate estimate of the average service time, and therefore the service rate}. %Note that the same threshold is used in both systems when the estimated mean service time is close to the true mean service time: when $\bar{K} > 0$, the estimated mean service time needs to be bounded from above and below for the two dispatchers to use the same threshold policy, but when $\bar{K} = 0$ it only needs to be bounded from below. This leads to different constants in the proposition below. 
	
	\begin{prop}\label{prop:AccuServ}
		\Added{Let $\hat{m}(j)$ denote  the empirical service time estimated by the learning dispatcher at the beginning of phase 2 of the $j^{th}$ batch.  For the proposed algorithm,
		\begin{align}\label{eq:AccuServ}
				\mathbb{P}\left[ \left\rvert \hat{m}(j) - m\right\rvert> \Delta_1 \right]\leq C_1\exp(-C_2\ln^{1+\epsilon}(j)),
		\end{align}
	where 
		\begin{align}\label{def:c1c2}
			\begin{split}
			C_1 &:= \max\left\{ \exp\left(-\frac{C_0(\epsilon)}{8}\right),\; \frac{2\exp{\left( \Delta_1^2/(8m^2)\right)}}{\exp{( \Delta_1^2/(8m^2))}-1},\; 1\right\},\\
			C_2 &:= \min\left\{ \frac{l_1\mu}{16(1+\epsilon)(\lambda+\mu)}, \;\frac{1}{8(1+\epsilon)}, \;\frac{l_1\mu\Delta_1^2}{32(1+\epsilon)m(\lambda m+1)}\right\},
		\end{split}
        \end{align}
    with $\Delta_1 := \min\{ \delta_1, 2m\}$, and $\delta_1$ is the constant from inequality \eqref{eq:RegionOfParam} which is one part of the condition needed for the conclusion in \eqref{eq:CorrectThreshold}. } 
  % \VS{Check - is satisfied we have 
  %   $$  	V(\bar{K}, 1/\hat{m},1/\hat{\nu}) <  \frac{R}{C} < V(\bar{K}+1,  1/\hat{m},1/\hat{\nu}).$$}
		% }
  % \VS{Now $\delta_1$ is only for service time error, so the equation above may not hold.}\YZ{edited.}
	\end{prop}
	
			The proof of the proposition relies upon tail concentration bounds for sub-exponential random variables. We follow the definition and concentration bounds as in \cite[Section~2.1]{wainwright_2019}. 
			\begin{defn}\label{def:SubExp}
				A random variable $X$ with mean $\mu$ is called sub-exponential if there are non-negative parameters ($\alpha^2, \beta$) such that $\mathbb{E}[e^{\gamma(X-\mu)}]\leq e^{\frac{\alpha^2\gamma^2}{2}}$ for all $\rvert \gamma\rvert < \frac{1}{\beta}$.
			\end{defn} 
			\begin{prop}\label{prop:SubExpTailBound}
				Suppose that $X$ is sub-exponential with parameters ($\alpha^2,\beta$). Then:
				$$\mathbb{P}[X\geq \mu +t] \leq\begin{cases} 
					e^{-\frac{t^2}{2\alpha^2}}, &0\leq t\leq \frac{\alpha^2}{\beta}, \\
					e^{-\frac{t}{2\beta}}, & t\geq \frac{\alpha^2}{\beta},
				\end{cases} = \max{\left\{ e^{-\frac{t^2}{2\alpha^2}} , e^{-\frac{t}{2\beta}} \right\}}.
				$$
			\end{prop}
		\begin{proof}[Proof of Proposition \ref{prop:AccuServ}]
		 Let $S_i$ denote the service time of the $i^{th}$ service completion. Since $S_i$ are \emph{i.i.d.} with distribution EXP$(1/m)$, which is a $(4m^2, 2m)$ sub-exponential random variable, $\sum_{i = 1}^{n} S_i$ is a $(4m^2n, 2m)$ sub-exponential random variable; see \cite[Section~2.8]{HDP_RV}. Observe that $0\leq k\Delta_1 \leq 2mk$. Using the sub-exponential concentration bounds above, we get:
	%	{\footnotesize
		\begin{align*}
			\mathbb{P}\left[ \left\rvert \hat{m}(j) - m\right\rvert> \Delta_1  \rvert D_j > n\right] &\leq \sum\limits_{k = n+1}^{\infty}\mathbb{P}\left[ \left\lvert \sum\limits_{i = 1}^{k}S_i - km \right\rvert \geq k\Delta_1 \right] \\
			&\leq  \sum\limits_{k = n+1}^{\infty}2 \exp{\left( -\frac{k\Delta_1^2}{8m^2}\right)}%\\
			%& 
            \leq \frac{2\exp{\left( \Delta_1^2/(8m^2)\right)}}{\exp{( \Delta_1^2/(8m^2))}-1}\exp{\left( -\frac{(n+1)\Delta_1^2}{8m^2}\right)}.
		\end{align*}
		%}%
		The third inequality follows by the  geometric sum formula. 
  
 \Added{
 Then, substituting $n  = \floor{l_1\ln^{1 + \epsilon}(j)\mu/(4(1+\epsilon)(\lambda+\mu))}$, we get:
		%{\scriptsize
  {\small
		\begin{align*}
			\mathbb{P}\left[ \left\rvert \hat{m}(j) - m\right\rvert> \Delta_1   \left\rvert D_j >\frac{l_1\ln^{1 + \epsilon}(j)\mu}{4(1 + \epsilon)(\lambda+\mu)} \right.\right]
			& =\mathbb{P}\left[ \left\rvert \hat{m}(j) - m\right\rvert> \Delta_1  \left\rvert D_j > \floor*{ \frac{l_1\ln^{1 + \epsilon}(j)\mu}{4(1 + \epsilon)(\lambda+\mu)} }\right.\right]\\
			& \leq \frac{2\exp{\left( \Delta_1^2/(8m^2)\right)}}{\exp{( \Delta_1^2/(8m^2))}-1}\exp{\left( -\left(\floor*{\frac{l_1\ln^{1 + \epsilon}(j)\mu}{4(1 + \epsilon)(\lambda+\mu)} } +1\right)\frac{\Delta_1^2}{8m^2}\right)}\\
			&\leq \frac{2\exp{\left( \Delta_1^2/(8m^2)\right)}}{\exp{( \Delta_1^2/(8m^2))}-1}\exp{\left( -\frac{l_1\mu\ln^{1 + \epsilon}(j)\Delta_1^2}{32(1 + \epsilon)m^2(\lambda+\mu)}\right)}.
			%\comm{&\leq \frac{2\exp{\left( \Delta^2/(8\bar{M}^2)\right)}}{\exp{( \Delta^2/(8\bar{M}^2))}-1}\exp{\left( -\frac{(l_1-1)\Delta^2}{16(\lambda \bar{M}+1)\bar{M}^2}j\right)}.}
		\end{align*}
  }
		%}%
		Using the last upper bound and Proposition \ref{prop:SampleEst}, we find:
		%{\small
		\begin{align*}
			\mathbb{P}\left[ \left\rvert \hat{m}(j) - m\right\rvert> \Delta_1  \right] &=\mathbb{P}\left[ \left\rvert \hat{m}(j) - m\right\rvert> \Delta_1  \left\rvert D_j\leq\frac{l_1\ln^{1+\epsilon}(j)\mu}{4(1+\epsilon)(\lambda+\mu)} \right.\right]\mathbb{P}\left[D_j\leq\frac{l_1\ln^{1+\epsilon}(j)\mu}{4(1+\epsilon)(\lambda+\mu)}\right]  \\
			&\qquad+\mathbb{P}\left[ \left\rvert \hat{m}(j) = m\right\rvert> \Delta_1  \left\rvert D_j>\frac{l_1\ln^{1+\epsilon}(j)\mu}{4(1+\epsilon)(\lambda+\mu)} \right.\right]\mathbb{P}\left[D_j>\frac{l_1\ln^{1+\epsilon}(j)\mu}{4(1+\epsilon)(\lambda+\mu)}\right]  \\
			&\leq \exp\left( -\frac{l_1\ln^{1+\epsilon}(j)\mu}{16(1+\epsilon)(\lambda+\mu)}\right) + \exp\left( -\frac{C_0(\epsilon)}{8}-\frac{\ln^{1+\epsilon}(j)}{8(1+\epsilon)}\right) \\
			&\quad\quad+ \frac{2\exp{\left( \Delta_1^2/(8m^2)\right)}}{\exp{( \Delta_1^2/(8m^2))}-1}\exp{\left( -\frac{l_1\mu\ln^{1+\epsilon}(j)\Delta_1^2}{32(1+\epsilon)m^2(\lambda+\mu)}\right)}\\
			&\leq C_1\exp(-C_2\ln^{1+\epsilon}(j)),
		\end{align*}
		%}%
		where $C_1$ and $C_2$ are given by \eqref{def:c1c2}.
%ASAF: I fixed here a square to 1+eps
  }
	\end{proof}
\Added{
\begin{prop}\label{prop:AccuArri}
	Let $\nu(j)$ denote the empirical inter-arrival time estimated by the learning dispatcher at the beginning of phase 2 of the $j^{th}$ batch. For the proposed algorithm, 
	\begin{align*}
		\mathbb{P}\left[ \rvert \nu - \hat{\nu}(j)\rvert >\Delta_2 \right]\leq C_3\exp(-C_4\beta_j)  ,
	\end{align*}
	where
	\begin{align}\label{def:c3c4}
	%	\begin{split}
			C_3 %&
   :=  \frac{2\exp(\Delta_2^2/(8\nu^2))}{\exp(\Delta_1^2/(8\nu^2))-1},\quad %\\
			C_4 %&
   :=  \frac{l_1\Delta_2^2}{8\nu^2},\quad\text{ and } \quad%\\
			\beta_j%&
   :=1 + \sum_{i = 1}^{j-1}\alpha_i,
	%	\end{split}
	\end{align}
	with $\Delta_2:= \min\{ \delta_2, 2\nu\}$,
	and $\delta_2$ is the constant from inequality \eqref{eq:RegionOfParam} which is the second part of the condition needed for the conclusion in \eqref{eq:CorrectThreshold}. 
 % \VS{Check - is satisfied, we have $$V(\bar{K}, 1/\hat{m},1/\hat{\nu}) < R/C<V(\bar{K}+1, 1/\hat{m},1/\hat{\nu}). $$
 % }
 % \VS{Now $\delta_2$ is only for arrival time error, so the equation above may not hold.}\YZ{edited.}
\end{prop}
}
\Added{
\begin{proof}
	Note that no matter whether customers are admitted to the queue or not, the learning dispatcher is able to observe all arrivals. We always have the first phase $1$, and that the number of customers who arrived during the $j^{th}$ phase 2 is at least $\alpha_jl_2$. Note that we also have $l_2>l_1$. Let $\beta_j = 1+\sum_{i = 1}^{j-1}\alpha_i$. Right before the $j^{th}$ phase 2, there are at least $l_1 + \sum_{n = 1}^{j-1} \alpha_nl_2  \geq \beta_jl_1$ customers that have arrived at the system, and the learning dispatcher would have observed all the inter-arrival times. Following a similar logic as in the proof of Proposition \ref{prop:AccuServ},  let $A_i$ denote the inter-arrival time of consecutive customers. $A_i$  are \emph{i.i.d.}  with distribution $\mathrm{EXP}(1/\nu)$, which is a $(4\nu^2,2\nu)$ sub-exponential random variable. Using the concentration result detailed in Proposition \ref{prop:SubExpTailBound} for sub-exponential random variables, we have:
	\begin{align*}
		\mathbb{P}\left[\rvert \nu -  \hat{\nu}(j)\rvert >\Delta_2\right] &\leq \sum_{k =\beta_j }^{\infty}\mathbb{P}\left[\left\rvert \sum_{i = 1}^{k} A_i - k\nu\right\rvert >k\Delta_2\right]\\
		&\leq \sum_{k= \beta_j}^{\infty}  2\exp\left(-\frac{k\Delta_2^2}{8\nu^2} \right)%\\
		%&
     \leq\frac{2\exp(\Delta_2^2/(8\nu^2))}{\exp(\Delta_1^2/(8\nu^2))-1}\exp\left( - \frac{\beta_jl_1\Delta_2^2}{8\nu^2}\right),
		%&\leq  \frac{2\exp(\Delta_2^2/(8\nu^2))}{\exp(\Delta_1^2/(8\nu^2))-1}\exp\left( - \frac{l_1\Delta_2^2}{32\nu^2}j^2\right).
	\end{align*}
%The last inequality follows from $j^2/2 -j+2>0$ for all $j>1$.
which establishes the result.
\end{proof}
}
\Added{
Note that since $\alpha_j \geq1$ for all $j$, $\beta_j \geq j$. Therefore, as the number of batches, $j$, increases, the probability of not having a correct estimate of the average arrival rate decreases faster than the probability of not having a correct estimate of the average service time. In the following corollary, we will combine Propositions \ref{prop:AccuServ} and \ref{prop:AccuArri} to get a bound on the probability of the learning dispatcher not using (an optimal) threshold $\bar{K}$ when $j$ is large.
}
\Added{
	\begin{cor}\label{cor:SameThreshProb}
		For the proposed algorithm, when $j\geq \ceil{e^{\bar{K}}}$,
		\begin{align}
			\mathbb{P}\left[ K(j)\neq \bar{K} \right]&\leq C_1\exp(-C_2\ln^{1+\epsilon}(j)) +C_3\exp(-C_4\beta_j)  ,
		\end{align}
		where $C_1$ and $C_2$ are defined in \eqref{def:c1c2}; $C_3$ and $C_4$  are defined in \eqref{def:c3c4}.
		% 		\end{align*}
\end{cor}
\begin{proof}
Recall that for  the true arrival and service rates $\lambda$ and $\mu$,  we have $$\hat{V}(\bar{K}, \mu,\lambda)< \frac{R}{C}<\hat{V}(\bar{K}+1, \mu,\lambda).$$ Proposition \ref{prop:FuncVProp} says that if $\hat{m}$ and $\hat{\nu}$ satisfy inequality \eqref{eq:RegionOfParam}, then the learning dispatcher would be able to solve for the desired threshold $\bar{K}$. Moreover, since $j>e^{\bar{K}}$, $K^*(j)\geq\bar{K}$, i.e., the learning dispatcher would be able to use $\bar{K}$ in the $j^{th}$ phase 2.  Using Proposition \ref{prop:AccuServ} and Proposition \ref{prop:AccuArri}, we have:
\begin{align*}
	\mathbb{P}\left[K(j) \neq \bar{K} \right]  &\leq \mathbb{P}\left[ \rvert m - \hat{m}(j)\rvert > \Delta_1 \right]  + \mathbb{P}\left[ \rvert \nu - \hat{\nu}(j)\rvert > \Delta_2 \right]   \\
	&\leq C_1\exp(-C_2\ln^{1+\epsilon}(j)) +C_3\exp(-C_4\beta_j),
\end{align*}
which concludes the proof.
\end{proof}
}
\Added{
When the learning dispatcher has knowledge of either $\mu$ or $\lambda$, one can obtain an inequality similar to that in Corollary \ref{cor:SameThreshProb} by setting the corresponding bound from Propositions \ref{prop:AccuServ} and \ref{prop:AccuArri} to $0$. When the service rate is known and the arrival rate is not known, then a better characterization of the regret obtains; see Remark~\ref{rem:arrival_rate}.} 
% \VS{While this part is true, it is likely that the regret scaling will be different when $\mu$ is known and $\lambda$ is unknown - constant regret in all regimes due to the free samples nature of things for the arrival rate. Do ensure that this comment filters through, at least as a remark after the main results.}\YZ{Added a remark after the proof of theorem \ref{thm:Thm2} saying an $O(1)$ regret bound can be obtained when the only arrival rate is unknown  and   phase 1 is always omitted. Also explained why phase 1 becomes unnecessary when only the arrival rate is unknown. }

	% \AC{This paragraph floats in the air here. It takes about 3 pages until we see $G^j_i$, no? Do we need it here or in the sequel?}\YZ{The following paragraph is moved here. It was before Proposition\ref{prop:SampleEst}}\AC{Vijay, please check if you like the new location of this paragraph. I do. Maybe adding something like: ``We now perform the regret analysis separately for the two phases."}\YZ{Added.}
 \Added{
 \subsection{Regret accumulated in each phase}\label{subsection:PhaseReg} We now analyze the regret. Let $G_1^j$ denote the expected regret accumulated during the period starting with the (potential) phase $1$ and ending at the first time the queue is emptied in the immediate phase $2$ for the $j^{th}$ batch that follows.  Let $G_2^j$ denote the expected regret accumulated in the remainder of phase 2 of the $j^{th}$ batch. Whenever phase $1$ of  the $j^{th}$ batch does not happen, there is no regret to be grouped to $G_1^j$, and  the regret accumulated in phase $2$ is entirely in $G_2^j$; in this case, the regret accumulated during the entire $j^{th}$ batch is also solely in $G_2^j$.  Both $G_1^j$ and $G_2^j$ count for the regret accumulated because of not having accurate estimates of the service rate as well as not estimating the arrival rate accurately. Intuitively, $G_1^j$ takes into consideration the regret accumulated because of the existence of a phase $1$, and $G_2^j$ considers the regret accumulated because of the learning system using an incorrect  threshold. Despite the subtleties, for easier recall, we refer to $G_i^j$ as the regret accumulated in phase $i\in\{1,2\}$ of batch $j$. 
 
    Let $N$ denote the number of arrivals as a function of which we will determine the regret. Then, we have: 
	\begin{align}\label{eq:RegretInJ}
		\tilde{G}(N)  \leq \mathbb{E}\left[\sum_{j = 1}^{J} (G^j_1+ G^j_2)\right]\leq \sum_{j = 1}^{\left\lceil N/l_2\right\rceil} (G^j_1+ G^j_2),
	\end{align}
% \VS{Change $\sqrt{2N/l_2}$ in such summations to $\left\lceil\sqrt{2N/l_2}\right\rceil$. }\YZ{edited.}
where $J := J(N)$ is the total number of batches until $N$ arrivals including the batch in progress or initiated by the $N^{\mathrm{th}}$ arrival. The last inequality follows by the observation:
$$N \geq \sum_{i = 1}^{J}  \alpha_i l_2  \geq \beta_J  l_2 \geq Jl_2,$$
which implies $J \leq N/l_2 $ a.s. When one uses $\alpha_j$ that grows like $j^\alpha$, for some $\alpha>0$, we obtain that $J$ is of order of $O(N^{1/(a+1)})$. This adjustment would not affect the order of the regret but only the constants: see Sections \ref{subsection:pfthm1} and \ref{subsection:pfthm2}.
}
    
    For each $j$, we will analyze $G^j_1$ and $G^j_2$ separately. %~\footnote{We believe that a different argument that combines phases could yield better regret guarantees.}. 
    Let $\mathcal{E}^j_1$ denote the event that phase 1 of the $j^{th}$ batch happens. %Note that   
    Since in the proposed algorithm, we always have the first phase $1$, we have $\mathbb{P}[\mathcal{E}^1_1] = 1$. Phase $1$ is omitted when the threshold used in the previous phase 2 is non-zero. By the independence of $B^j$ and $K(j)$, for $j>1$ we have:
	\begin{align}\label{eq:Phase1Prob}
		\mathbb{P}\left[\mathcal{E}^j_1\right] &= \mathbb{P}\left[\left.\mathcal{E}^j_1\right\rvert K(j-1) = 0\right]\mathbb{P}\left[K(j-1) = 0\right]%\nonumber \\
		 %&\qquad
   +  \mathbb{P}\left[\left.\mathcal{E}^j_1\right\rvert K(j-1) \neq 0 \right]\mathbb{P}\left[K(j-1) \neq 0\right] \nonumber\\
		& = \mathbb{P}\left[ B^j = 1 \right]\mathbb{P}\left[K(j-1) = 0\right].
	\end{align}	
 %\VS{What does $\hat{B}^j$ stand for?}\YZ{it should be $B^j$, it has been fixed.} 
 Let $\mathcal{E}^j_2$ denote the event that $K(j) = \bar{K}$, and $\mathcal{E}^j_3$ denote the event that the queue-lengths of the two systems are the same at the beginning of the $j^{th}$ batch, i.e.,
\begin{align*}
	\mathcal{E}^j_2 &:=\{ K(j) = \bar{K}\} \qquad\text{ and }\qquad \mathcal{E}^j_3 :=\{ Q_{n^j} = \bar{Q}_{n^j}\}. 
\end{align*}
\Added{
	Also, denote by $\tau^{K,l}$ the number of arrivals during a busy period of an $M/M/1/K$ queue with initial queue-length $l$.
	The proof of  lemmas  \ref{lem:Phase1Reg} and \ref{lem:Phase2Reg}  rely on an upper bound of $\mathbb{E}\left[ \tau ^{K,l}\right]$ which is stated in the following proposition.
	\begin{prop}\label{prop:BusyPeriodArrival}
		Consider an $M/M/1/K$ queue with arrival rate $\lambda$, service rate $\mu$ and intial queue length $0<l\leq K$.
		\begin{align}
			\mathbb{E}\left[ \tau^{K,l}\right]\leq g(l;K),
		\end{align}
		where 
            \begin{align*}
			g(1;K) = \begin{cases} 
				\frac{\lambda/\mu +1}{\lambda/\mu-1}\left( \left(\frac{\lambda}{\mu} \right)^{K}-1\right),& \lambda \neq \mu, \\
				2K, & \lambda = \mu,
			\end{cases}
		\end{align*}
            and for all $1<l\leq K$, 
            \begin{align*}
        g(l;K) = \begin{cases} 
				\frac{\lambda/\mu +1}{(\lambda/\mu-1)^2}\left(\left(1-\left(\frac{\lambda}{\mu}\right)^l\right)\left(\left(\frac{\lambda}{\mu}\right)^{K+1} - \frac{\lambda}{\mu} +1\right) + (l-1) \left(1-\frac{\lambda}{\mu}\right)\right),& \lambda \neq \mu, \\
				l(2K-l+1), & \lambda = \mu.
			\end{cases}
        \end{align*}
        In particular, $\mathbb{E}\left[ \tau^{K,l}\right]$ is of order $O((\lambda/\mu)^K + K^2)$.
  
	\end{prop}  
	\begin{proof}
		Consider a finite state Markov chain  with state space $\{0, 1,...,. K\}$, and with the following transition matrix: 
		\begin{align*}
			&p(0,0) = 1;\\
			&p(l,l+1) = \frac{\lambda}{\lambda+\mu},\quad	p(l,l-1)  = \frac{\mu}{\lambda+\mu}, \quad\quad \text{when }l  \in \{ 1,...,K-1\};\\
			&p(K,K) =\frac{\lambda}{\lambda+\mu}, \quad	p(K,K-1)=\frac{\mu}{\lambda+\mu}; %\VS{p(K,K-1)=\frac{\mu}{\lambda+\mu}}\YZ{edited.};
		\end{align*}
		Let $g(l;K)$ denote the expected number of jumps of this Markov chain until it hits 0 for the first time when the initial state is $l$ and the threshold is $K$. Conditional on the first jump, we obtain the following relationship for $g(l:K)$, 
		\begin{align*}
			&g(l;K) = \frac{\lambda}{\lambda + \mu}g(l+1;K) + \frac{\mu}{\lambda + \mu} g(l-1;K) +1, \quad\quad \text{ when } l \in \{1,..., K-1\};\\
			&g(K;K) = \frac{\lambda}{\lambda + \mu}g(K;K) + \frac{\mu}{\lambda + \mu} g(K-1;K) +1;
		\end{align*}
		together with the condition $g(0;K) = 0$, we can solve for $g(l;K)$, and obtain: 
		\begin{align*}
			g(1;K) = \begin{cases} 
				\frac{\lambda/\mu +1}{\lambda/\mu-1}\left( \left(\frac{\lambda}{\mu} \right)^{K}-1\right),& \lambda \neq \mu, \\
				2K, & \lambda = \mu,
			\end{cases}
		\end{align*}
            and for all $1<l\leq K$, 
            \begin{align*}
        g(l;K) = \begin{cases} 
				\frac{\lambda/\mu +1}{(\lambda/\mu-1)^2}\left(\left(1-\left(\frac{\lambda}{\mu}\right)^l\right)\left(\left(\frac{\lambda}{\mu}\right)^{K+1} - \frac{\lambda}{\mu} +1\right) + (l-1) \left(1-\frac{\lambda}{\mu}\right)\right),& \lambda \neq \mu, \\
				l(2K-l+1), & \lambda = \mu.
			\end{cases}
        \end{align*}
		From the transition probabilities of the Markov chain, $g(n:K)$ is also the expected number of services and arrivals of the corresponding $M/M/1/K$ queue with arrival rate $\lambda >0$, service rate $\mu>0$ and initial queue length $l$ during the busy period which is initiated with $n$ customers in the queue. Since each arrival must also be served when the Markov chain hits 0, $\mathbb{E}\left[\tau^{K,l}\right]\leq g(l;K)\leq 2 \mathbb{E}\left[\tau^{K,l}\right] + K$. Therefore, $g(l;K)$ serves as an upper bound on $\mathbb{E}\left[\tau^{K,l}\right]$.  This upper-bound is tight in the sense that $g(l;K)$ is at most $2 \mathbb{E}\left[\tau^{K,l}\right] + K$.
\end{proof}
}

	\Added{
		\begin{lemma}\label{lem:Phase1Reg}
			For $j>e^{\bar{K}}$, we have the following:
			\begin{enumerate}[leftmargin=*]
				\item
				When $\bar{K}>0$, 
				\begin{align*}
					G^j_1 \leq \left(R +  \frac{C}{\lambda}\right)\bigg(l_1^2 + \left(\bar{K}+1\right) l_1 +\big(1+ K^*(j)\big)g(l_1;K^*(j))\bigg)\mathbb{P}\left[\mathcal{E}^j_1\right];
				\end{align*}
				\item 
				When $\bar{K} = 0$, 
				\begin{align*}
					G^j_1 \leq \left(R +  \frac{C}{\lambda}\right)\bigg(l_1^2 +  l_1  +C_5\bigg) \mathbb{P}\left[\mathcal{E}^j_1\right]+\left(R +  \frac{C}{\lambda}\right)  \big(    1+ K^*(j)\big)g(l_1;K^*(j))\mathbb{P}\left[\left(\mathcal{E}^j_2\right)^c\right];
				\end{align*}
			\end{enumerate}
			where  
			\begin{align*}
			     C_5 := \left(1+ l_1 \right) \frac{l_1\lambda}{\mu}.
			\end{align*}
            The function $g(l;K)$ is defined in Proposition \ref{prop:BusyPeriodArrival}, and is $O((\lambda/\mu)^K + K^2)$ for all $l\leq K$.
		\end{lemma}
            
	}
 \Added{
	\begin{proof}
		Let $n^j$ denote the total number of customers that arrived until the beginning of the $j^{th}$ batch, and  %$L_1^j := \min\{n\mid Q_{n^j +l_1+n} = 0\}-n^j-l_1$. 
        $L_1^j := \min\{n\mid Q_{n^j +l_1+n} = 0\}$. 
        Recall that $\mathcal{E}^j_1$ denotes the event that phase 1 happens during the $j^{th}$ batch. Using \eqref{eq:RegretUpperBound3} and observing that regret accumulates in $G_1^j$ only when $\mathcal{E}^j_1$ happens, we have:%\VS{The equation below is incorrect. Inside the expectation you need two indicator functions event that previous threshold for phase 2 is 0 and Bernoulli for phase 1 being 1. Note that $L_1^j$ is defined otherwise too but could be part or whole of the $G_2$ term.}\YZ{$\mathcal{E}_1$ is the  event that phase 1 of the $j^{th}$ batch happens, which have already include the two events you mentioned. also see \eqref{eq:Phase1Prob} two pages before.}
		\begin{align*}
			G_1^j %&= \left(R + \frac{C}{\lambda}\right)\mathbb{E}\left[ \sum\limits_{i =n^j+1}^{n^j+l_1+ L_1^j} \left\rvert \mathbbm{1}_{\left\{\bar{Q}_i<\bar{K}\right\}} -  \mathbbm{1}_{\left\{Q_i<K(j)\right\}} \right\rvert +\left\rvert \bar{Q}_i - Q_i \right\rvert \right]\\
			&\leq \left( R + \frac{C}{\lambda}\right) \mathbb{E}\left[\left. \sum_{i = n^j+1}^{n^j + l_1}    \left\rvert \mathbbm{1}_{\{\bar{Q}_i<\bar{K}_i\}} -  \mathbbm{1}_{\{Q_i<K_i\}} \right\rvert + \Big\rvert \bar{Q}_i   - Q_i \Big\rvert \right\rvert  \mathcal{E}^j_1 \right]\mathbb{P}\left[\mathcal{E}^j_1\right] \\
			&\qquad +  \left( R + \frac{C}{\lambda}\right) \mathbb{E}\left[\left. \sum_{i = n^j+ l_1+1}^{n^j + l_1+ L_1^j}    \left\rvert \mathbbm{1}_{\{\bar{Q}_i<\bar{K}_i\}} -  \mathbbm{1}_{\{Q_i<K_i\}} \right\rvert + \Big\rvert \bar{Q}_i   - Q_i \Big\rvert \right\rvert  \mathcal{E}^j_1 \right]\mathbb{P}\left[\mathcal{E}^j_1\right]\\
			& =: (\rom{1}) + (\rom{2}).
		\end{align*} 
	Note that $\rom{1}$ is a bound on the regret accumulated during phase $1$ of the $j^{th}$ batch (when it occurs), and $\rom{2}$ is a bound on the regret accumulated in phase $2$ of the $j^{th}$ batch until the queue is emptied for the first time in this phase $2$. When $\bar{K}>0$ or $\bar{K} = 0$, we can follow the same logic to bound  $\rom{1}$, i.e.  the regret accumulated during phase $1$ for $j>e^{\bar{K}}$:
	\begin{align*}
		(\rom{1})&\leq  \left( R + \frac{C}{\lambda}\right) \mathbb{E}\left[ \sum_{i = n^j}^{n^j + l_1}  \Big( 1 + \bar{K} + l_1\Big) \right] \mathbb{P}\left[\mathcal{E}^j_1\right] %\\
		%&
  \leq \left(R +  \frac{C}{\lambda}\right)\Big(l_1^2 + \left(\bar{K}+1\right) l_1 \Big) \mathbb{P}\left[\mathcal{E}^j_1\right].
	\end{align*}
Now, we bound $\rom{2}$ in the case $\bar{K}>0$. We use $K^*(j)$ to obtain a bound on the queue-length difference of the two systems as well as the expectation of $L_1^j$. The queue-length of the learning system at the beginning of each phase $2$ is at most $l_1$ since the queue-length of the learning system is $0$ at the end of the previous phase $2$. Moreover, the threshold used by the learning dispatcher in the $j^{th}$ batch is bounded above by $K^*(j)\geq l_1$. Hence the queue-length of the learning system is bounded by $K^*(j)$ during phase $2$. Consider a system $S_2$ that uses the admission policy with threshold $K^*(j)$ and which is coupled with the learning system according to Section \ref{coupling}. Assume that the initial queue-length of $S_2$ is the same as the queue-length of the learning system at the beginning of the $j^{th}$ phase $2$ which is at most $l_1$. %\VS{(which is at the most $l_1$)}.\YZ{edited.} 
Note that the threshold used in the learning system is less or equal to the one used in $S_2$. Let $\tau$ denote the total number of arrivals during the first busy period of the system $S_2$. Using Proposition \ref{prop:OrderedSystems}, we get $Q_i\leq Q^{S_2}_i$ for $n^j + l_1 +1 \leq i \leq n^j + l_1  + L_1^j $, and  $\mathbb{E}[ L_1^j \rvert  \mathcal{E}^j_1 ]\leq \mathbb{E}\left[ \tau^{K^*(j),l_1} \right]$. %\VS{Specify the correct value of $K$ here.}\YZ{Edited.}
	Using Proposition \ref{prop:BusyPeriodArrival},
	and together with the upper bound $K^*(j)$ of the queue-length of the learning system, we get :
	{%\smaller
		\begin{align*}
			%&\left( R + \frac{C}{\lambda}\right) \mathbb{E}\left[\left. \sum_{i = n^j+ l_1+1}^{n^j + l_1+ L_1^j}    \left\rvert \mathbbm{1}_{\{\bar{Q}_i<\bar{K}_i\}} -  \mathbbm{1}_{\{Q_i<K_i\}} \right\rvert + \Big\rvert \bar{Q}_i   - Q_i \Big\rvert \right\rvert  \mathcal{E}^j_1 \right]\mathbb{P}\left[\mathcal{E}^j_1\right]\\
			(\rom{2}) &\leq \left(R + \frac{C}{\lambda}\right)\left(1+ K^*(j) \right) \mathbb{E}\left[L_1^j\rvert  \mathcal{E}^j_1 \right]\mathbb{P}\left[\mathcal{E}^j_1\right]%\\
			%& 
   \leq \left(R + \frac{C}{\lambda}\right)\left(1+ K^*(j) \right) g(l_1;K^*(j))\mathbb{P}\left[\mathcal{E}^j_1\right],%\\
			%&\leq \left(R + \frac{C}{\lambda}\right)F_1(j)\mathbb{P}\left[\mathcal{E}^j_1\right].
		\end{align*}
	}
	where $F_1(j)$ is defined in the statement of Lemma \ref{lem:Phase1Reg}. %\VS{You're using $S_2$ as a bound so you could tighten the bound on the conditional expectation of $L_1$ to $g(l_1)$ if you're able to solve for it in Prop. 4.5, since $S_2$ has maximum queue-length being $l_1$. However, it will likely have an exponential term in $\lambda/\mu$ term, but it is worth adding.}\YZ{techinically we are using $S_1$ which uses threshold $K^*(j)$ and have initial queue length $K^*(j)$ as the bound. From the expression of $g(1)$ and $g(n)$ would have the exponential term of $K^*(j)$ as well. I do not think using  $g(l_1)$ would improve the order of the regret bound and its form is quite complicated.}
	
	Together, we have the following bound for $G_1^j$ when $\bar{K}>0$:
	\begin{align*}
		G_1^j\leq \left(R +  \frac{C}{\lambda}\right)\Big(l_1^2 + \left(\bar{K}+1\right) l_1 + \big( 1+ K^*(j)\big)g(l_1;K^*(j))\Big) \mathbb{P}\left[\mathcal{E}^j_1\right].
	\end{align*}

In the case of $\bar{K}=0$, we take a slightly different path of analyzing $\rom{2}$: we consider the threshold used in the $j^{th}$ phase 2 to get a better regret bound compared to using the same argument as in the case $\bar{K}>0$. We have: %\VS{Again we need the correct indicators included regarding the past threshold ised.}\YZ{already in $\mathcal{E}_1^j$.}
		\begin{align*}
			(\rom{2}) &= \left( R + \frac{C}{\lambda}\right) \mathbb{E}\left[\left. \sum_{i = n^j+ l_1+1}^{n^j + l_1+ L_1^j}    \left\rvert \mathbbm{1}_{\{\bar{Q}_i<\bar{K}_i\}} -  \mathbbm{1}_{\{Q_i<K_i\}} \right\rvert + \Big\rvert \bar{Q}_i   - Q_i \Big\rvert \right\rvert  \mathcal{E}^j_1 \cap \mathcal{E}_2^j \right]\mathbb{P}\left[\mathcal{E}^j_1\cap \mathcal{E}_2^j\right]  \\
			&\quad\quad + \left( R + \frac{C}{\lambda}\right) \mathbb{E}\left[\left. \sum_{i = n^j+ l_1+1}^{n^j + l_1+ L_1^j}    \left\rvert \mathbbm{1}_{\{\bar{Q}_i<\bar{K}_i\}} -  \mathbbm{1}_{\{Q_i<K_i\}} \right\rvert + \Big\rvert \bar{Q}_i   - Q_i \Big\rvert \right\rvert  \mathcal{E}^j_1 \cap (\mathcal{E}_2^j)^c \right]\mathbb{P}\left[\mathcal{E}^j_1\cap (\mathcal{E}_2^j)^c\right]\\
			&\leq \left(R+\frac{C}{\lambda}\right)\left(1+ l_1 \right) \mathbb{E}\left[L_1^j\rvert  \mathcal{E}^j_1\cap \mathcal{E}_2^j \right]\mathbb{P}\left[\mathcal{E}^j_1\cap \mathcal{E}_2^j\right] \\
			&\quad\quad+ \left(R+\frac{C}{\lambda}\right)\left(1+ K^*(j) \right)\mathbb{E}\left[L_1^j\rvert  \mathcal{E}^j_1\cap (\mathcal{E}_2^j)^c \right]\mathbb{P}\left[ \mathcal{E}^j_1\cap (\mathcal{E}_2^j)^c\right]\\
			&\leq \left(R+\frac{C}{\lambda}\right)\left(1+ l_1 \right)\frac{l_1\lambda}{\mu}\mathbb{P}\left[\mathcal{E}_1^j\right] + \left(R+\frac{C}{\lambda}\right)\left(1+ K^*(j) \right)g(l_1;K^*(j))\mathbb{P}\left[(\mathcal{E}_2^j)^c\right]. %\\
			%&\leq  \left(R+\frac{C}{\lambda}\right)C_5\mathbb{P}\left[\mathcal{E}_1^j\right] + \left(R+\frac{C}{\lambda}\right)F_1(j)\mathbb{P}\left[(\mathcal{E}_2^j)^c\right] .
		\end{align*}
		The first follows since the total number of customers admitted in phase 1 is $l_1$, and since in the case $\bar{K}= 0$ and under $\mathcal{E}^j_2$, the threshold used in phase 2 is 0. Under $\mathcal{E}^j_1\cap \mathcal{E}_2^j$, the learning system does not accept any new customers to the queue, and $ \mathbb{E}[\mathcal{E}^j_1\cap \mathcal{E}_2^j]$ is the number of arrivals during the period of serving all the remaining customers in the queue. Observe the queue-length of the learning system at the beginning of phase 2 is at most $l_1$, conditioning on the time used to serve $l_1$ customers, we get the desired bound on $ \mathbb{E}[\mathcal{E}^j_1\cap \mathcal{E}_2^j]$. The bound on $\mathbb{E}[L_1^j\rvert  \mathcal{E}^j_1\cap (\mathcal{E}_2^j)^c ]$ follows the same logic as the bound of $\mathbb{E}[L_1^j\rvert  \mathcal{E}^j_1 ]$. %\VS{Arguments need some tightening and also some explanation. In the first case, Prop 4.5 is likely to give a bad upper bound since there are no admitted arrivals after the first $l_1$ since the correct threshold of $0$ is being used. Each inter-arrival samples a geometric number of potential services, and you need the first time their sum is at least $l_1$, which is a simple renewal theory calculation. For the second event, you need first assert a monotonicity in the threshold value of the upper bound from Prop. 4.5 (easy) and then state that the learning system uses a max threshold of $K^*(j)$} \YZ{I agree that when phase 1 happened yet the threshold used in the consecutive phase 2 is 0 the bound provided by Propsition\ref{prop:BusyPeriodArrival} is a band bound. Yet it would still be a constant and does not effect the order of the regret. For the second case, the monotinicy argument is provided in the comparision of system $S_1$ and $S_2$ in the previous page.} 
        Combined with the bound for $\rom{1}$, we get the desired result.
	\end{proof}
}
	
\Added{
  We observe that under the event $ \mathcal{E}^j_2\cap \mathcal{E}^j_3$,  there will be no regret accumulated in $G^j_2$: indeed,  under the  event $(\mathcal{E}^j_1 )^c\cap \mathcal{E}^j_2\cap \mathcal{E}^j_3$, the dispatcher of  the learning system  and the dispatcher of the genie-aided system will make the same decision on every arrival customer in phase 2 of the $j^{th}$ batch. As a result, their queue-lengths will be matched and there will be no regret accumulated during this exploitation phase, thus also no regret accumulated in $G_2^j$. The threshold used in phase $1$ can be considered as the maximum allowed value, namely $K^*(j) (\geq l_1)$, since all the arriving customers during phase $1$ are admitted.  Under the event  $\mathcal{E}^j_2 $, the threshold used in the $j^{th}$ phase $2$ is the same as the genie-aided system. Therefore, under the event $\mathcal{E}^j_1\cap \mathcal{E}^j_2 \cap \mathcal{E}^j_3$, although phase $1$ of the $j^{th}$ batch happens, the queue-length at the beginning of the $j^{th}$ batch is the same for both systems and the thresholds used in the learning system is no smaller than the threshold used in the genie-aided system. The coupling between the learning and genie-aided system preserves the order between the queue-lengths of the two systems as proved in Proposition \ref{prop:OrderedSystems}: when the queue-length of the learning system hits $0$ the first time after phase $1$, the queue-length of the genie-aided system is also $0$. Therefore, under event $\mathcal{E}^j_1\cap \mathcal{E}^j_2 \cap \mathcal{E}^j_3$, after the queue-length of the learning system hits $0$ after phase $1$, the queue-lengths of the learning and genie-aided system are matched, and no regret is accumulated in $G_2^j$. 
}

\Added{
	 The next proposition shows that the probability of the event $ \mathcal{E}^j_2\cap \mathcal{E}^j_3$  is high. We use De Morgan's law to  get an upper bound on the probability of this event by using already characterized bounds on the probabilities of a few events. 
	}
\Added{
	 \begin{prop}\label{prop:BadEventProb}
	 	Fix $j\geq \ceil{e^{\bar{K}}}$. Then, we have the following:
	 		\begin{enumerate}[leftmargin=*]
	 		\item In the case $\bar{K}>0$,
	 		%{\small
	 			\begin{align*}
	 				\mathbb{P}\left[ \left( \mathcal{E}^j_2 \cap \mathcal{E}^j_3\right)^c\right]
	 				&\leq C_1\exp\Big(-C_2\ln^{1+\epsilon}(j)\Big)  +C_1\exp\Big(-C_2\ln^{1+\epsilon}(j-1)\Big)\\
	 				&\quad\quad+ C_3\exp\Big(-C_4\beta_j\Big)+C_3\exp\Big(-C_4\beta_{j-1}\Big)
	 				+ \big(c_{\bar{K}}\big)^{\alpha_{j-1}l_2}.
	 			\end{align*}
	 			
	 			\item In the case $\bar{K} = 0$, 
	 			%{\small
	 				\begin{align*}
	 					\mathbb{P}\left[ \left( \mathcal{E}^j_2 \cap \mathcal{E}^j_3\right)^c\right] \leq  C_1\exp(-C_2\ln^{1+\epsilon}(j))+C_3\exp(-C_4\beta_j).
	 				\end{align*}
	 			\end{enumerate}		
	 			The constants $C_1$, $C_2$, $C_3$ and $C_4$ are defined in \eqref{def:c1c2} and \eqref{def:c3c4}, and $$c_{\bar{K}} :=1- \left(\frac{\mu}{\lambda+\mu}\right)^{\bar{K}}\in(0,1).$$
	 	
	 \end{prop}
 }
\Added{
	 \begin{proof}
	 	We first consider the case $\bar{K}>0$. Let $\mathcal{E}^j_4$ denote the event that the queue-length of the genie-aided system hits $0$ during phase $2$ of the $j^{th}$ batch. The probability that at least $\bar{K}$ potential services occur between two consecutive inter-arrivals is $1-c_{\bar{K}}$. %\VS{I think you need at least $\bar{K}$ potential services so the $\lambda/(\lambda+\mu)$ term isn't there in $c_{\bar{K}}$.} \YZ{edited.} 
        Since the genie-aided system is an $M/M/1/\bar{K}$ queue, there are at most $\bar{K}$ customers in the queue. Since the total number of arrivals during the phase $2$ of the $j^{th}$ batch  is at least $\alpha_jl_2$, we get: 
	 	$$\mathbb{P}[(\mathcal{E}^j_4)^c]\leq (c_{\bar{K}} )^{\alpha_jl_2}.$$
	 	By Corollary \ref{cor:OrderedSystems},  we have: 
	 	%{\small
	 		\begin{align*}
	 			\mathbb{P}\left[ \left(\mathcal{E}^j_3\right)^c \;\left\rvert\; \mathcal{E}^{j-1}_2 \right.\right]
	 			& \leq \mathbb{P} \left[\left(\mathcal{E}^{j-1}_4\right)^c \;\left\rvert\; \mathcal{E}^{j-1}_2\right.\right] \leq (c_{\bar{K}} )^{\alpha_{j-1}l_2}.
	 		\end{align*}
	 		%}%
	 	Using De Morgan's laws we can re-write the event $( \mathcal{E}^j_2 \cap \mathcal{E}^j_3)^c$ as $( \mathcal{E}^j_2)^c \cup (\mathcal{E}^j_3)^c$ and by using Corollary \ref{cor:SameThreshProb} for  $j>e^{\bar{K}}$ we obtain: 
	 	%{\small
	 		\begin{align*}
	 			\mathbb{P}\left[ \left(  \mathcal{E}^j_2 \cap \mathcal{E}^j_3\right)^c\right] %\\
	 			&\leq \mathbb{P}\left[ \left(\mathcal{E}^j_2\right)^c\right] +\mathbb{P}\left[ \left( \mathcal{E}^{j-1}_2\right)^c\right]+ \mathbb{P}\left[ \left.\left( \mathcal{E}^j_3\right)^c \;\right\rvert\; \mathcal{E}^{j-1}_2 \right]\\
	 			&\leq C_1\exp\Big(-C_2\ln^{1+\epsilon}(j)\Big)  +C_1\exp\Big(-C_2\ln^{1+\epsilon}(j-1)\Big)\\
	 			&\quad\quad+ C_3\exp\Big(-C_4\beta_j\Big)+C_3\exp\Big(-C_4\beta_{j-1}\Big)
	 			+ (c_{\bar{K}})^{\alpha_{j-1}l_2}.
	 		\end{align*}
	 		%}%
	 	In case that $\bar{K} = 0$, the queue-length of the genie-aided system is always 0, and $\mathcal{E}^j_3$ happens with probability 1. Hence,
	 	%{\scriptsize
	 		\begin{align*}
	 			\mathbb{P}\left[ \left( \mathcal{E}^j_2 \cap \mathcal{E}^j_3\right)^c\right]  =  \mathbb{P}\left[ (\mathcal{E}^j_2)^c\right] \leq  C_1\exp(-C_2\ln^{1+\epsilon}(j))+C_3\exp(-C_4\beta_j).
	 		\end{align*}
	 		%}
	 	This completes the proof.
	 \end{proof}
 }
\Added{
 	Next, we will estimate $G^j_2$, which considers the regret accumulated during the $j^{th}$ batch after the first time the queue-length of the learning system hit $0$ during the $j^{th}$ phase $2$  if there is a phase $1$, and considers the regret accumulated during phase $2$ if phase 1 did not happen. %\VS{(if there is a phase $1$)}.\YZ{edited.}  
    As we mentioned before, only under the event $\left( \mathcal{E}^j_2 \cap \mathcal{E}^j_3\right)^c$, regret is accumulated to $G_2^j$.
}
\Added{
 \begin{lemma}\label{lem:Phase2Reg}
 	For $j>e^{\bar{K}}$ ,
 	%	{\small
 		\begin{align*}
 			G^j_2 &\leq \left(R + \frac{C}{\lambda}\right)\bigg(\left(1+ K^*(j)\right)\alpha_jl_2 + \big(    1+ K^*(j)\big)g(K^*(j);K^*(j))\bigg)\mathbb{P}\left[ \left( \mathcal{E}^j_2 \cap \mathcal{E}^j_3\right)^c\right],
 		\end{align*}
   with $g(l;K) $ defined in Proposition \ref{prop:BusyPeriodArrival}.
 		%}%
 	%with $F_1(j)$ defined in Lemma \ref{lem:Phase1Reg}, and
 		%\begin{align*}
 			%F_2(j) :=  \left(1+ K^*(j)\right).
 		%\end{align*}
 \end{lemma}
}
\Added{
 \begin{proof}\comm{[Proof of Proposition \ref{prop:Phase2Reg}]}
 	Let $\tilde{n}^j$ denote the total number of customers that arrived until the beginning of phase 2 of the $j^{th}$ batch. Note that when phase 1 did not happen in the $j^{th}$ batch, $\tilde{n}^j = n^j$, and when phase 1 happened, $\tilde{n}^j = n^j + l_1$. However, since we are analyzing the regret accumulated in phase 2 because of using an incorrect threshold and not conditional on having a phase 1 or no, using $\tilde{n}^j$ would give simpler expressions during the analysis. By its definition, $G_2^j$ takes into consideration only part of the regret that is accumulated in phase $2$. Since we are interested in finding an upper bound, we will ``double-count'' parts of the regret that are already considered in $G_1^j$ in the case that there is a phase $1$ and compute the regret accumulated during phase 2.   
 	Set %$L_2^j := \min\{n\mid Q_{n^j +l_1+\alpha_jl_2+n} = 0\}-n^j-l_1-\alpha_jl_2$. 
  $L_2^j := \min\{n\mid Q_{\tilde{n}^j +\alpha_jl_2+n} = 0\}$.
  This is the total number of arriving customers beyond the first $\alpha_jl_2$ ones during the exploitation phase for the $j^{th}$ batch. 
 	Using \eqref{eq:RegretUpperBound3} and $(\mathcal{E}^j_2 \cap \mathcal{E}^j_3)^c$, we get:
 		\begin{align*}
 			G^j_2
 			&\leq \left(R + \frac{C}{\lambda}\right)\mathbb{E}\left[ \sum\limits_{i =\tilde{n}^j+1}^{\tilde{n}^j+\alpha_{j}l_2 + L_2^j} \left\rvert \mathbbm{1}_{\left\{\bar{Q}_i<\bar{K}\right\}} -  \mathbbm{1}_{\left\{Q_i<K(j)\right\}} \right\rvert\mathbbm{1}_{\left\{\left( \mathcal{E}^j_2 \cap \mathcal{E}^j_3\right)^c\right\}} \right]\\
 			& \quad + \left(R + \frac{C}{\lambda}\right) \mathbb{E}\left[ \sum_{i = \tilde{n}^j}^{\tilde{n}^j+\alpha_j l_2 + L_2^j} \left\rvert \bar{Q}_i - Q_i \right\rvert \mathbbm{1}_{ \left\{\left( \mathcal{E}^j_2 \cap \mathcal{E}^j_3\right)^c\right\}}\right] \\
 			&=:\left(R + \frac{C}{\lambda}\right)\Big((\rom{3})+(\rom{4})\Big).
 		\end{align*}
 		%}%
 	In what follows we bound the two expectations on the RHS. For the first expectation, since $\rvert \mathbbm{1}_{\{\bar{Q}_i<\bar{K}\}} -  \mathbbm{1}_{\{Q_i<K(j)\}}\rvert \leq 1$, after splitting phase $2$ into two parts, we get: 
 	%{\footnotesize
 		\begin{align*}
 			(\rom{3}) & \leq  \mathbb{E}\left[  \sum\limits_{i =\tilde{n}^j+1}^{\tilde{n}^j+\alpha_jl_2 }\mathbbm{1}_{ \left\{\left(\mathcal{E}^j_2 \cap \mathcal{E}^j_3\right)^c\right\}} \right] +  \mathbb{E}\left[  \sum\limits_{i =\tilde{n}^j+\alpha_jl_2+1}^{\tilde{n}^j+\alpha_jl_2 + L_2^j} \mathbbm{1}_{ \left\{\left( \mathcal{E}^j_2 \cap \mathcal{E}^j_3\right)^c\right\}} \right]\\
 			&= \mathbb{E}\left[  \alpha_jl_2\mathbbm{1}_{ \left\{\left(\mathcal{E}^j_2 \cap \mathcal{E}^j_3\right)^c\right\}} \right] +  \mathbb{E}\left[  L_2^j \mathbbm{1}_{ \left\{\left( \mathcal{E}^j_2 \cap \mathcal{E}^j_3\right)^c\right\}} \right]\\
 			& = \alpha_jl_2\mathbb{P}\left[ \left( \mathcal{E}^j_2 \cap \mathcal{E}^j_3\right)^c\right] +  \mathbb{E}\left[  L_2^j \left\rvert  \left(\mathcal{E}^j_2 \cap \mathcal{E}^j_3\right)^c\right. \right]\mathbb{P}\left[ \left( \mathcal{E}^j_2 \cap \mathcal{E}^j_3\right)^c\right].
 		\end{align*}
 		Using a similar way of analyzing $L_1^j$ in the proof of Lemma \ref{lem:Phase1Reg} but comparing with a coupled system that uses threshold $K^*(j)$ and having initial queue-length $K^*(j)$,  we get: %\VS{Again, just assert monotonicity of the upper bound in the threshold being used so that max possible threshold can be used)}\YZ{technically this monotonicity property is in the bound for $L_1^j$.}
 		\begin{align*}
 			 \mathbb{E}\left[ L_2^j \left\rvert \left( \mathcal{E}^j_2 \cap \mathcal{E}^j_3\right)^c\right.\right]\leq \mathbb{E}\left[\tau^{K^*(j),K^*(j)}\right] \leq g(K^*(j); K^*(j)).
 		\end{align*}
 		Together with the inequalities above, we get a bound for (\rom{3}): 
 		\begin{align*}
 			(\rom{3}) &\leq \left( \alpha_j l_2 + g(K^*(j);K^*(j))\right)\mathbb{P}\left[ \left( \mathcal{E}^j_2 \cap \mathcal{E}^j_3\right)^c\right].
 		\end{align*}
 		%}%
 	We can split (\rom{4}) in a similar manner as above, and then, together with $Q_i \leq K^*(j)$, we have: 
 	%{\small
 		\begin{align*}
 			(\rom{4}) %\\
 			& \leq K^*(j) \left(\mathbb{E}\left[\sum_{i = \tilde{n}^j}^{\tilde{n}^j+\alpha_jl_2 }  \mathbbm{1}_{ \left\{\left(\mathcal{E}^j_2 \cap \mathcal{E}^j_3\right)^c\right\}}\right]+\mathbb{E}\left[ \sum_{i = \tilde{n}^j\alpha_jl_2}^{\tilde{n}^j+\alpha_jl_2 + L_2^j} \mathbbm{1}_{ \left\{\left(\mathcal{E}^j_2 \cap \mathcal{E}^j_3\right)^c\right\}}\right] \right)\\
 			&\leq K^*(j)\left( \alpha_j l_2 + g(K^*(j);K^*(j))\right)\mathbb{P}\left[  \left(\mathcal{E}^j_2 \cap \mathcal{E}^j_3\right)^c \right].
 		\end{align*}
 		Combining the bounds for $(\rom{3})$ and $(\rom{4})$ , we get: 
 		%{\small
 			\begin{align*}
 				G^j_2 &\leq  \left(R + \frac{C}{\lambda}\right)\bigg(\left(1+ K^*(j)\right)\alpha_jl_2 + \big(    1+ K^*(j)\big)g(K^*(j),K^*(j))\bigg)\mathbb{P}\left[ \left( \mathcal{E}^j_2 \cap \mathcal{E}^j_3\right)^c\right].
 			\end{align*}
 			%}%
 		with $g(l;K) $ defined in Proposition \ref{prop:BusyPeriodArrival}.
\end{proof}
}
\Added{
Before proving the regret bound for Algorithm \ref{alg:Alg3}, the following remark gives an upper bound on the regret accumulated during the first $\floor{e^{\bar{K}}}$ batches where the upper-bound of the threshold used in the phase $2$ of the learning systems may be smaller than $\bar{K}$.
\begin{remark}\label{rem:RegretBeforeReachBound}
	Recall that the queue-length of each batch does not exceed $K^*(j)$ in the $j^{th}$ batch. Following the definition of $K^*(j)$, when $j\geq \ceil{e^{\bar{K}}}$, $K^*(j)\geq \bar{K} + l_1 + Q_0 \geq \bar{K}$.  The regret accumulated during the first $\floor{e^{\bar{K}}}$ batches is at the most 
	\begin{align*}
		G_0 := \left(R + \frac{C}{\lambda}\right)\sum_{j = 1}^{\ceil{e^{\bar{K}}}}  \bigg( K^*(j)+ \bar{K}+1\bigg)\left(l_1 +\alpha_j l_2 + g(K^*(j);K^*(j))\right),
	\end{align*}
	where $g(l;K)$ is defined in Proposition \ref{prop:BusyPeriodArrival}. This bound is loose since it assumes that phase 1 happens at each batch and a worst-case assumption of regret being accumulated at all times is enforced. Note that the bound is a finite function of the system parameters. 
\end{remark} 
}
\Added{
 	\subsection{Proof of theorem \ref{thm:Thm1}}\label{subsection:pfthm1}
 In  the case that $\bar{K}> 0$, using inequality \eqref{eq:RegretInJ}, Lemma \ref{lem:Phase1Reg}, and Lemma \ref{lem:Phase2Reg}, we have: 
 	\begin{align*}
 			&  \sum_{j = \ceil{e^{\bar{K}}}}^{\left\lceil N/l_2\right\rceil} G^j_1+ G^j_2%\\
 			 %&
     \leq \left(R +  \frac{C}{\lambda}\right)\sum_{j =\ceil{e^{\bar{K}}}}^{\left\lceil N/l_2 \right\rceil} \Big(l_1^2 + \left(\bar{K}+1\right) l_1 +\big(    1+ K^*(j)\big)g(l_1;K^*(j))\Big)\mathbb{P}\left[\mathcal{E}^j_1\right]\\
     &\qquad\qquad +\left(R +  \frac{C}{\lambda}\right)\sum_{j =\ceil{e^{\bar{K}}}}^{\left\lceil N/l_2 \right\rceil}  \Big((1+K^*(j))\alpha_jl_2 + \big(    1+ K^*(j)\big)g(K^*(j);K^*(j))\Big) \mathbb{P}\left[ \left(\mathcal{E}^j_2 \cap \mathcal{E}^j_3\right)^c\right].
 	\end{align*}
 Substituting values/bounds for $\mathbb{P}[\mathcal{E}^j_1]$ and $\mathbb{P}[ \left(\mathcal{E}^j_2 \cap \mathcal{E}^j_3\right)^c]$ from Corollary \ref{cor:SameThreshProb} and Proposition \ref{prop:BadEventProb}, we get:
{\footnotesize
\begin{align*}
	&\sum_{j = \ceil{e^{\bar{K}}}}^{\left\lceil N/l_2 \right\rceil} G^j_1+ G^j_2\\
	&\leq \sum_{j = \ceil{e^{\bar{K}}}}^{\left\lceil N/l_2 \right\rceil} \left(R +  \frac{C}{\lambda}\right)\Big(l_1^2 + \left(\bar{K}+1\right) l_1 +\big(    1+ K^*(j)\big)g(l_1,K^*(j))\Big)\frac{\ln^{\epsilon}(j)}{j}\Big(C_1\exp(-C_2\ln^{1+\epsilon}(j))
    +C_3e^{-C_4\beta_j}\Big)\\
 %+C_3\exp(-C_4\beta_j)\right)\\		
	&\qquad +  \sum_{j = \ceil{e^{\bar{K}}}}^{\left\lceil N/l_2 \right\rceil} \left(R +  \frac{C}{\lambda}\right)\big(    1+ K^*(j)\big)\Big( \alpha_jl_2+g(K^*(j),K^*(j))\Big) \\
	&\qquad\qquad\quad \times \bigg( C_1\exp\Big(-C_2\ln^{1+\epsilon}(j-1)\Big)  +C_1\exp\Big(-C_2\ln^{1+\epsilon}(j)\Big)
    + C_3 e^{-C_4\beta_j} + C_3 e^{-C_4 \beta_{j-1}}
 %+ C_3\exp\Big(-C_4\beta_j\Big)+C_3\exp\Big(-C_4\beta_{j-1}\Big)
	+ \big(c_{\bar{K}}\big)^{\alpha_{j-1}l_2}\bigg),
\end{align*} 
}where $g(l;K)$ is defined in Proposition \ref{prop:BusyPeriodArrival} and is of order $O((\lambda/\mu)^{K} + K^2)$.
%\begin{align*}
%&F_1(j) = \left(    1+ K^*(j)\right)C^{*}\left(\left( \frac{\lambda}{\mu}\right)^{K^*(j)+1} + %(K^*(j))^2 +K^*(j)\right), \\
%&F_2(j)=(1+K^*(j)).
%\end{align*}
Recall that $\beta_j\geq j$.  All terms involved are partial sums of convergent series when $\alpha_j$ increases to infinity as a function bounded by polynomial in $j$. Therefore $\lim_{N\rightarrow\infty}G(N)$ is bounded, and the proposed algorithm achieves $O(1)$ regret in the case that $\bar{K}>0$. 
}
\Added{
 	\subsection{Proof of Theorem \ref{thm:Thm2}}\label{subsection:pfthm2}
 Similarly to the proof of Theorem \ref{thm:Thm1}, using inequality \eqref{eq:RegretInJ}, Lemma \ref{lem:Phase1Reg},  Lemma \ref{lem:Phase2Reg}, Corollary \ref{cor:SameThreshProb} and Proposition \ref{prop:BadEventProb}, we have: 
 {\small
 	\begin{align*}
 		&\sum_{j =\ceil{e^{\bar{K}} }}^{\left\lceil N/l_2 \right\rceil}   G^j_1+ G^j_2%\\
 		%&
   \leq    \sum_{j = \ceil{e^{\bar{K}} }}^{\left\lceil N/l_2 \right\rceil}    \left(R +  \frac{C}{\lambda}\right)\left(l_1^2 +  l_1  +C_5\right) \frac{\ln^{\epsilon}(j)}{j}\\
   &\quad \qquad + \sum_{j = \ceil{e^{\bar{K}} }}^{\left\lceil N/l_2 \right\rceil}  \left(R +  \frac{C}{\lambda}\right) \big(    1+ K^*(j)\big)g(l_1,K^*(j)) \Big( C_1\exp(-C_2\ln^{1+\epsilon}(j)) +C_3\exp(-C_4\beta_j)\Big) \\
 		&\quad \qquad+  \sum_{j = \ceil{e^{\bar{K}}}}^{\left\lceil N/l_2 \right\rceil}  \left(R + \frac{C}{\lambda}\right)(1+K^*(j))\Big(\alpha_jl_2 + g(K^*(j);K^*(j))\Big)\Big(C_1\exp(-C_2\ln^{1+\epsilon}(j))+C_3\exp(-C_4\beta_j)\Big).
 	\end{align*}
 	%}%
  }
 The dominant term on the RHS above is $$\sum_{j 
 	= \ceil{e^{\bar{K}}}}^{\left\lceil N/l_2 \right\rceil}  \left(R + \frac{C}{\lambda}\right)\Big(l_1^2 + l_1+ C_5\Big)\frac{\ln^{\epsilon}(j)}{j}.$$ 
  When $N$ is large, we have $$\sum_{j = 2}^{\left\lceil N/l_2 \right\rceil} \frac{\ln^{\epsilon}(j)}{j} = O(\ln^{1+\epsilon}(N)).$$
 % 		Now, for large $J$,
 % 		{\small
 	% 		\begin{align*}
 		% 			\sum_{j = 2}^{J} \frac{\ln(j)}{j} = O(\ln^2(J)).
 		% 		\end{align*}
 	% 		}%
 Hence the regret for $\bar{K} = 0$ is of order $O(\ln^{1+\epsilon}(N))$. 
}
\Added{
\begin{remark}\label{rem:arrival_rate}
    We mentioned earlier that one can adapt the analysis to the case when only the service rate is unknown or only the arrival rate is unknown by adjusting the probability of the learning system using the optimal thresholds in phase 2 and receiving similar regret bounds. As shown in the prof above, in the case when the optimal threshold is 0, the reason why the regret is $O(\ln^{1+\epsilon}(N))$ is that phase 1 is likely to happen infinitely  often so that enough samples of the service rate can be obtained. This explicit exploration phase is necessary when the service rate is unknown. However, when only the arrival rate is unknown, the learning system would always obtain free samples for the arrival rates whether accepting customers to the queue or not. In this case, it would be  unnecessary to explore explicitly, so that an $O(1)$ regret results similar to the case where the optimal threshold is non-zero when one always omits phase $1$ and only the arrival rate is unknown.
\end{remark}
}
\Added{
\begin{remark}\label{rem:order}
    The regret analysis above showed that we can obtain constant regret for the case where the optimal thresholds are non-zeros, and an $O(\ln^{1+\epsilon}(N))$ regret when $0$ is an optimal threshold for any fixed $\epsilon >0$. From the proof of Theorem \ref{thm:Thm2}, the order of the regret is a result of explicit exploration as it is the dominant term. One natural question is the following: can we further reduce the order of the regret in the case that $0$ is an optimal threshold while preserving the constant regret in the case that the optimal threshold is non-zero, if we reduce $\mathbb{P}\left[B^j = 1\right]$, the probability of having phase 1 when the previous phase 2 uses threshold 0? Following the steps of our proof we can show that having $\mathbb{P}\left[B^j = 1\right] = \ln(\ln(j))/j$ would result in regret accumulating slower than $O(\ln^{1+\epsilon}(N))$ for any $\epsilon >0$ in the case that 0 is a optimal threshold, and constant regret in the case that the optimal threshold is non-zero. However, this result would hold for large enough $N$, as the finite time performance of using $\mathbb{P}\left[B^j = 1\right] = \ln(\ln(j))/j$ may not out-perform our discussed choices for $\mathbb{P}\left[B^j = 1\right]$ as it would require $j$ to be extremely large (but still finite) to show improved performance. 
    \end{remark}

    \begin{remark}\label{rem:conjecture}
    We believe that the dramatically different behaviors for our algorithm between cases when $0$ is an optimal threshold, and when it is not, is fundamental to our problem owing to completely different demands in two parameter regimes: in one case, no customers should be dispatched at all, versus the other case where asymptotically a positive fraction of customers are dispatched. Hence, we conjecture that for any given learning-based dispatching algorithm the regret accumulated would grow at least at $\Omega(\ln(N))$ when the parameters are chosen in an adversarial manner. Note that our algorithm satisfies this conjecture. We will argue later on in Section~\ref{section:Numerical} that an Upper-Confidence Bound (UCB) scheme will have a worst-case regret over parameter choices of $\Omega(\ln(N))$. %\VS{See change to $\Omega(\ln(N))$.}\YZ{Thank you.}
\end{remark}
}
	\section{Non-unique admittance threshold case}\label{section:MultiThresh}
	%{We used the condition $V(\bar{K},\mu) \neq R/C$ in our analysis of  the regret of our algorithm in the previous section when we defined $\delta, \delta_0,\Delta, \Delta_0$. This assumption allowed us to proof that the learning threshold will coincide with $\bar{K}$ since there is some freedom for the estimated average service time to be different from the true average service time. Now we drop this assumption. Recall that the optimal threshold $\bar{K}$ is defined as the unique integer satisfying \eqref{eq:ThresholdIneq}. When equality is achieved, i.e. $V(\bar{K},\mu) = R/C$, the learned threshold may oscillate between $\bar{K}-1$ and $\bar{K}$ even when the empirical average service time is accurate and this will make the analysis more involved. }
	%{Naor's paper showed that the social welfare (net gain per unit time) is maximized when using the threshold $\bar{K}$ that satisfies \eqref{eq:ThresholdIneq} by analyzing the stationary distribution of the queue-length process under fixed threshold policies. Note that when a dispatcher uses a fixed threshold, the queue-length process is Markovian and ergodic. \blue{AC: here is my new version: When \eqref{eq:ThresholdIneq} holds with equality and $\bar{K} \geq 1$, both fixed thresholds $\bar K$ and $\bar K-1$ are optimal. It is not hard to see that all the policies that (stochastically) alternate between the thresholds $\bar{K}$ and $\bar{K}-1$ with any fixed probability \blue{AC: why fixed probabilities?} yield the same net gain per unit time.} }
	 When the dispatcher uses a static threshold policy, the queue-length process is Markovian and ergodic.  %Naor~
    \cite{Naor} showed that the social welfare (long-term average profit in \eqref{def:Long-termAverageProfit}) is maximized when using the static threshold $\bar{K}$ that uniquely satisfies \eqref{eq:ThresholdIneq} by analyzing the stationary distributions of the queue-length process for all possible static threshold policies.  When \eqref{eq:ThresholdIneq} holds with equality and $\bar{K} \geq 1$, static thresholds $\bar K$ and $\bar K-1$ are both optimal, and furthermore, policies that (stochastically) alternate between the thresholds $\bar{K}$ and $\bar{K}-1$ with a fixed probability yield the same long-term average profit, i.e., are optimal for the ergodic reward maximization problem. This complicates our regret analysis as we will need to pick a specific ergodic reward-maximizing policy for our regret analysis.
	
	In Section \ref{subsection:IntCaseLearningThresh} we analyze the learned threshold; in Section \ref{subsection:IntCaseAnotherGenie}, we introduce the specific ergodic reward maximizing genie-aided dispatcher that we will compare to, which we will label the alternating genie-aided dispatcher; and finally, Section \ref{subsection:IntCaseRegret} is devoted to the analysis of the regret of the learning algorithm compared to the specific genie-aided dispatcher introduced earlier.

	\subsection{ Threshold used by the learning dispatcher in phase 2. } \label{subsection:IntCaseLearningThresh}
	Following Algorithm \ref{alg:Alg3}, the threshold used by the learning dispatcher in the $j^{th}$ phase $2$ is $K(j) = \min(K^*(j),K)$, where $K$ is the unique integer that satisfies the inequality \Added{$V(K, 1/\hat{m}, \hat{\nu})\leq R/C< V(K+1, 1/\hat{m},1/\hat{\nu})$, where $\hat{m}$ is the empirical average service time, and $\hat{\nu}$ is the empirical inter-arrival time, computed using all completed services and observed arrivals before each phase $2$. As mentioned earlier, the threshold is fixed throughout each phase $2$. Proposition \ref{prop:FuncVProp} implies that as long as the estimations are accurate so that inequalities \eqref{eq:ICRegionOfParam} are satisfied, and when $j\geq \ceil{e^{\bar{K}}}$, the learning dispatcher would use a threshold in $\{ \bar{K}, \bar{K}-1\}$ during the $j^{th}$ phase $2$. Proposition \ref{prop:SampleEst} still holds when equality holds in \eqref{eq:ThresholdIneq}.} Unlike in the previous case where we showed that eventually, the learning dispatcher uses the same threshold $\bar{K}$ as the genie-aided dispatcher in phase $2$, we now show that as the number of batches goes to infinity, the learning algorithm will (eventually) stochastically alternate only between %either uses 
	 the thresholds $\bar{K}$ or $\bar{K}-1$. 

\Added{We first state the analogous of Proposition \ref{prop:AccuServ}, Proposition \ref{prop:AccuArri} and Corollary \ref{cor:SameThreshProb}.
}
\Added{
\begin{prop}\label{prop:ICAccuServ}
	\Added{Let $\hat{m}(j)$ denote  the empirical service time estimated by the learning dispatcher at the beginning of phase 2 of the $j^{th}$ batch.  For the proposed algorithm, in case that $V(\bar{K},\mu,\lambda) = R/C$, we have, 
		\begin{align}\label{eq:ICAccuServ}
			\mathbb{P}\left[ \left\rvert \hat{m}(j) = m\right\rvert>\tilde{\Delta}_1 \right]\leq \tilde{C}_1\exp(-\tilde{C}_2\ln^{1 + \epsilon}(j)),
		\end{align}
		where 
		\begin{align}\label{def:ICc1c2}
			\begin{split}
				\tilde{C}_1 &:= \max\left\{ \exp\left(-\frac{C_0(\epsilon)}{8(1 + \epsilon)}\right),\; \frac{2\exp{\left( \tilde{\Delta}_1^2/(8m^2)\right)}}{\exp{( \tilde{\Delta}_1^2/(8m^2))}-1},\; 1\right\},\\
				\tilde{C}_2 &:= \min\left\{ \frac{l_1\mu}{16(1 + \epsilon)(\lambda+\mu)}, \;\frac{1}{8(1 + \epsilon)}, \;\frac{l_1\mu\tilde{\Delta}_1^2}{32(1 + \epsilon)m(\lambda m+1)}\right\},
			\end{split}
		\end{align}
		with $\tilde{\Delta}_1 := \min\{ \tilde{\delta}_1, 2m\}$, and $\tilde{\delta}_1$ is a constant for the first inequality in \eqref{eq:ICRegionOfParam}  which is one part of the condition needed to reach the conclusion in \eqref{eq:ICCorrectThreshold}.  %\VS{Can we fix this just like previous result?}\YZ{edited.}
	}
\end{prop}
\begin{proof}
	The proof is the same as the proof of Proposition \ref{prop:AccuServ}, but with different constants.
\end{proof}
}
\Added{
	
\begin{prop}\label{prop:ICAccuArri}
	Let $\nu(j)$ denote the empirical inter-arrival time estimated by the learning dispatcher at the beginning of phase 2 of the $j^{th}$ batch. For the proposed algorithm, in case that $V(\bar{K},\mu,\lambda) = R/C$, we have,
	\begin{align*}
		\mathbb{P}\left[ \rvert \nu - \hat{\nu}(j)\rvert >\tilde{\Delta}_2 \right]\leq\tilde{C}_3\exp(-\tilde{C}_4\beta_j)  ,
	\end{align*}
	where
	\begin{align}\label{def:ICc3c4}
%		\begin{split}
			\tilde{C}_3 %&
   :=  \frac{2\exp(\tilde{\Delta}_2^2/(8\nu^2))}{\exp(\tilde{\Delta}_1^2/(8\nu^2))-1}\qquad\text{ and }\qquad%\\
			\tilde{C}_4 %&
   :=  \frac{l_1\tilde{\Delta}_2^2}{8\nu^2},
%		\end{split}
	\end{align}
	where $\beta_j$ is defined in Proposition \ref{prop:AccuArri}, and $\tilde{\Delta}_2:= \min\{ \tilde{\delta}_2, 2\nu\}$,
	where $\tilde{\delta}_2$ is the constant in the second inequality in \eqref{eq:ICRegionOfParam} that is the second part needed to reach the conclusion in \eqref{eq:ICCorrectThreshold}. %\VS{Can we fix this just like previous result?}\YZ{edited.}
\end{prop}
\begin{proof}
	The proof is the same as the proof of Proposition \ref{prop:AccuArri}, but with different constants.
\end{proof}
}
\Added{
	 	\begin{cor}\label{cor:ICSamePolicy}
	 		For the proposed algorithm, when $j\geq \ceil{e^{\bar{K}}}$ , in case that $V(\bar{K},\mu,\lambda) = R/C$,
	 		\begin{align}
	 			\mathbb{P}\left[ \{K(j) \neq \bar{K} \} \cap\{ K(j) \neq \bar{K}-1\} \right]&\leq \tilde{C}_1\exp(-\tilde{C}_2\ln^{1 + \epsilon}(j)) +\tilde{C}_3\exp(-\tilde{C}_4\beta_j)  ,
	 		\end{align}
	 		where $\tilde{C}_1$ and $\tilde{C}_2$ are defined in \eqref{def:ICc1c2} and $\tilde{C}_3$ and $C_4$  are defined in \eqref{def:ICc3c4}.
	 		% 		\end{align*}
	 
	 \end{cor}
	 \begin{proof}
	The proof for this proposition follows the same logic as the proof of  Corollary \ref{cor:SameThreshProb}, but with different constants.
	 \end{proof}
}
 	\begin{cor} \label{cor:OptInd}
 		In case that $V(\bar{K},\mu,\lambda) = R/C$, there exists a random index $\mathcal{J}$ that is finite with probability $1$, where the learning algorithm would use threshold $\bar{K}$ or $\bar{K}-1$ after the $\mathcal{J}^{th}$ batch. 
   %\VS{Changed $\mathcal{I}$ to $\mathcal{J}$.}
 	\end{cor}
 	\begin{proof}
 		We show that the learning algorithm uses thresholds that are not $\bar{K}$ nor $\bar{K}-1$ only finitely many times with probability $1$. From Corollary \ref{cor:ICSamePolicy}, when $\bar{K}>1$, we have:
 		\begin{align*}
 			\sum_{j = 1}^{\infty} \mathbb{P}\left[\Big(\{K(j) = \bar{K}\}\cup\{K(j) = \bar{K}-1\}\Big)^c\right] \leq \sum_{j = 1}^{\infty}  \tilde{C}_1\exp(-\tilde{C}_2\ln^{1 + \epsilon}(j)) +\tilde{C}_3\exp(-\tilde{C}_4j^2) <\infty.
 		\end{align*}
 	By the Borel–Cantelli lemma (See \cite{Durrett}), we have 
  $$\mathbb{P}\left[ \limsup_{j\rightarrow \infty} \Big(\{K(j) = \bar{K}\}\cup\{K(j) = \bar{K}-1\}\Big)^c\right] = 0,$$ 
    that is,  with probability $1$, the learning algorithm uses thresholds not in $\{\bar{K},\bar{K}-1\}$ only a finite number of times. Thus, almost surely the learning algorithm uses the optimal thresholds $\bar{K}$ and $\bar{K}-1$ after a finite random time. When $\bar{K} = 1$, a similar proof holds. 
 	\end{proof}
 
	\subsection{An alternating genie-aided dispatcher coupled with the learning dispatcher that maximizes the long-term average profit} \label{subsection:IntCaseAnotherGenie} If we compare our learning algorithm with a genie-aided system that uses a static threshold $\bar{K}$ (or alternatively $\bar{K}-1$), the regret will not be constant even when $\bar{K}>1$. The reason is that the learning dispatcher may switch between the thresholds $\bar{K}$ and $\bar{K}-1$ in different phase $2$s even when $\hat{m} \in (m - \epsilon, m+\epsilon)$, where $\epsilon$ is sufficiently small. However, we can compare the queue-length process under the learning dispatcher with an  optimal  genie-aided dispatcher to which we refer to as the \emph{alternating genie-aided dispatcher}: a dispatcher who may change the threshold used between $\bar{K}$ and $\bar{K}-1$ at the beginning of any busy cycle (a busy period plus an immediately following idle period). \Added{We will ensure that the threshold-changing policy of this alternating genie-aided dispatcher is adapted to the filtration generated by the queue-lengths of the two systems and the random variable $B^j$, with the threshold remaining unchanged during each busy cycle}. %a dispatcher who may change the threshold used between $\bar{K}$ and $\bar{K}-1$ at the beginning of any busy cycle (a busy period plus an immediately following idle period) according to a specific policy adapted to the filtration generated by the queue-lengths of the learning and genie-aided systems with the threshold used being unchanged during each busy cycle. 
It is worth mentioning that although the learning dispatcher may compute and change the threshold at the beginning of each phase $2$ (which may involve multiple busy cycles), only the genie-aided dispatcher may change the threshold at the beginning of a busy cycle. %\VS{Threshold being unchanged during a busy cycle doesn't apply to the learning system!} 
 This alternating genie-aided dispatcher is aware of the fact that the learning dispatcher follows Algorithm \ref{alg:Alg3} and can compute the threshold learned by the learning dispatcher. This alternating genie-aided dispatcher is coupled with the learning dispatcher under the coupling described in Section \ref{coupling}. \Added{Moreover, when a customer arrives, having seen the realization of $B^j$, this genie-aided dispatcher is aware of whether this customer arrives during a phase $1$ or $2$ of the learning system, and would pick the proper threshold to use when this customer initiates a busy cycle.} 
	
	Recall that $K_{i}$ denotes the threshold used by the learning system at the arrival of the $i^{th}$ customer. Following similar notation as in Section \ref{section:ProblemSetUp} for the alternating genie-aided dispatcher,  let $\tilde{K}_i$ denote the threshold policy used at the arrival of the $i^{th}$ customer, $\tilde{Q}_i$ denote the queue-length right before the arrival of the $i^{th}$ customer, $\tilde{Q}(t)$ denote the queue-length at time $t$,  $\tau^B_n$ denote the time of the beginning of the $n^{th}$ busy cycle,  $\tilde{N}_A(\tau^B_n)$ denote the index of the arrival customer who arrives at the beginning of the $n^{th}$ busy cycle, $\tilde{N}(t)$ denote the total number of completed busy cycles up to time $t$, and $\tilde{K}^n$ denote the threshold used during the $n^{th}$ busy cycle; note that $\tau_1^B=0$. 
 % \VS{These definitions should come outside this Proposition and before (5.1). Use $\tau^B_n$ for time of beginning of $n^{\mathrm{th}}$ busy cycle and $B_n=\tau^B_{n+1}-\tau^B_n$ to be length of the $n^{\mathrm{th}}$ busy cycle.} \YZ{moved before  the statement before Proposition \ref{prop:AnotherGenieOpt}.}%Let $\tilde{K}^n $ denote the threshold used by this specific genie-aided system at the $n^{th}$ busy cycle, and $\tilde{N}_A(\tau_n)$ denote the index of the arrival customers who arrive at the beginning of the $n^{th}$ busy cycle. 
	 %\VS{In Algorithm thresholds are defined for a batch as $K(j)$ so you need to make another definition here for $K_i$ using the batch that the arrival belongs to.} \YZ{Added links of $K_i$ and $K(j)$ when introducing the learning algorithm. }
 At the beginning of each busy cycle, the alternating genie-aided dispatcher then chooses a threshold $\tilde{K}^n \in \{\bar{K},\bar{K}-1\}$, where we have
\Added{
 \begin{align}\label{eq:SwitchRule}
		 {\small \tilde{K}^n =
			\begin{cases}
			\bar{K}-1, & \text{if } n=1, \\
			\bar{K}-1, & \text{if } n>1 \text{ and } \{K_{\tilde{N}_A(\tau^B_n)} \leq \bar{K}-1 \text{ OR customer }\tilde{N}_A(\tau^B_n) \text{ arrives during phase $1$} \},\\
			\bar{K}, & \text{if } n>1 \text{ and } \{K_{\tilde{N}_A(\tau^B_n)}\geq \bar{K} \text{ AND customer }\tilde{N}_A(\tau^B_n) \text{ arrives during phase $2$}  \}. 
		\end{cases}
	}
\end{align}
}
\Added{
That is, when the customer who initiates a busy cycle in the genie-aided system arrives during phase $1$ of the learning system, the genie-aided dispatcher uses threshold $\bar{K}-1$ in the initiated busy cycle. When the customer arrives during phase $2$, in the initiated busy cycle, the genie-aided dispatcher uses a threshold from $\{ \bar{K}, \bar{K}-1\}$ that is closer to the threshold used by the learning system. This threshold choice would help to preserve the queue-lengths ordering under desired events, as explained in subsection \ref{subsection:IntCaseRegret}.}
		% \VS{What is $\tau_n$? In 5.1 you use $\tau$ for the index of an arrival but now you're redefining it for real-valued time. Do use a different symbol.}\YZ{now using  $\tau^B_n$ to denoted the time of the beginning to the $n^{th}$ busy cycle as suggested.}
 In other words, for customers $i_1$ and  $i_2$  who arrive during the $n^{th}$ busy cycle, i.e., $\tilde{N}_A(\tau^B_n)\leq i_1<i_2< \tilde{N}_A(\tau^B_{n+1})$, we have $\tilde{K}_{i_1} = \tilde{K}_{i_2} = \tilde{K}^n$. This switching policy is adapted to the filtration generated by the queue-lengths of the genie-aided and learning systems. Since the learning algorithm always has the first exploration phase, we set \Added{$\tilde{K}^1 = \bar{K}-1$.}
			
	The following proposition shows the optimality of the alternating genie-aided dispatcher described above using the strong law of large numbers for martingales.
	\begin{prop}\label{prop:AnotherGenieOpt}
		Consider a dispatcher who uses a  static threshold policy, either $\bar{K}$ or $\bar{K}-1$, during a busy cycle, and may switch between these two thresholds only at the beginning of a busy cycle following  the switching rule described in \eqref{eq:SwitchRule}. 
% 		\red{Done. }\VS{Maybe good to formally define the switching using a definition so that it can be referenced.} 
		The long-term average profit of the system under this dispatcher is the same as a dispatcher using either one of the static thresholds $\bar{K}$ or $\bar{K} - 1$. 
	\end{prop} 
	\begin{proof}
		%Assume the initial queue-length is a. 
		Assume the initial queue-length is some $a\in \{0,1,\dotsc,\bar{K}\}$, where the particular value doesn't impact the asymptotic results. 
  %\VS{Change as - Assume the initial queue-length is some $a\in \{0,1,\dotsc,\bar{K}\}$, where the particular value doesn't impact the asymptotic results.}\YZ{Done.}
% 		\VS{Where does this come in and why is it without loss of generality?} \magenta{Yili: If the initial queue-length is not 1, then the busy cycles are no longer i.i.d. I was trying to avoid this case.} \VS{The issue will be for the first busy cycle, so you can use results on delayed renewal processes to fix things.}\magenta{I removed this assumption, and assumed the initial queue-length to be a. Now the indexes in the summations in this proof have also been edited.}
		We are interested in finding: %\VS{You should account for the $a$ customers already in the queue.}\YZ{Now included the customers that are already in queue at t = 0.}
			\begin{align}
	%			\label{ergowelfare}
	%			\begin{split}
					& \liminf_{t \rightarrow \infty} \frac{1}{t} \left(aR+\sum_{i= 1}^{\tilde{N}_A(t)} R \mathbbm{1}_{\{\tilde{Q}_i \leq \tilde{K}_i\} } - \int_{0}^{t}C\tilde{Q}(u)du\right) \nonumber\\
					& = \liminf_{t \rightarrow \infty} \frac{1}{t} \left(aR+\sum_{i = 1}^{\tilde{N}_A(\tau^B_{2})-1} R\mathbbm{1}_{\{\tilde{Q}_i \leq \tilde{K}^1\}}-\int_0^{\tau_2^B} C \tilde{Q}(u)du \right) \nonumber\\
     & \quad\quad\quad +\liminf_{t \rightarrow \infty} \frac{1}{t} \left(\sum_{n= 2}^{\tilde{N}(t)} \left(\sum_{i = \tilde{N}_A(\tau^B_n)}^{\tilde{N}_A(\tau^B_{n+1})-1} R\mathbbm{1}_{\{\tilde{Q}_i \leq \tilde{K}^n\}} - \int_{\tau^B_n}^{\tau^B_{n+1}}C\tilde{Q}(u)du\right) \right)\nonumber\\
					&\quad\quad\quad+ \liminf_{t \rightarrow \infty}\frac{1}{t} \left(\sum_{i = \tilde{N}_A\left(\tau^B_{\tilde{N}(t)+1}\right)}^{\tilde{N}_A(t)} R \mathbbm{1}_{\{ \tilde{Q}_i\leq \tilde{K}^{\tilde{N}(t)+1}\}} - \int_{\tau^B_{\tilde{N}(t)+1}}^{t}C\tilde{Q}(u)du\right). 
     		\label{ergowelfare}
	%			\end{split}
			\end{align}
		%\magenta{Since $\liminf_{t \rightarrow \infty} aR/t = 0$ for all $a \in \{ 0, ... \bar{K}\}$, equity holds in \eqref{ergowelfare}.}
	  % \VS{Define $X_1,\tilde{t}_1$ separately from rest. Use the following: let $(X_n,B_n)$ denote the total net gain and duration of the $n^{\mathrm{th}}$ busy cycle under this dispatcher. Note that the first busy cycle starts with the initial queue-length of $a$ whereas all the other busy cycles start with an initial queue-length of $1$. Definition of $X_1$ mathematically needs to account for initial queue-length terms as well so it needs to be different.} 
   Let the tuple $(X_n,\mathcal{B}_n)$ denote the total net profit and duration of the $n^{th}$ busy cycle under this dispatcher. 
   % \YZ{changed the notation for the duration of the  $n^{th}$ busy cycle to be $\mathcal{B}_n$, and use upper scripts to distinguish the ones using different static thresholds.} 
   For the first busy cycle, we have:
	  \begin{align*}
	  		X_1 &:=aR+ \sum_{i = 1}^{\tilde{N}_A(\tau^B_{2})-1} R\mathbbm{1}_{\{\tilde{Q}_i \leq \tilde{K}^1\}} - \int_{0}^{\tau^B_{2}}C\tilde{Q}(u)du,\text{ and }%\\
	  		\mathcal{B}_1 :=  \tau^B_{2}.
	  \end{align*}
  For $n \geq 2$, we have:
	\begin{align*}
		X_n &:= 	\sum_{i = \tilde{N}_A(\tau_n)}^{\tilde{N}_A(\tau^B_{n+1})-1} R\mathbbm{1}_{\{\tilde{Q}_i \leq \tilde{K}^n\}} - \int_{\tau^B_n}^{\tau^B_{n+1}}C\tilde{Q}(u)du,\text{ and }%\\
		\mathcal{B}_n := \tau^B_{n+1} - \tau^B_{n}.
	\end{align*} 
	We can rewrite \eqref{ergowelfare} as: 
	\begin{align*}
		\liminf_{t \rightarrow \infty} \frac{1}{t} \sum_{n = 2}^{\tilde{N}(t)} X_n + \liminf_{t \rightarrow \infty}\frac{1}{t} \left(X_1 + \sum_{i = \tilde{N}_A(\tau^B_{\tilde{N}(t)+1})}^{\tilde{N}_A(t)} R \mathbbm{1}_{\{ \tilde{Q}_i\leq \tilde{K}^{\tilde{N}(t)+1}\}} - \int_{\tau^B_{\tilde{N}(t)+1}}^{t}C\tilde{Q}(u)du\right). 
	\end{align*}
When the initial queue-length is finite, $\mathbb{E}\left[ \mathcal{B}_1\right]$ and $\mathbb{E}[\left( \mathcal{B}_1\right)^2]$ are finite; see \cite{Takagi2009}.
	
	Let $(Y_n^{\bar{K}}, \mathcal{B}_n^{\bar{K}})$ denote the total net profit and the duration of the $n^{th}$ busy cycle of a dispatcher that uses static threshold $\bar{K}$ and with initial queue-length $1$, and let  $\mathcal{Y}^{\bar{K}}(t)$ denote the accumulated total net profit of this dispatcher up to time $t$. 
 %\VS{Use $B^{\bar{K}}_n$ and $B^{\bar{K}-1}_n$ to denote the respective busy cycle duration.} \YZ{using $\mathcal{B}^{\bar{K}}_n$ and $\mathcal{B}^{\bar{K}-1}_n$ instead, $B$ was used to denote a Bernoulli r.v. used in the learning algorithm.} 
 Setting the initial queue-length to $1$ is owing to a generic busy cycle starting as such. The random variables $(Y_n^{\bar{K}}, \mathcal{B}_n^{\bar{K}})$ are  \emph{i.i.d.}, and $\mathcal{Y}^{\bar{K}}(t)$ is a renewal reward process: see \cite[Section~3.1]{Durrett}. Similarly, we can define $(Y_n^{\bar{K}-1}, \mathcal{B}_n^{\bar{K}-1})$  and $\mathcal{Y}^{\bar{K}-1}(t)$ for a dispatcher that uses static threshold $\bar{K}-1$. %Naor~
 \cite{Naor} showed that there exists a constant $\mathcal{O}$ denoting the optimal long-term average profit of the dispatcher, where with probability 1, 
\begin{align*}
	\lim_{t \rightarrow \infty} \frac{1}{t}\mathcal{Y}^{\bar{K}}(t) = \lim_{t \rightarrow \infty} \frac{1}{t}\mathcal{Y}^{\bar{K}-1}(t) =\mathcal{O}.
\end{align*}
By the renewal-reward theorem, \cite[Section~3.1]{Durrett}, we have:
\begin{align*}
	\mathbb{E}\left[Y_1^{\bar{K}}\right] =  \mathbb{E}\left[\mathcal{B}_1^{\bar{K}}\right] \mathcal{O},  \quad\text{ and }\quad \mathbb{E}\left[Y_1^{\bar{K}-1}\right] =  \mathbb{E}\left[\mathcal{B}_1^{\bar{K}-1}\right] \mathcal{O}. 
\end{align*}

Let $\tilde{\mathcal{F}}_{n-1}:=\tilde{\mathcal{F}}_{\tau_n} $ denote the sigma-algebra generated by the queue-length process of the coupled learning dispatcher and the dispatcher described in Proposition \ref{prop:AnotherGenieOpt} up to time $\tau^B_n$ (the end of the $(n-1)^{th}$ busy cycle of the dispatcher described in Proposition \ref{prop:AnotherGenieOpt}). By the independence of the Poisson arrival and Poisson potential service process, the distribution of $(X_n, \mathcal{B}_n)$ conditioned on $\tilde{\mathcal{F}}_{n-1}$  is the same as the distribution of  $(X_n, \mathcal{B}_n)$ conditioned on the filtration generated by $\tilde{K}^n$.  Moreover, for $n\geq 2$, $(X_n, \mathcal{B}_n)$ conditioned on the event $\{ \tilde{K}^n = \bar{K}\}$ has the same distribution as $(Y_1^{\bar{K}} , \mathcal{B}_{1}^{\bar{K}})$ and $(X_n, \mathcal{B}_n)$ conditional on the event $\{ \tilde{K}^n = \bar{K}-1\}$ has the same distribution as $(Y_1^{\bar{K}-1} , \mathcal{B}_{1}^{\bar{K}-1})$. Using these, for $i\geq 2$, we have: 
	\begin{align*}
		\mathbb{E}\left[ \mathcal{B}_n\right] &= \mathbb{E}\left[\mathcal{B}_n \bigg\rvert \tilde{K}^n = \bar{K}  \right]\mathbb{P}\left[ \tilde{K}^n = \bar{K} \right]+ \mathbb{E}\left[\mathcal{B}_n \bigg\rvert \tilde{K}^n = \bar{K} -1 \right]\mathbb{P}\left[ \tilde{K}^n = \bar{K}-1 \right]\\
		 &=   \mathbb{E}\left[\mathcal{B}^{\bar{K}}_1 \right]\mathbb{P}\left[\tilde{K}^n = \bar{K}\right]+  \mathbb{E}\left[\mathcal{B}^{\bar{K}-1}_1\right]\mathbb{P}\left[\tilde{K}^n = \bar{K}-1\right], 
	\end{align*}
and similarly,
\begin{align*}
	\mathbb{E}\left[ (\mathcal{B}_n)^2\right] &=  \mathbb{E}\left[ (\mathcal{B}_n)^2 \bigg\rvert \tilde{K}^n = \bar{K}\right]\mathbb{P}\left[ \tilde{K}^n = \bar{K}\right] +\mathbb{E}\left[ (\mathcal{B}_n)^2 \bigg\rvert  \tilde{K}^n = \bar{K}-1 \right] \mathbb{P}\left[  \tilde{K}^n = \bar{K}-1 \right] \\
	& = \mathbb{E}\left[ (\mathcal{B}_1^{\bar{K}})^2 \right]\mathbb{P}\left[ \tilde{K}^n = \bar{K}\right] +\mathbb{E}\left[ (\mathcal{B}_1^{\bar{K}-1})^2 \right] \mathbb{P}\left[  \tilde{K}^n = \bar{K}-1 \right].
\end{align*}
Both $\mathcal{B}_1^{\bar{K}}$ and $\mathcal{B}_1^{\bar{K}-1}$ have  finite first and second moments, \cite{Takagi2009}, and thus, so does $\mathcal{B}_i$. 

Let $\tilde{N}_{\mathrm{join}}^n$ denote the number of the customers joining the queue during the $n^{th}$ busy cycle under the dispatching policy described in Proposition \ref{prop:AnotherGenieOpt}. 
% \VS{Choose something different from $\tilde{N}_J$ as it impossible to parse. Plus, $J$ already refers to the number of batches.} \YZ{renamed to $\tilde{N}_{\mathrm{join}}^n$}
Observe that the total number of arrivals joining the queue and services are equal during a busy cycle except for the first one for which there are exactly $a$ more service completions than the number of customers joining the queue during the first busy cycle. 
% \VS{This is not true for the first busy cycle as there are $a$ more departures than arrivals.}\YZ{edited.} 
When there are at least $\bar{K}$ potential services between two consecutive arrivals, the queue-length under the dispatcher described in Proposition \ref{prop:AnotherGenieOpt} hits $0$ and a busy period ends. Therefore, for any integer $M$, %\VS{with probability 1?}\YZ{added.}\VS{I don't understand the statement w.p.1 as you're computing a probability, and not a conditional probability.}\YZ{My bad. I thought you want me to add "with probability 1" but now I realized that we want to remove it.}, 
we have: 
\begin{align*}
	\mathbb{P}\left[\tilde{N}^n_{\mathrm{join}} > M \right] \leq  \left(1-\left(\frac{\mu}{\lambda+\mu}\right)^{\bar{K}}\right)^M, 
\end{align*} 
which then implies that the random variable $\tilde{N}^i_J$ has finite first and second moments. %\VS{Check formula above, again the $\lambda/(\lambda+\mu)$ term is not needed - you are looking at the complement of at least $\bar{K}$ potential services.}\YZ{edited.}

Since $\rvert X_n \rvert \leq R\tilde{N}^n_{\mathrm{join}} + C\bar{K}\mathcal{B}_n$, a.s., for all $n\geq2$, and  $\rvert X_1 \rvert \leq R\tilde{N}^1_{\mathrm{join}} + aR+ C\bar{K}\mathcal{B}_1$ a.s., we can conclude that $X_n$ also has finite first and second moments, and it is clear that with probability $1$, 
% \VS{The claim for $X_1$ needs a little more work.}\YZ{added the influence of the initial queue-length}
$$\liminf_{t \rightarrow \infty}\frac{1}{t} \left(X_1+ \sum_{n =\tilde{N}_A\left(\tau^B_{\tilde{N}(t)+1}\right)}^{\tilde{N}_A(t)} R \mathbbm{1}_{\{ \tilde{Q}_i\leq \tilde{K}^{\tilde{N}(t)+1}\}} - \int_{\tau^B_{\tilde{N}(t)+1}}^{t}C\tilde{Q}(u)du \right)= 0.$$
	For almost every sample path, there exists $t^*$ such that $\tilde{N}(t)>1$ for all $t\geq t^*$, and we have the following upper and lower bounds with probability $1$: 
	\begin{align*}
		 \liminf_{t \rightarrow \infty}  \frac{1}{\sum_{n= 1 }^{\tilde{N}(t)+1} \mathcal{B}_i} \sum_{n = 2}^{\tilde{N}(t)} X_n  \leq \liminf_{t \rightarrow \infty}  \frac{1}{t} \sum_{n = 2}^{\tilde{N}(t)} X_n\leq \liminf_{t \rightarrow \infty}  \frac{1}{\sum_{n= 2}^{\tilde{N}(t)} \mathcal{B}_n} \sum_{i = 2}^{\tilde{N}(t)} X_n. 
	\end{align*}
	We show $ \liminf_{t \rightarrow \infty}  (1/t) \sum_{n = 2}^{\tilde{N}(t)} X_n = \mathcal{O}$ a.s. by showing that with probability 1, both 
	\begin{align}
	 & \liminf_{t \rightarrow \infty}  \frac{1}{\sum_{n= 1}^{\tilde{N}(t)+1} \mathcal{B}_n} \sum_{n = 2}^{\tilde{N}(t)} X_n  = \mathcal{O}, \text{ and}	 \label{limitprop5.2lower}\\
		 &\liminf_{t \rightarrow \infty}  \frac{1}{\sum_{n= 2}^{\tilde{N}(t)} \mathcal{B}_n} \sum_{n = 2}^{\tilde{N}(t)} X_n  = \mathcal{O}. \label{limitprop5.2upper}
	\end{align}
	Note that we have: 
	\begin{align*}
		 \liminf_{t \rightarrow \infty}  \frac{1}{\sum_{n= 1}^{\tilde{N}(t)+1} \mathcal{B}_n} \sum_{n = 2}^{\tilde{N}(t)} X_n  & = \liminf_{t \rightarrow \infty}  \frac{\sum_{n= 2}^{\tilde{N}(t)} \mathcal{B}_n}{\sum_{n= 1}^{\tilde{N}(t)+1} \mathcal{B}_n} \frac{1}{\sum_{n= 2}^{\tilde{N}(t)} \mathcal{B}_n}\sum_{n = 2}^{\tilde{N}(t)} X_n \\
		& = \liminf_{t \rightarrow \infty} \frac{\tilde{N}(t)+1}{\sum_{n= 1}^{\tilde{N}(t)+1} \mathcal{B}_n} \times \frac{\sum_{n= 2}^{\tilde{N}(t)} \mathcal{B}_n}{\tilde{N}(t)-1} \times \frac{\tilde{N}(t)-1}{\tilde{N}(t)+1} \times \frac{1}{\sum_{n= 2}^{\tilde{N}(t)} \mathcal{B}_n}\sum_{n = 2}^{\tilde{N}(t)} X_n.
	\end{align*}
We can also  rewrite \eqref{limitprop5.2upper} as 
\begin{align*}
	\liminf_{n \rightarrow \infty}  \frac{\tilde{N}(t)-1}{\sum_{n = 2}^{\tilde{N}(t)}\mathcal{B}_n} \frac{1}{\tilde{N}(t)-1}\sum_{n = 2}^{\tilde{N}(t)} \left(X_n -\mathcal{B}_n\mathcal{O} \right) =0.
\end{align*}
Note that $\lim_{t  \rightarrow \infty} \tilde{N}(t) = \infty$ and $\lim_{t  \rightarrow \infty} \sum_{n = 2}^{\tilde{N}(t)}\mathcal{B}_n = \infty$ a.s., which in turn imply that a.s. we have:
 \begin{align*}
 	\liminf_{t \rightarrow \infty}\frac{\tilde{N}(t)+1}{\sum^{\tilde{N}(t)+1}_{n = 1} \mathcal{B}_n} =\liminf _{k\rightarrow \infty}\frac{k}{\sum_{n= 1}^{k} \mathcal{B}_n}  =\liminf_{t \rightarrow \infty}\frac{\tilde{N}(t)-1}{\sum^{\tilde{N}(t)}_{n = 2} \mathcal{B}_n} \text{ and  } \lim_{t \rightarrow \infty} \frac{\tilde{N}(t)-1}{\tilde{N}(t)+1} = \lim_{k \rightarrow \infty}\frac{k-1}{k+1} = 1 .
 \end{align*}
Then, in order to establish \eqref{limitprop5.2lower} and \eqref{limitprop5.2upper}, it is sufficient to show that with probability 1,
\begin{align}
	& \liminf_{k \rightarrow \infty}  \frac{1}{k-1} \sum_{n = 2}^{k}\left(X_n   - \mathcal{B}_n \mathcal{O} \right)  = 0, \text{ and} \label{limitisO}\\
	0< & \liminf_{k\rightarrow \infty} \frac{k}{\sum_{n= 1}^{k}\mathcal{B}_n } \leq \limsup_{k\rightarrow \infty} \frac{k}{\sum_{n= 1}^{k}\mathcal{B}_n}  <\infty. \label{limitbounded}
\end{align}

We will prove \eqref{limitisO} by using the strong law of large numbers for martingales~\cite[Theorem 1]{mtgLLN}. %\VS{Add a specific reference to the result or state it directly.}\YZ{reference added.} 
Let $M_k = \sum_{n = 2}^{k} \left(X_n -\mathcal{B}_n\mathcal{O} \right)$ for $k\geq 2$, $M_1 = 0$. Clearly $\mathbb{E}\left[\rvert M_k\rvert\right]<\infty$ for all $k$. Also, 
\begin{align}
	\mathbb{E}\left[M_{k+1} - M_k \Big\rvert \tilde{\mathcal{F}_{k}}\right] %\nonumber\\
	& = \mathbb{E}\left[X_{k+1}   - \mathcal{B}_{k+1}\mathcal{O} \Big\rvert \tilde{\mathcal{F}_{k}}\right] \nonumber\\
	& = \mathbb{E}\left[X_{k+1}   - \mathcal{B}_{k+1}\mathcal{O} \Big\rvert \tilde{K}^k \right] \nonumber\\
	& = \mathbbm{1}_{\{\tilde{K}^{k+1} = \bar{K}\}}\mathbb{E}\left[Y^{\bar{K}}_{1} - \mathcal{B}^{\bar{K}}_{1}\mathcal{O}\right]+ \mathbbm{1}_{\{\tilde{K}^{k+1} = \bar{K}-1\}}\mathbb{E}\left[Y^{\bar{K}-1}_{1} - \mathcal{B}^{\bar{K}-1}_{1}\mathcal{O} \right]= 0.
\end{align} 
The second equality follows since the distribution of $(X_n, \mathcal{B}_n)$ conditioned on $\tilde{\mathcal{F}}_{n-1}$ is the same as the distribution of $(X_n, \mathcal{B}_n)$ conditioned on the filtration generated by $\tilde{K}^n$ for all $n\geq 2$. Therefore, we have shown that $M_k$ is a martingale with respect to filtration $\{\tilde{\mathcal{F}}_k\}_{k\geq 1}$ with martingale difference sequence $X_{k} - \mathcal{B}_k\mathcal{O}$ for $k\geq 2$.

Next, we will show that $\sum_{k = 2}^{\infty } k^{-2}\mathbb{E} \left[(X_{k} - \mathcal{B}_k \mathcal{O})^2\right]$ is finite. For $k\geq 2$, we have:
\begin{align*}
	\mathbb{E}\left[ (X_{k} - \mathcal{B}_k\mathcal{O})^2\right] %\\
	& = \mathbb{E}\left[ \left( \sum_{i = \tilde{N}_A(\tau^B_k)}^{\tilde{N}_A(\tau^B_{k+1})-1} R\mathbbm{1}_{\{\tilde{Q}_i \leq \tilde{K}^k\}} - \int_{\tau^B_k}^{\tau^B_{k+1}}C\tilde{Q}(u)du - \mathcal{B}_n \mathcal{O}\right)^2\right]\\
	&\leq  \mathbb{E}\left[ \left( \sum_{i = \tilde{N}_A(\tau^B_k)}^{\tilde{N}_A(\tau^B_{k+1})-1} R\mathbbm{1}_{\{\tilde{Q}_i \leq \tilde{K}^k\}} \right)^2+ \left( \int_{\tau^B_k}^{\tau^B_{k+1}}C\tilde{Q}(u)du + \mathcal{B}_n O\right)^2\right]\\
	&\leq \mathbb{E}\left[ R^2(\tilde{N}_\mathrm{join}^k)^2+ (\mathcal{B}_k )^2(\mathcal{O}+C\bar{K})^2 \right], 
\end{align*}
where we recall that $\tilde{N}_\mathrm{join}^k$ denotes the customers joining the queue during the $k^{th}$ busy cycle, and $\mathcal{B}_k = \tau^B_{k+1} - \tau^B_{k}$ is the duration of the $k^{th}$ busy cycle. 
% \VS{Choose something different from $\tilde{N}_J$ as it impossible to parse.}\YZ{renamed.} 
When $k\geq 2$, both $\tilde{N}_\mathrm{join}^k$ and $\mathcal{B}_k$ have finite second moments that do not depend on $k$, so that $\sum_{k = 2}^{\infty } k^{-2}\mathbb{E} \left[(X_{k} - \mathcal{B}_k \mathcal{O})^2\right]<\infty.$
Therefore, by the strong law of large numbers for martingales~\cite[Theorem 1]{mtgLLN}, \eqref{limitisO} holds.

Next, we prove \eqref{limitbounded}. Consider a dispatcher that uses the  static threshold policy $\bar{K}$, which is coupled with the dispatcher described in Proposition \ref{prop:AnotherGenieOpt}, and also has initial queue-length $a$. The duration of the $n^{th}$ busy cycles of this dispatcher is denoted  $\tilde{\mathcal{B}}^{\bar{K}}_n$. 
% \VS{With the new symbols you can use $\tilde{B}^{\bar{K}}_n$ and $\tilde{B}^{\bar{K}-1}_n$ now. However, do point out that the first random variable is independent but distributed differently from the rest which are \emph{i.i.d.}} \YZ{Done.} 
The random variables $\tilde{\mathcal{B}}^{\bar{K}}_n$s are \emph{i.i.d.}~for all $n\geq 2$. Although having a different distribution, $\tilde{\mathcal{B}}^{\bar{K}}_1$ is independent of $\tilde{\mathcal{B}}^{\bar{K}}_n$ for all $n\geq 2$. 

Using Proposition \ref{prop:OrderedSystems}, observe that on any sample path, when the dispatcher that uses the static threshold $\bar{K}$ has experienced  $k$ busy periods, the dispatcher described in Proposition \ref{prop:AnotherGenieOpt} would have experienced more than $k$ busy periods. Thus, we can conclude that, with probability $1$, 
$$\sum_{n = 1}^{k}\tilde{\mathcal{B}}^{\bar{K}}_i \geq \sum_{n = 1}^{k}\mathcal{B}_k,$$ for all $k$. Moreover, since $\mathcal{B}^{\bar{K}}_n$s have finite first moments, \cite{Takagi2009}, and are non-negative, they are finite a.s. Therefore, $\lim_{k\rightarrow \infty} k/\sum_{n = 1}^{k}\tilde{\mathcal{B}}^{\bar{K}}_n  = 1/\mathbb{E}[ \mathcal{B}^{\bar{K}}_2]$ exists a.s. and is strictly positive. 
% \VS{This needs to be qualified since $\tilde{B}^{\bar{K}}_1$ has a different distribution from $\tilde{B}^{\bar{K}}_n$ for $n\geq 1$. In essence, $\tilde{B}^{\bar{K}}_1$ needs to be finite \emph{a.s.} Also, many of these limits in this result are for random variables, so please use \emph{a.s.} with them. The same is true for comparisons.} 
Therefore, with probability $1$, we have:
 \begin{align*}
 	\liminf_{k \rightarrow \infty} \frac{k}{\sum_{n= 1}^{k}\mathcal{B}_n } \geq \lim_{k \rightarrow \infty} \frac{k}{\sum_{n = 1}^{k}\tilde{\mathcal{B}}^{\bar{K}}_n } =\frac{1}{\mathbb{E}[ \mathcal{B}^{\bar{K}}_2]}>0.
\end{align*} 
Similarly, comparing with the dispatcher using static threshold policy $\bar{K}-1$ that is coupled with the genie-aided dispatcher described in Proposition \ref{prop:AnotherGenieOpt}, with probability 1, we have:
\begin{align*}
		\limsup_{k \rightarrow \infty} \frac{k}{\sum_{n = 1}^{k}\mathcal{B}_n} \leq \lim_{k\rightarrow \infty} \frac{k}{\sum_{n = 1}^{k}\tilde{\mathcal{B}}^{\bar{K}-1}_n } = \frac{1}{\mathbb{E}[ \mathcal{B}^{\bar{K}-1}_2]} <\infty.
\end{align*}
The last two results imply \eqref{limitbounded}. Then, \eqref{limitbounded} and \eqref{limitisO} prove the desired result.
	\end{proof}

 \begin{remark}\label{rem:optimal}
  \Added{When there exists a unique optimal threshold policy, the definition of regret is straightforward and without any ambiguity. However, in the case where there are multiple optimal threshold policies, we need to define the regret with respect to one of the optimal policies. Proposition~\ref{prop:AnotherGenieOpt} shows that the alternating genie-aided system is asymptotically optimal for almost all sample paths in the sense that it achieves the same long-term average profit as the system that uses either static threshold $\bar{K}$ or $\bar{K}-1$ starting from the beginning. The {\rm total} net profit achieved by this alternating genie-aided system up to time $T$ is not necessarily equal to the total net profit achieved by the genie-aided system using static threshold $\bar{K}$ or $\bar{K}-1$. These three policies (including the two static policies) do not necessarily achieve the same net profit up to time $T$ on given sample paths of the arrival and service processes.  Note that by Propositoin~\ref{prop:OrderedSystems}, the net profit process of the alternating genie-aided system during any busy cycle is either the same as the gain of one of the systems using static thresholds $\bar{K}$ and $\bar{K}-1$ or the net profit during the busy cycle is no smaller than the gain in the system using the static threshold $\bar{K}$: consider the case that the alternating system switches from using threshold $\bar{K}-1$ to $\bar{K}$, and the queue-length hits $\bar{K}$ during the current busy cycle. This is the only case where the behavior of the alternating genie-aided system may be different from the two systems using a static threshold. However, during the time between the switch and the time that the queue length of the alternating system hits $\bar{K}$ in the current busy cycle, the queue-length of the system using threshold $\bar{K}$ is greater than or equal to the queue-length of the alternating system. Moreover, the number of customers being served is the same for these two systems (in the current busy cycle). A similar but opposite comparison can be made with the system using static threshold $\bar{K}-1$. In fact, the total net profit  achieved (as a function of time) by the two systems using the static thresholds $\bar{K}$ and $\bar{K}-1$, respectively, are not necessarily equal on given  sample paths of the arrival and service processes either. We expect that the difference between the net profit of pairs of such systems obeys a Central-Limit Theorem behavior (including a functional form of the Central-Limit Theorem) when appropriately normalized and scaled (in time). 

Take as a concrete example the situation where $\bar{K} = 1$ and $\bar{K}-1 = 0$ are both optimal thresholds and assume that the initial queue length is $0$ for both systems. Using the inequalities in  \ref{eq:ThresholdIneq}, we get that these two optimal thresholds only occur when $C/\mu = R$. The system that uses the static threshold $0$ does not admit any customers into the system, and clearly achieves a total net profit equal to $0$ for any time $T$. The system that uses the static threshold $1$ admits a customer in the queue if and only if the system is empty when this customer arrives. The busy periods of this system using the static threshold $1$ are exactly the periods when a single customer is served, and the expected net profit during any busy period of this system is $R - C/\mu = 0$. However, this does not imply that the total net profit up to time $T$ of the system using threshold $1$ is $0$. In fact, the difference of the total net profit between these two systems over the busy periods of the system using threshold $1$ is a sum of mean-zero random variables (with each random variable being $R-C\times S$ where $S\sim \mathrm{EXP}(\mu)$ is the service time of the customer-in-service), which, intuitively, will lead to the claimed Central-Limit Theorem behavior. Furthermore, by the (finite-time) Law of the Iterated Logarithm \cite{balsubramani2015sharp}, along (almost all) sample paths the difference of the total net profit of the two systems may grow at most as  $O(\sqrt{T\ln(\ln(T))})$ (with high probability).

 For this example, we can also carry out an explicit analysis of $\mathbb{E}[\mathcal{G}(t)]$, the expected total net profit  up to any time $t$ of the system using static threshold 1. With the assumption that the initial queue length is 0, it is easier to consider the busy cycle as the idle period together with the consecutive busy period. Let $(Y^1_n, \mathcal{B}^{1}_{n})$ denote the total net profit and the duration of the $n^{th}$ busy cycle of the dispatcher that uses threshold $1$.  %Since the behavior of the queuing systems with initial queue length $0$ and using thresholds $0$ and $1$ starts to differ at the arrival of the first customer,  we assume the first customer arrives at $t = 0$. 
 As mentioned in the previous paragraph, $\mathbb{E}[Y^1_n] = 0$ for all $n$. The random variables $\mathcal{B}^1_n$ are \textit{i.i.d.} and have the same distribution as $A + S$, where  $A$ is an $\mathrm{EXP}(\lambda)$ random variable and $S$ is an $\mathrm{EXP}(\mu)$ random variable  independent of $A$. Let $N(t)$ denote the number of completed busy cycles until time $t$, $n(t) = \mathbb{E}\left[N(t)\right]$ denote the expected number of completed busy cycles up to time $t$, $\sigma_s(t)$ denote the residual service time of the current busy cycle at time $t$, and $\tau_{t} = \sum_{n = 1}^{N(t) +1} \mathcal{B}^1_n$ denote the end-time of the current busy cycle. Recalling that the reward $R$ is given to the dispatcher at each service completion, we have:
 \begin{align*}
     \mathbb{E}\left[\mathcal{G}(t)\right] = \mathbb{E}\left[\mathcal{G}(\tau_t)\right]   - R + C \mathbb{E}[\sigma_s(t)].
 \end{align*}
 Note that $n(t)$ is the renewal function of the associated (alternating) renewal process with renewal interval distributed the same as $A + S$. 
%  Let $F_{\mathcal{B}^1}(t)$ denote the CDF of ${\mathcal{B}^1_1}$. Using Laplace transform on the renewal function 
%  \begin{align*}
%      n(t) = F_{\mathcal{B}^1}(t) + \int_{0}^{t}n(t-s)dF_{\mathcal{B}^1}(s),
%  \end{align*}
%  we get: 
% \begin{align*}
%     L_n(s) = \frac{L_{\mathcal{B}^1}(s)}{s(1-L_{\mathcal{B}^1}(s))}.
% \end{align*}
% Substituting the Laplace transform of $L_{\mathcal{B}^1}(s)$, we get:
% \begin{align*}
%     L_n(s) = 
%       \frac{\frac{\lambda\mu}{(s+\lambda)(s+\mu)}}{s\left(1-\frac{\lambda\mu}{(s+\lambda)(s+\mu)} \right)}.
% \end{align*}
% Finding the inverse Laplace transform of $L_n(s)$, we get:
% \begin{align*}
%     n(t) = \lambda\mu\left(\frac{t}{\lambda +\mu} + \frac{e^{-(\lambda +\mu)t}-1}{(\lambda +\mu)^2 }\right).
% \end{align*}
By standard renewal theory arguments, $n(t)$ is finite for all t, and $N(t)+1$ is a stopping time of the sequence $(Y^1_n, \mathbb{B}^1_n)$. Applying Wald's equality, we get 
\begin{align*}
    \mathbb{E}[\mathcal{G}(\tau_t)] = \mathbb{E}\left[\sum_{i = 1}^{N(t)+1} Y^1_i \right] = \mathbb{E}\left[N(t)+1\right]\mathbb{E}\left[Y^1_1\right] = 0.
\end{align*}
Note that the distribution of $\sigma_s(t)$ follows $\mathrm{EXP}(\mu)$: if at time t the busy period has not started yet, clearly the residual service time is an $\mathrm{EXP}(\mu)$ random variable. If there is a customer being served at time t, the busy cycle ends at the completion of this service. Using the memory-less property of exponential random variable, the residual service time is again an $\mathrm{EXP}(\mu)$ random variable. Then, using $ \mathbb{E}[\mathcal{G}(\tau_t)]=0$, we get:
\begin{align*}
    \mathbb{E}\left[\mathcal{G}(t)\right] = \mathbb{E}\left[\mathcal{G}(\tau_t)\right]   - R + C \mathbb{E}[\sigma_s(t)]
     = 0 - R + C/\mu = 0.
\end{align*}

Despite admitting a customer when the queue is empty, the expected net profit at any time is exactly $0$ for the dispatcher using static threshold $1$ when both $\bar{K}=1$ and $\bar{K}-1=0$ are optimal thresholds. We expect that a similar but more complicated computation using renewal theory (as the memory-less argument no longer holds for the busy period, which is now a phase-type distribution, plus we need to determine the remaining workload to be served) can be carried out for systems using threshold $\bar{K} >1$ and $\bar{K}-1>0$, when both are optimal thresholds. We expect that as $t \rightarrow \infty$, the expected total net profit of the two systems using static thresholds differ by at most a constant, and so is the difference of the expected total net profit of the alternating system and the two systems using a static threshold. These questions are outside the scope of the paper and are left for future research. 
}
\end{remark}

	\subsection{Regret analysis with respect to the alternating  genie-aided dispatcher. }\label{subsection:IntCaseRegret}
	In Proposition \ref{prop:AnotherGenieOpt} we proved that the alternating genie-aided dispatcher described in Section~\ref{subsection:IntCaseAnotherGenie} that uses $\bar{K}$ and $\bar{K}-1$ ``in favor" of the learning algorithm is  optimal for \eqref{def:Long-termAverageProfit}. Next, we bound the regret of the learning dispatcher when compared with this genie-aided dispatcher. 
	
	Recall from Section~\ref{subsection:IntCaseAnotherGenie} that $\tilde{K}_i$ denotes the threshold used by the alternating genie-aided dispatcher  at the arrival of the $i^{th}$ arriving customer. 
    % \VS{This definition might need to go the previous subsection.}\YZ{$\tilde{K}_i$ is defined at previous section.}
    Following \eqref{eq:RegretUpperBound3}, we have:
	\begin{align*}
		G(t) \leq \bigg(R+\frac{C}{\lambda}\bigg)\mathbb{E}\left[ \sum\limits_{i = 1}^{N_A(t)} \left\rvert \mathbbm{1}_{\{\tilde{Q}_i<\tilde{K}_i\}} -  \mathbbm{1}_{\{Q_i<K_i\}} \right\rvert + \rvert \tilde{Q}_i   - Q_i \rvert\right] .
	\end{align*} 
	\Added{
Similar to the earlier analysis, assuming that both systems start with the same initial queue-length, we use $\tilde{G}_1^j$ to denote the expected regret accumulated during the (potential) phase $1$ and the first time the queue is emptied in the consecutive phase $2$ for the $j^{th}$ batch. Again, we use $\tilde{G}_2^j$ to denote the expected regret accumulated in the remainder of (the phase $2$ of the) $j^{th}$ batch.
 
Set $\tilde{\mathcal{E}}^j_2:=\{K(j) = \bar{K}\}\cup\{K(j) = \bar{K}-1\}$. We will reuse the events $\mathcal{E}^j_1$ and $\mathcal{E}^j_3$ that were first introduced in Section \ref{section:UniqueThresh}. Recall that $\mathcal{E}^j_1$ denotes the event that phase $1$ of the $j^{th}$ batch happens, and $\mathcal{E}^j_3=\{ Q_{n^j} = \tilde{Q}_{n^j}\}$ denotes the event that at the beginning of the $j^{th}$ phase $2$ of the learning system, the queue-length of the two systems are the same.

 Only under the event $\mathcal{E}_1^j$ there is a regret contribution to $\tilde{G}^j_1$ (since otherwise phase $1$ of the $j^{th}$ batch is omitted, and the queue-length at the beginning of phase 2  is 0). Under the event $(\mathcal{E}^j_1)^c\cap \tilde{\mathcal{E}}^j_2 \cap \mathcal{E}^j_3$, there is no regret contribution to $\tilde{G}^j_2$: indeed, for this batch of customers, $\tilde{\mathcal{E}}^j_2$ ensures the learned threshold is either $\bar{K}$ or $\bar{K}-1$. The event $(\mathcal{E}^j_1)^c$ ensures that phase $1$ is omitted, so the queue-length at the beginning of this phase $2$ of the learning system is $0$. Moreover, $\mathcal{E}^j_3$ ensures that the queue-length of the alternating genie-aided system is also $0$ at this time, which means that the arrival of the first customer of this phase $2$ initiates a busy cycle for both systems. In this case, the alternating genie-aided system would pick the same threshold used as the learning system for all the busy cycles in this phase $2$. Both systems would make the same choices of admitting each arrival in this phase $2$, and the queue-length processes of the two systems would also coincide for the entire phase $2$. Under the event $\mathcal{E}^j_1\cap \tilde{\mathcal{E}}^j_2 \cap \mathcal{E}^j_3$, although phase $1$ happens, Proposition \ref{prop:OrderedSystems} tells us that the queue-length of the learning system at the end of phase $1$ is no smaller than the queue-length of the genie-aided system. The event $\tilde{\mathcal{E}}^j_2$ ensures that the threshold used by the learning system during the entire phase $2$ is no smaller than the threshold used by the genie-aided system (since the genie-aided system would be either using the same threshold as the learning system when a busy cycle is initiated by a customer who arrives during phase $2$ or using threshold $\bar{K}-1$ when a busy cycle is initiated by a customer who arrives during phase $1$), when the queue-length of the learning system hits $0$ for the first time after phase $1$, the queue-length of the genie-aided system also hits $0$. The next proposition gives a bound that holds in the current setting for the probability of $\left(\tilde{\mathcal{E}}^j_2 \cap \mathcal{E}^j_3\right)^c$. 
}
\Added{
\begin{prop}
	Fix $j\geq \ceil{e^{\bar{K}}}$. In case that $V(\bar{K},\mu,\lambda) = R/C$, we have the following:
		\begin{align*}
				\mathbb{P}\left[ \left( \tilde{\mathcal{E}}^j_2 \cap \mathcal{E}^j_3\right)^c\right]
				&\leq \tilde{C}_1\exp\Big(-\tilde{C}_2{\ln^{1 + \epsilon}(j)}\Big)  +\tilde{C}_1\exp\Big(-\tilde{C}_2\ln^{1 + \epsilon}(j-1)\Big)\\
				&\quad\quad+ \tilde{C}_3\exp\Big(-\tilde{C}_4\beta_j\Big)+\tilde{C}_3\exp\Big(-\tilde{C}_4\beta_{j-1}\Big)
				+ \big(c_{\bar{K}}\big)^{\alpha_{j-1}l_2}.
		\end{align*}
			$\tilde{C}_1$, $\tilde{C}_2$, $\tilde{C}_3$ and $\tilde{C}_4$ are defined in \eqref{def:ICc1c2} and \eqref{def:ICc3c4}, and $$c_{\bar{K}} :=1- \left(\frac{\mu}{\lambda+\mu}\right)^{\bar{K}}\in(0,1).$$
			
\end{prop}
\begin{proof}
	The proof for both cases $\bar{K}>1$ and $\bar{K} = 1$ follows the same logic as in the case $\bar{K}>0$ in Proposition \ref{prop:BadEventProb}. 
\end{proof}
}
\Added{
Since we are using $l_1$, $K^*(j)$ and $\bar{K}$ to bound the queue-length in the proof of Lemmas \ref{lem:Phase1Reg} and \ref{lem:Phase2Reg}, these two lemmas still hold when the optimal threshold is not unique. %\VS{Double-check with my comments.}\YZ{with the updated regret bounds this would hold still.} 
It should be now clear that Theorem \ref{thm:Thm1} and Theorem \ref{thm:Thm2} also hold when equality holds in \eqref{eq:ThresholdIneq}.  

% \magenta{Done} \VS{The argument for $U'$ should be checked carefully. See comments in unique threshold setting.}

}

	\section{Simulation-based numerical results} \label{section:Numerical}
	In this section, we demonstrate the performance of our proposed Algorithm \ref{alg:Alg3} using simulations. To compute the regret we compare our algorithm to the genie-aided system that has the knowledge of the arrival and service rates and uses the optimal strategy proposed by %Naor
    \cite{Naor}. For the simulations, we set the initial queue-length to be $0$ for both the genie-aided and learning systems. 
 % \VS{How does this differ from your $a$ in proofs of Section V?}\YZ{ This is the case that $a = 0$ if you are referring to the initial queue-length in the proof of Proposition \ref{prop:AnotherGenieOpt}.} 
 For all numerical experiments, unless specified otherwise,  we use the following set of parameters:  $l_2 = 10$, $C = R = 1$, \Added{$\mathbb{E}[B^j] = \ln(j)/j$, $\alpha_j = j$ where recall that $l_2$ is the minimum length of phase $2$, $C$ is the cost per unit time, $R$ is the reward granted to the dispatcher when each service completes, $B^j$ is the random variable which controls the probability of having phase 1 when the threshold used in the previous phase 2 is 0, and $\alpha_j$ is the rate at which the minimum length of phase $2$ increases. Note that, unless specified otherwise, we use $\epsilon = 1$ 
 in $\mathbb{E}[B^j] = \ln^{\epsilon}(j)/j$. }%\VS{$B^j$ is a Bernoulli random variable ...}\YZ{rephrased.}
 % \VS{$\epsilon$ is used for $\alpha_j$? Quickly remind the reader abotu where each algorithm parameter stands for.}\YZ{$\epsilon$ is used for $\alpha_j$. edited.}  
 We vary $\mu$ and $\lambda$ for different experiments, and explore zero and non-zero optimal threshold cases, as well as the cases where the optimal threshold is unique and when it is not unique. To show the pattern of the regret within a reasonable number of arriving customers, when the largest optimal threshold is $0$, we use $l_1 =1$ and when the largest optimal threshold is positive, we use $l_1 = 3$, where $l_1$ is the length of phase $1$ (when used), and stays unchanged for all batches. Our theoretical analysis holds for arbitrary choices of the constants $l_1\geq 1$. However, when $l_1$ is large and the service rate is small, it will take a long time for the queue to empty during phase $2$, and therefore, will require more arrivals to show the correct asymptotic behavior of the regret. 
 
 The finite-time performance of the simulated results agrees qualitatively with our upper bound: when an optimal strategy is to use threshold $0$, the learning system achieves an expected regret that grows in a sub-linear manner; and when all optimal strategies use a non-zero threshold, the learning system achieves an $O(1)$ expected regret. 
 % \VS{In simulations we should use the $\alpha_j$ from our results, but you can mention that other choices like linear seem to work too (can include plots as well). This can then be another open question - tighter characterization of regret in all regimes with a better understanding of allowed $\alpha_j$ sequences.}\YZ{all plots updated.}

\begin{figure}[htbp]
	\centering
	\begin{subfigure}{.48\textwidth}
		\centering
		\includegraphics[width=\linewidth]{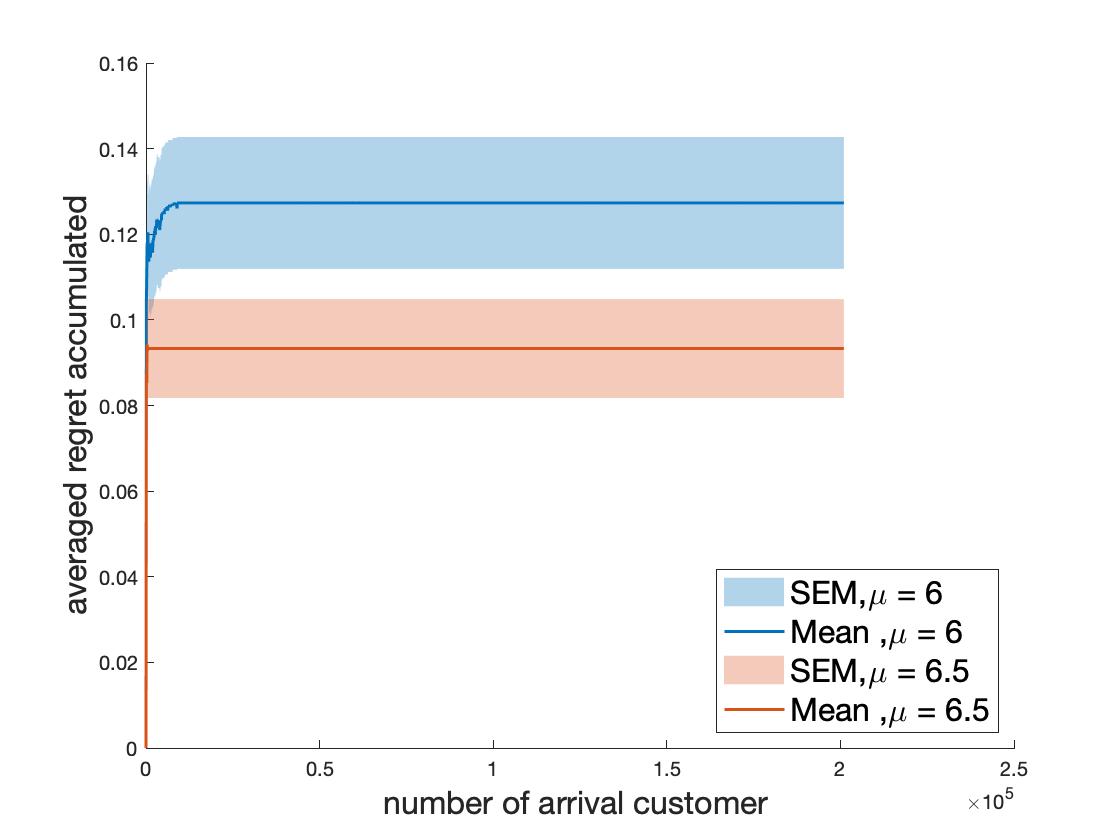}%@@@%{Fig1a.pdf}
		\caption{$\lambda = 1$, $R = 1$,  and the optimal threshold is $\bar{K} = 5$. }
		\label{fig1}
	\end{subfigure}%
	\begin{subfigure}{.48\textwidth}
		\centering
		\includegraphics[width=\linewidth]{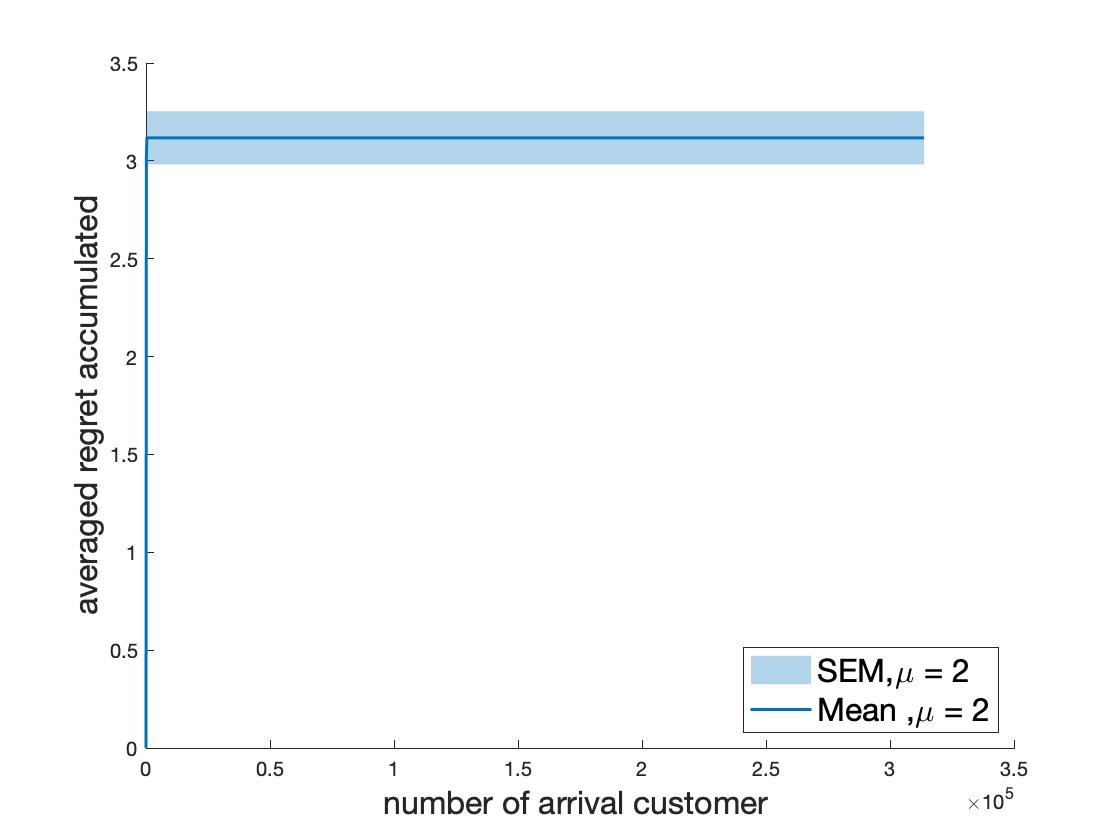}%@@@%{Fig1b.pdf}
		\caption{$\lambda = 1$, $R = \tfrac{129}{32}$, and the optimal thresholds $\{4,5\}$ ($\bar{K} = 5$). }
		\label{fig2}
	\end{subfigure}
	\caption{Regret of the learning system when all optimal thresholds are positive. We set $C = 1$, $\mathbb{E}[B^j] = \ln(j)/j$ ,  $K^*(j)\sim \ln(j)$, and $\alpha_j = j$. }
 %\VS{$B_j$ is a Bernoulli random variable, so at least write $\mathbb{E}[B_j]=\tfrac{\ln(j)}{j}$, if $\mathbb{P}(B_j=1)=\tfrac{\ln(j)}{j}$ is taking too much space. Please fix this everywhere, particularly in simulations section.}\YZ{fixed.}}%\VS{Other parameters?}\YZ{added.}}
	\label{fig_pos}
\end{figure}

\Added{\paragraph{\bfseries Expected regret with non-zero optimal thresholds:} Figure \ref{fig1} shows the variation of the (expected) regret with respect to the number of arrivals for  $\mu = 6$ and $\mu = 6.5$ when $l_1 = 3$ and $\lambda = 1$. The regret is averaged over 1000 simulations and there are more than $2*10^5$ customers arrivals to the system. The optimal threshold is unique, and the genie-aided dispatcher uses the threshold $\bar{K} = 5$ in both cases that are plotted in Figure \ref{fig1}. The initial upper bound is $K^*(1) = l_1$, which is smaller than the optimal threshold but increases slowly so that eventually $\bar{K}< K^*(j)$ for large $j$. As shown in the analysis and the numerical experiments, the regret is $O(1)$. }Figure \ref{fig2} shows the regret plot with respect to the number of arrivals for  $\mu = 2$, \Added{ $\lambda = 1$  and $R = 129/32$ with $l_1 = 3$. The regret is averaged over 2000 simulations and there are more than $2*10^5$ customers arrivals to the system.}  In this case, the optimal threshold is not unique: both $\bar{K}-1 = 4$ and $\bar{K} = 5$ are optimal thresholds. The alternating genie-aided algorithm uses the policy that is described in Proposition~\ref{prop:AnotherGenieOpt} and only changes the threshold used between busy cycles. %\Added{The learning dispatcher is allowed to use threshold 2 in phase 2 since the first batch. We are not able to solve the polynomial of order 7 explicitly and get the corresponding arrival and service rates that give a system that is optimal to use thresholds 5 and 4, so we used lower optimal thresholds in Figure~\ref{fig2}.}
Similarly, as in Figure~\ref{fig1}, the learning algorithm will not be able to use $\bar{K}$ in the first few batches because of the truncation. %\VS{Instead of clipping use truncation.}\YZ{edited.} 
The plots indicate that constant regret is accumulated, which is consistent with our analytical results; interestingly, in all cases, convergence to the constant regret value happens rapidly.

	\begin{figure}[htbp]
		\centering
		\begin{subfigure}{.48\textwidth}
			\centering
			\includegraphics[width=\linewidth]{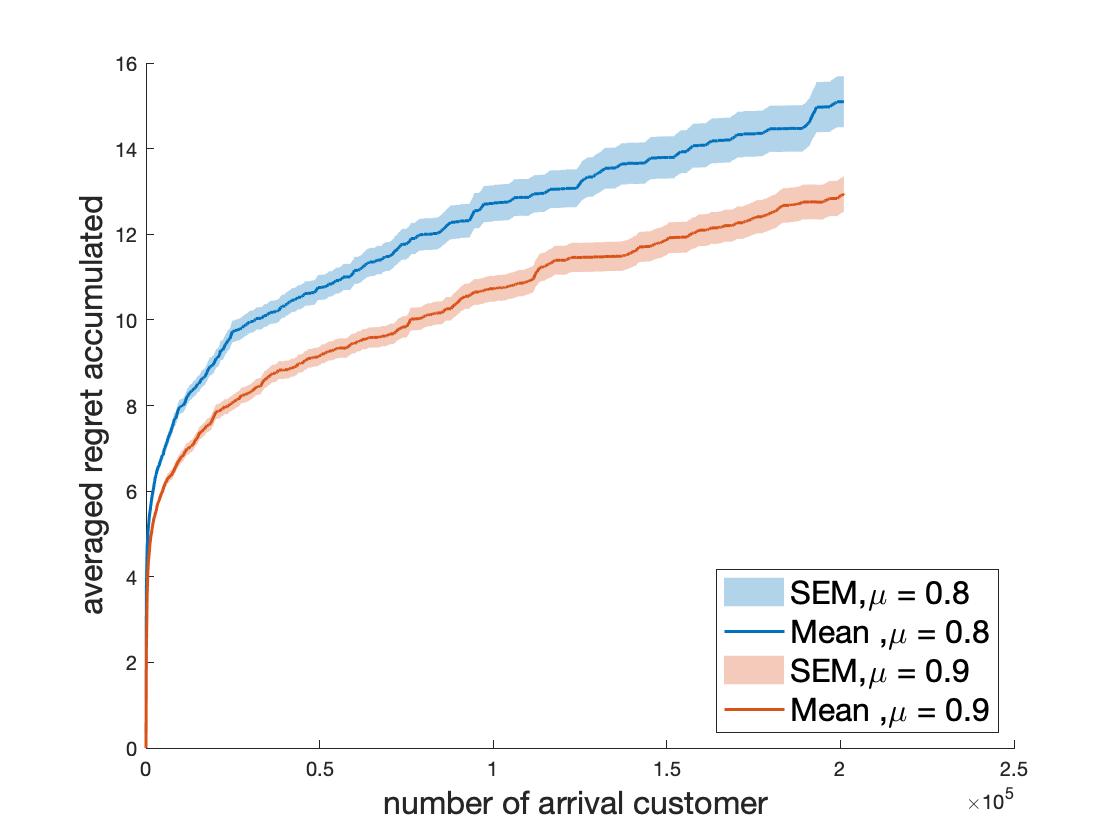}%@@@%{Fig2a.pdf}
			\caption{$\lambda= 1$ and the optimal threshold is $\bar{K} = 0$. }
			\label{fig3}
		\end{subfigure}%
		\begin{subfigure}{.48\textwidth}
			\centering
			\includegraphics[width=\linewidth]{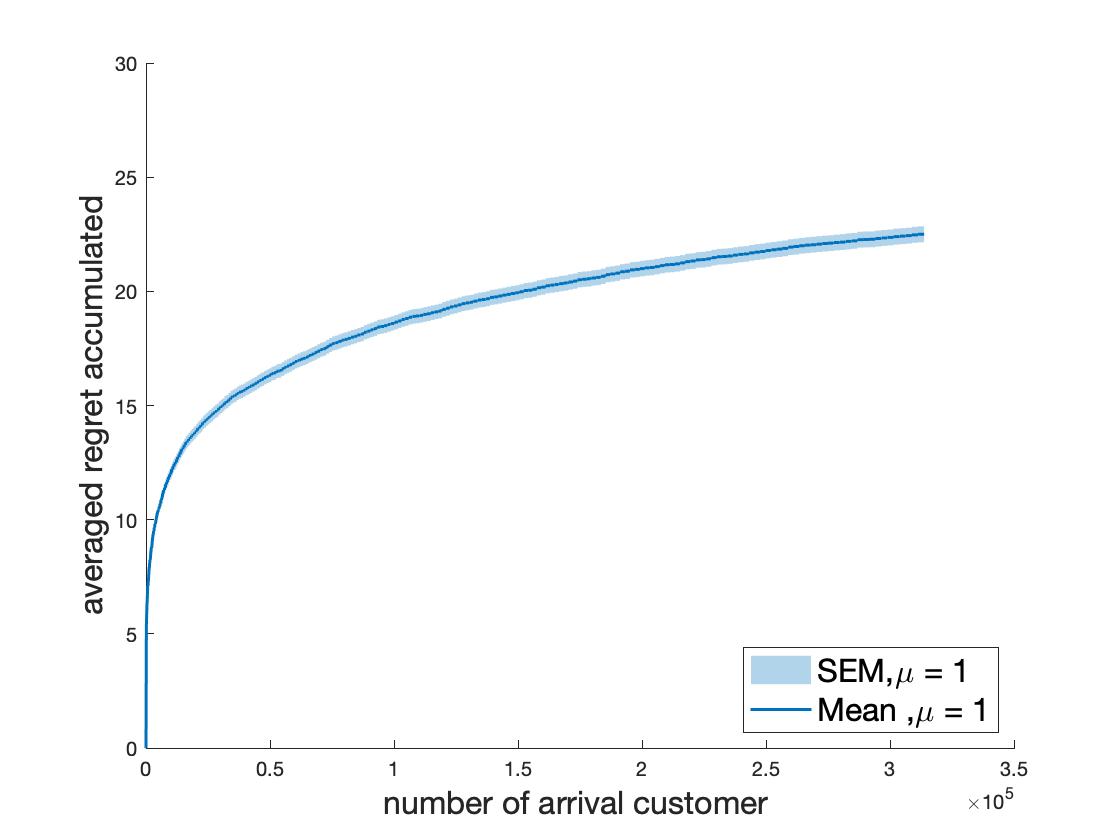}%@@@%{Fig2b.pdf}
			\caption{$\lambda = 1 $ and the optimal thresholds are $\{0,1\}$; $\bar{K} = 1$.}
			\label{fig4}
		\end{subfigure}
		\caption{Regret of the learning system when an optimal threshold is zero. We set $C = R = 1$, $\mathbb{E}[B^j] = \ln(j)/j$ , $K^*(j)\sim \ln(j)$ and $\alpha_j = j$.}%\VS{Other parameters?}\YZ{added.}}
		\label{fig_0}
	\end{figure}

	\Added{\paragraph{\bfseries Expected regret with zero being an optimal threshold:}Figure~\ref{fig3} shows how the regret changes with respect to the number of arrivals for $\mu = 0.8$ and $\mu = 0.9$ when $l_1 = 1$ and $\lambda = 1$. The regret is averaged over 2000 simulations and there are more than $10^5$ customers arrived in the system. In both cases shown in Figure~\ref{fig3}, the genie-aided dispatcher uses threshold $\bar{K} =0$.} Figure~\ref{fig4} shows the regret plot with respect to the number of customers for  $\mu =1$  \Added{and $\lambda = 1$ }when $l_1 = 3$. The regret is averaged over 2000 simulations and there are more than $2*10^5$ customers arrived in the system. In this case, the optimal threshold is not unique: both $\bar{K}-1 = 0$ and $\bar{K} =1$ are optimal thresholds. The alternating genie-aided dispatcher uses the policy that is described in Proposition~\ref{prop:AnotherGenieOpt} and only changes the threshold between busy cycles. The plots indicate that sub-linear regret is accumulated in all cases. Here, when the learning dispatcher uses threshold $0$ in phase $2$ of a given batch, the existence of the forced exploration phase in the next batch results in regret being accumulated. \Added{Note that for all plots shown in Figure~\ref{fig_0}, the optimal thresholds can be used by the learning dispatcher in phase $2$ right from the first batch.}

\begin{figure}[htbp]
		\centering
		\begin{subfigure}{.48\textwidth}
			\centering
			\includegraphics[width=\linewidth]{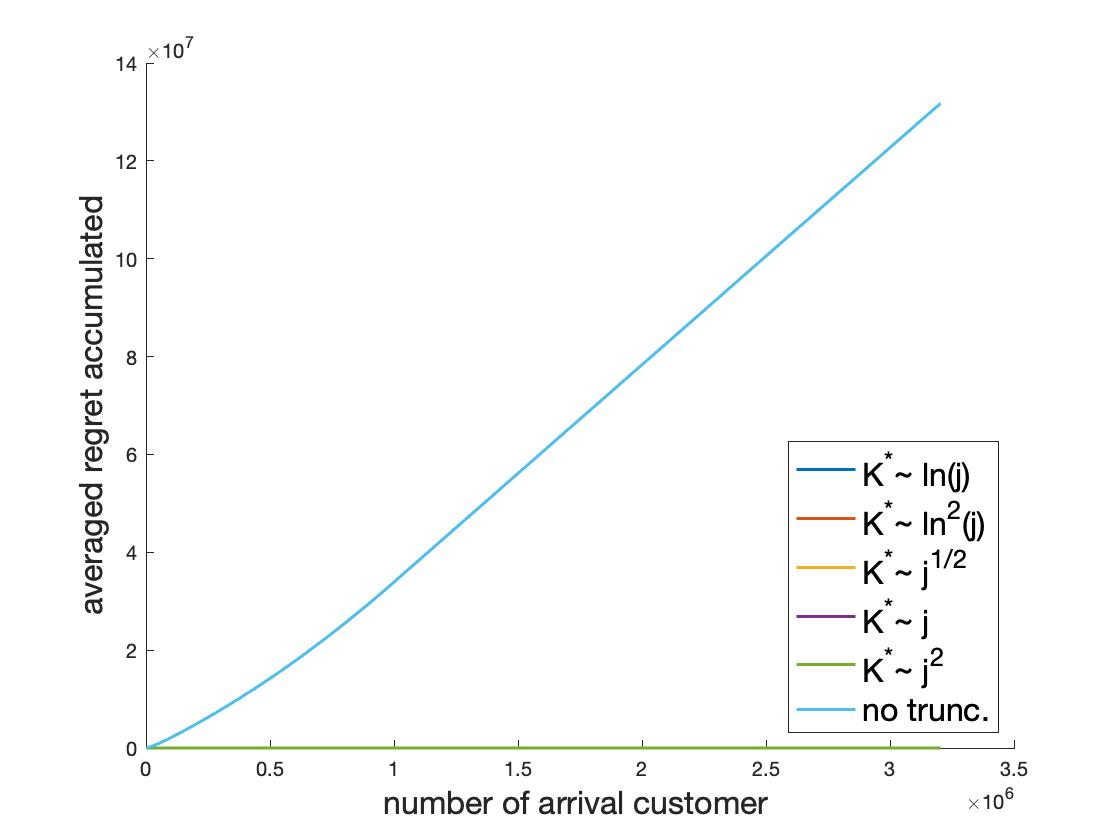}%@@@%{Fig3a.pdf}
			\caption{ With the no-truncation option included. }
			\label{fig5}
		\end{subfigure}%
		\begin{subfigure}{.48\textwidth}
			\centering
			\includegraphics[width=\linewidth]{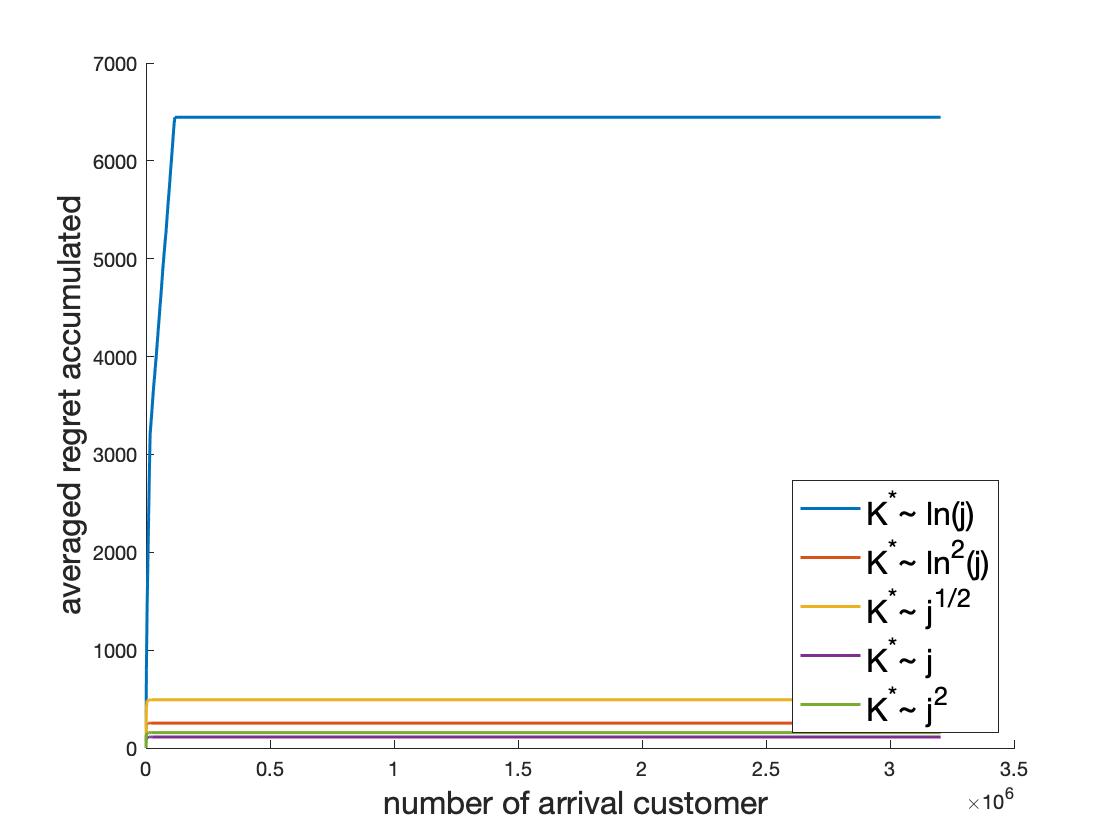}%@@@%{Fig3b.pdf}
			\caption{ Exclude the no-truncation option.}
			\label{fig6}
		\end{subfigure}
		\caption{Regret of the learning system when $\mu = 3$, $\lambda = 3.5$, $R = 21$ and the optimal threshold is $8$ using and not using the truncation for the threshold used in phase $2$. We set $C  = 1$, $\mathbb{E}[B^j] = \ln(j)/j$  and $\alpha_j = j$.}%\VS{Other parameters?}\YZ{added.}}
		\label{fig_difK2}
	\end{figure}

\Added{\paragraph{\bfseries Expected regret with different choices of $K^*(j)$:}
    We introduced truncation with the parameter $K^*(j)$ in our analysis since we needed a bound on the worst-case queue length for the learning system. We obtained a particular order of the regret with the choice of $K^*(j) = \max\{ \floor{\ln(j)},0 \}+ l_1+ Q_0$. Next, we explore the impact of different choices of $K^*(j)$ in %Figure~\ref{fig_difK} and 
    Figure~\ref{fig_difK2}. We use $\sim$ to indicate the order at which $K^*(j)$ increases: specifically, $K^*(j) \sim f(j)$ means $K^*(j) =  \max\{ \floor{f(j)},0 \}+ l_1+ Q_0$. %In Figure~\ref{fig_difK}, we use $\mu = 9$, $\lambda = 1$ and the optimal threshold is $\bar{K} = 8$.  We also set $l_1 = 3$, and $Q_0 = 0$. This ensures that when $K^*(j) \sim \ln(j)$, starting at the $149^{\mathrm{th}}$ batch, the learning system is allowed to use threshold $8$. %\VS{Figure~\ref{fig_difK} says threshold is $8$! Your legends are generally unreadable as they're too small. The lines should be thicker as well.}\YZ{8 is the correct threshold. edited.} 
    The regret values are averaged over $2000$ simulations, and there are more than $3*10^5$ arrival customers that arrive in more than $700$ batches. 
    In Figure~\ref{fig_difK2}, we use  $\mu = 3$, $\lambda = 3.5$, and $R = 21$. The optimal threshold is  $\bar{K} = 8$. %Unlike in Figure~\ref{fig_difK}, 
    The $M/M/1$ queue with $\mu = 3$ and  $\lambda = 3.5$ is not stable.  Despite this, Figure~\ref{fig6} suggests that constant regret is achieved for various truncation choices. However, when no truncation is enforced, the regret accumulated seems to grow linearly with respect to the number of arrivals, see Figure~\ref{fig5}. This suggests that the truncation helps to ensure a lower regret yet one may use a $K^*(j)$ that grows faster than $\ln(j)$. Confirming this through analysis is a topic to explore in future research. %For both Figure~\ref{fig_difK} and Figure~\ref{fig_difK2}, there are more than $3*10^5$ arrival customers that arrive in more than $700$ batches.
}

    %\begin{figure}[htbp]
%		\centering
%		\begin{subfigure}{.48\textwidth}
%			\centering
%			\includegraphics[width=\linewidth]{compare_thresh5_1_erb.jpg}
%			\caption{$\lambda= 1$, $\mu = 6$ and  the optimal threshold is $\bar{K} = 5$. }
%			\label{fig5}
%		\end{subfigure}%
%		\begin{subfigure}{.48\textwidth}
%			\centering
			%\includegraphics[width=\linewidth]{compare_thresh5_1_erb2.jpg}
		%	\caption{$\lambda= 1$, $\mu = 6.5$ and  the optimal threshold is $\bar{K} = 5$.}
		%	\label{fig6}
		%\end{subfigure}
		%\caption{Regret of the learning system when the optimal threshold is 5 using and not using the clipping for the threshold used in phase 2.}
		%\label{compare_fig_th5}
	%\end{figure}
    %\Added{The parameter choices and number of simulations run for Figure~\ref{compare_fig_th5} are the same as Figure~\ref{fig1}. Since the optimal threshold is 5 which is larger than $K^*(1)$, the learning algorithm with the clipping accumulates higher regret compared to one without clipping. %However, as shown in Figurein the case that the optimal threshold is lower and in particularly smaller than $K^*(1)$, the existence of the clipping may help to achieve a lower regret as shown on Figure~\ref{fig6}. 
    %}
    
    %\begin{figure}[htbp]
	%	\centering
    %\includegraphics[width=0.5\linewidth]{t0multi_loglog_serv.jpg}
		%\caption{Log versus log-log plot of expected regret with respect to the number of arrivals when $\lambda = 1$ and $\bar{K} = 0$. We set $C = R = 1$, $\mathbb{E}[B^j] = \ln(j)/j$ , $K^*(j)\sim \ln(j)$ and $\alpha_j = j$.}%\VS{Other parameters?}\YZ{added.}}
		%\label{fig_t0_serv}
	%\end{figure}

\Added{
  }

  \begin{figure}[htbp]
        \centering
		%\begin{subfigure}{.48\textwidth}
	%		\centering
			%\includegraphics[width=\linewidth]{dif_alp_2ksim.jpg}
		%	\caption{Average regret plot.}
	%		\label{fig7}
		%\end{subfigure}%
		\begin{subfigure}{.48\textwidth}
			\centering
			\includegraphics[width=\linewidth]{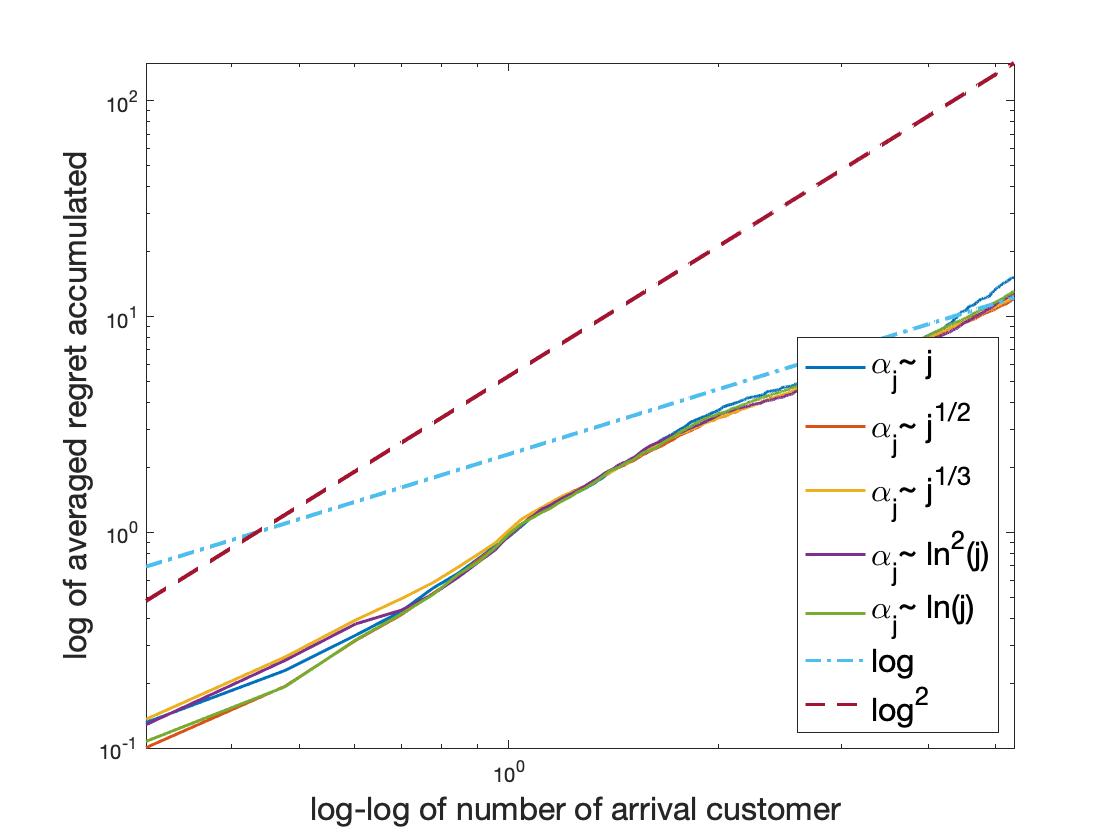}%@@@%{Fig4a.pdf}
			\caption{ Log versus log-log  regret plot on regret accumulated when $\mu = 0.8$, $\lambda = 1$ and $R = 1$. Optimal threshold is $\bar{K} = 0$.}
			\label{fig8}
		\end{subfigure}
  \begin{subfigure}{.48\textwidth}
			\centering
			\includegraphics[width=\linewidth]{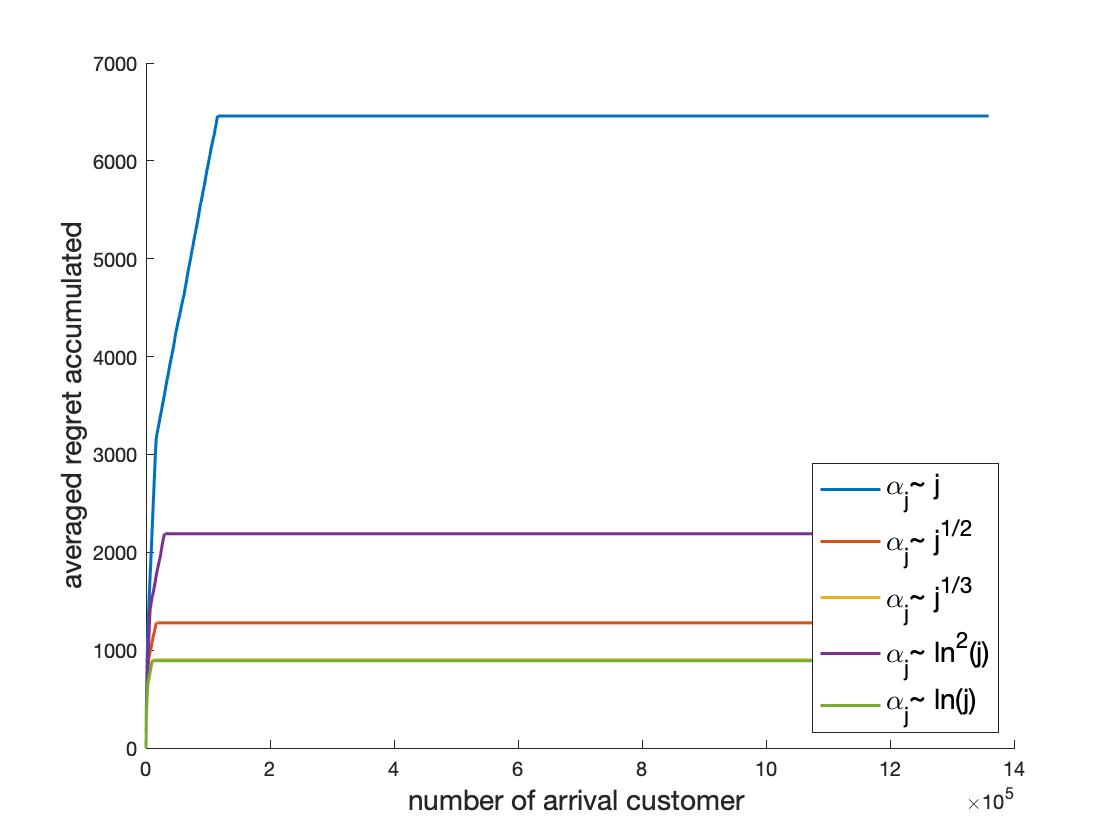}%@@@%{Fig4b.pdf}
			\caption{Regret accumulated when $\mu = 3$, $\lambda = 3.5$, and $R =21$. Optimal threshold is $\bar{K} = 8$. }
			\label{fig_t8_da1}
		\end{subfigure}
		\caption{Regret accumulated for different choice of $\alpha_j$. We set $C =  1$, $\mathbb{E}[B^j] = \ln(j)/j$ and  $K^*(j)\sim \ln(j)$.}%\VS{Other parameters?}\YZ{added.}}
		\label{fig_alpha}

	\end{figure}
 
%\begin{figure}[htbp]
   %     \centering
	%	\begin{subfigure}{.48\textwidth}
			%\centering
			%\includegraphics[width=\linewidth]{thresh8_difalph.jpg}
			%\caption{Average regret plot.}
			%\label{fig_t8_da1}
		%\end{subfigure}%
		%\begin{subfigure}{.48\textwidth}
			%\centering
			%\includegraphics[width=\linewidth]{thresh8_zoomed.jpg}
			%\caption{Average regret plot when $\alpha_j \sim j^{1/3}$ and $\alpha_j\sim \ln(j)$. }
			%\label{fig_t8_da2}
		%\end{subfigure}
		%\caption{Regret accumulated when $\mu = 3$, $\lambda = 3.5$, and $R =21$ for different choices of $\alpha_j$. Optimal threshold is $\bar{K} = 8$. We set $C = 1$, $\mathbb{E}[B^j] = \ln(j)/j$ and $K^*(j)\sim \ln(j)$.}%  \VS{What is $C$?}\YZ{added.}}
		%\label{fig_t8_alpha}

	%\end{figure}
 
\Added{ \paragraph{\bfseries Expected regret with different choices of $\alpha_j$:}
We introduced $\alpha_jl_2$  to be the minimum length of phase 2 for the $j^{th}$ batch. Figure~\ref{fig_alpha} plots the average regret accumulated with different choices of $\alpha_j$'s. In particular,
Figure~\ref{fig8} is the log versus log-log plot of the regret accumulated when $\mu = 0.8$, $\lambda = 1$ with more than $2*10^5$ arrival customers, and Figure~\ref{fig_t8_da1}  plots the regret accumulated when $\mu = 3$, $\lambda = 3.5$ with more than $10*10^5$ arrival customers.  We use $\alpha_j \sim f(j)$ to denote $\alpha_j = \max\{\floor{f(j)},1\}$. The regret is averaged over 2000 simulations in both plots.  %Figure~\ref{fig8} is the log versus log-log plot of the regret accumulated with respect to the number of arrivals. 
 Figure~\ref{fig_alpha} suggests that for all these choices of $\alpha_j$, a sub-linear regret is accumulated,  %Together with Figure~\ref{fig_t0_serv}, these plots suggest that while the bound of $O(\ln^2(N))$ may not be a tight bound on the regret, the regret likely grows faster than $\omega(\ln(N))$. Determining the true scaling of the regret when an optimal threshold is $0$ is yet another future work direction. Moreover, 
%Figure~\ref{fig_alpha} and Figure~\ref{fig_t8_alpha} together  suggest that 
 and having an $\alpha_j$ that grows slower may still be able to achieve the regret bounds proved for $\alpha_j = j$. %\VS{Please fix this paragraph too based on current results.}\YZ{description should match current theoretical results now.}
}

\begin{figure}[htbp]
        \centering
		%\begin{subfigure}{.48\textwidth}
	%		\centering
			%\includegraphics[width=\linewidth]{Jun_23_dif_prob_th1.jpg}
		%	\caption{Average regret plot.}
	%		\label{fig13}
		%\end{subfigure}%
		\begin{subfigure}{.48\textwidth}
			\centering
			\includegraphics[width=\linewidth]{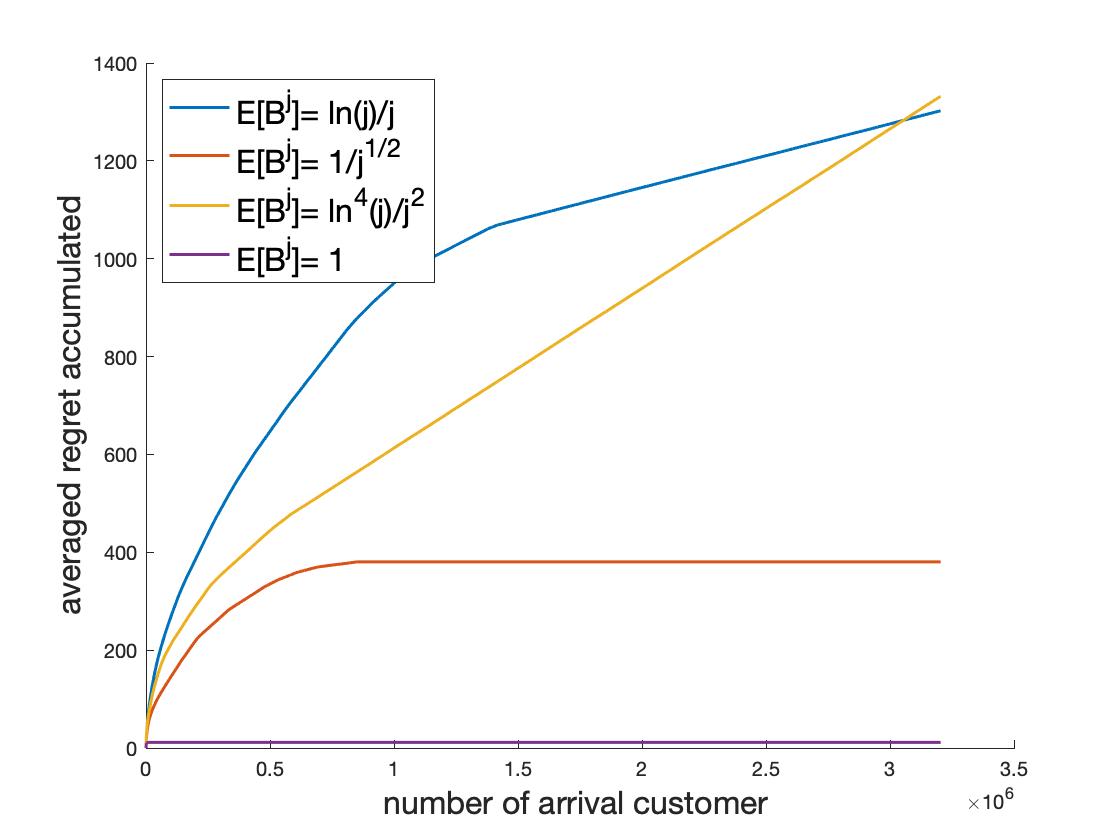}%@@@%{Fig5a.pdf}
			\caption{Average regret plot when $\mu = 1.3$, $\lambda = 1$. The optimal threshold is $\bar{K} = 1$.}
			\label{fig14}
		\end{subfigure}
    \begin{subfigure}{.48\textwidth}
			\centering
			\includegraphics[width=\linewidth]{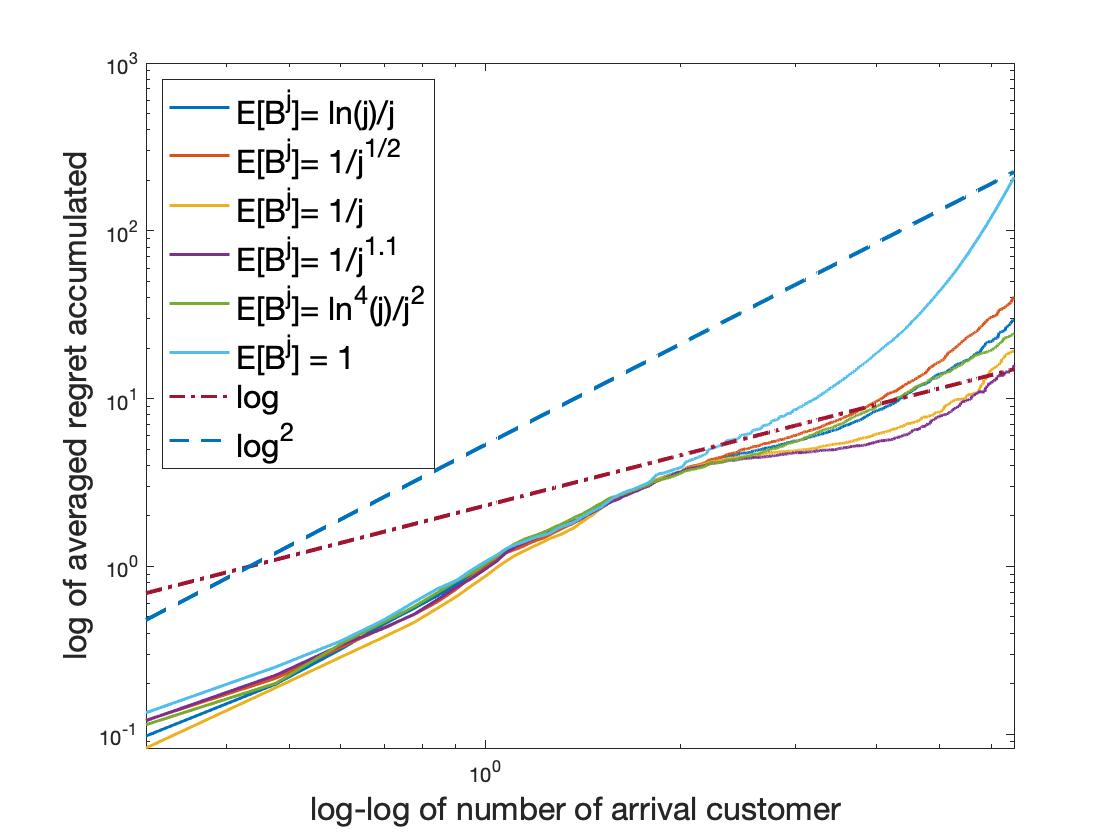}%@@@%{Fig5b.pdf}
			\caption{ Log versus log-log  regret plot when $\mu = 0.8$, $\lambda = 1$. The optimal threshold is $\bar{K} =0$.}
			\label{fig12}
		\end{subfigure}
		\caption{Regret accumulated when the choices of $\mathbb{E}[B^j]$ vary. We set $C = R = 1$, $K^*(j)\sim \ln(j)$ and $\alpha_j = j$.}%  \VS{Other parameters?}\YZ{added.}}
		\label{fig_t1_difprob}

	\end{figure}

%\begin{figure}[htbp]
 %       \centering
%		\begin{subfigure}{.48\textwidth}
%			\centering
			%\includegraphics[width=\linewidth]{Jun_23_dif_prop_th0.jpg}
		%	\caption{Average regret plot. }
	%		\label{fig11}
		%\end{subfigure}%
		%\begin{subfigure}{.48\textwidth}
			%\centering
			%\includegraphics[width=\linewidth]{Jun_dif_prob_th0_log.jpg}
		%	\caption{ Log versus log-log  regret plot.}
		%	\label{fig12}
		%\end{subfigure}
		%\caption{Regret accumulated when $\mu = 0.8$, $\lambda = 1$. The optimal threshold is $\bar{K} =0$ and the choices of $\mathbb{E}[B^j]$ vary. We set $C = R = 1$, $K^*(j)\sim \ln(j)$ and $\alpha_j = j$.}%  \VS{Other parameters?}}
		%\label{fig_t0_difprob}

	%\end{figure}
\Added{ \paragraph{\bfseries Expected regret with different choices of $\mathbb{E}[B^j]$:} We also examined difference choices of $\mathbb{E}[B^j]$, which controls the probability of having a phase 1 when the threshold used in the previous phase 2 is 0. %Figure~\ref{fig_t0_difprob} and 
Figure~\ref{fig_t1_difprob} shows the plots of various choices of $\mathbb{E}[B^j]$. %when the optimal threshold is 1 or 0. 
%On the legends on the plots, we used $B^j \sim f(j)$ to denote $\mathbb{E}[B^j] = f(j)$. 
From these finite-time experiments, it seems that having a high enough chance to explore during the first few batches the learning dispatcher observes helps to reduce the regret accumulated. However, comparing the plots of $\mathbb{E}[B^j] = \ln^4 (j)/j^2$ and $\mathbb{E}[B^j] = \ln(j)/j$ in Figure~\ref{fig14}, it seems that only having a high probability of exploration for the first few batches is not be enough to achieve $O(1)$ regret since the slope of the plot for $\mathbb{E}[B^j] = \ln(j)/j $ decreases a lot faster than the plot of $\mathbb{E}[B^j] = \ln^4 (j)/j^2$. Although all the choices of $\mathbb{E}[B^j]$ seem to achieve sub-linear regret for the case $\bar{K} = 0$, always having the exploration phase when the threshold used in the previous phase 2 is 0 accumulates a higher regret with a different scaling behavior. %\VS{Please harmonize $B_j$ vs $B^j$. Plus say various choices of $\mathbb{E}[B^j]$ and not $B^j$s! Also, just fix the legend to have $\mathbb{E}[B^j]$ instead of $B^j$.}\YZ{now all with $B^j$, the legends are changed as well.} 
}

  \begin{figure}[htbp]
        \centering
		\begin{subfigure}{.48\textwidth}
			\centering
			\includegraphics[width=\linewidth]{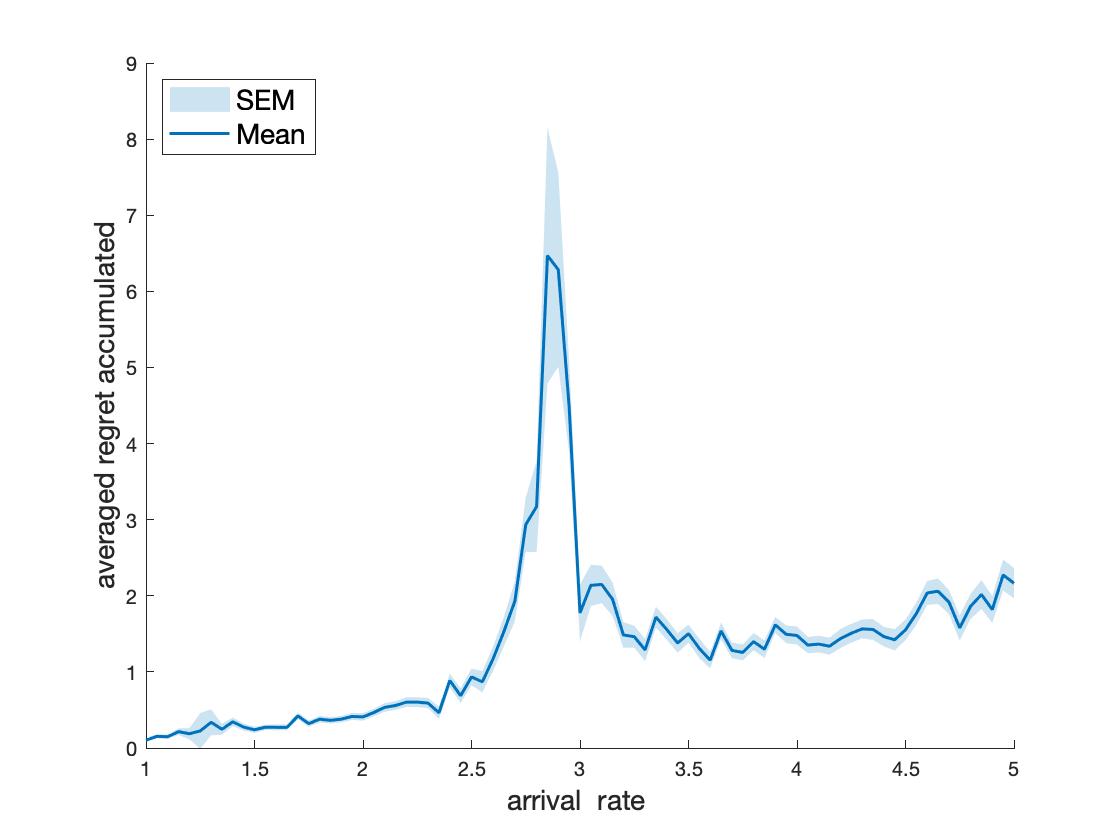}%@@@%{Fig6a.pdf}
			\caption{ Average regret plot versus various $\lambda$'s when $\mu = 6$.}
			\label{fig9}
		\end{subfigure}%
		\begin{subfigure}{.48\textwidth}
			\centering
			\includegraphics[width=\linewidth]{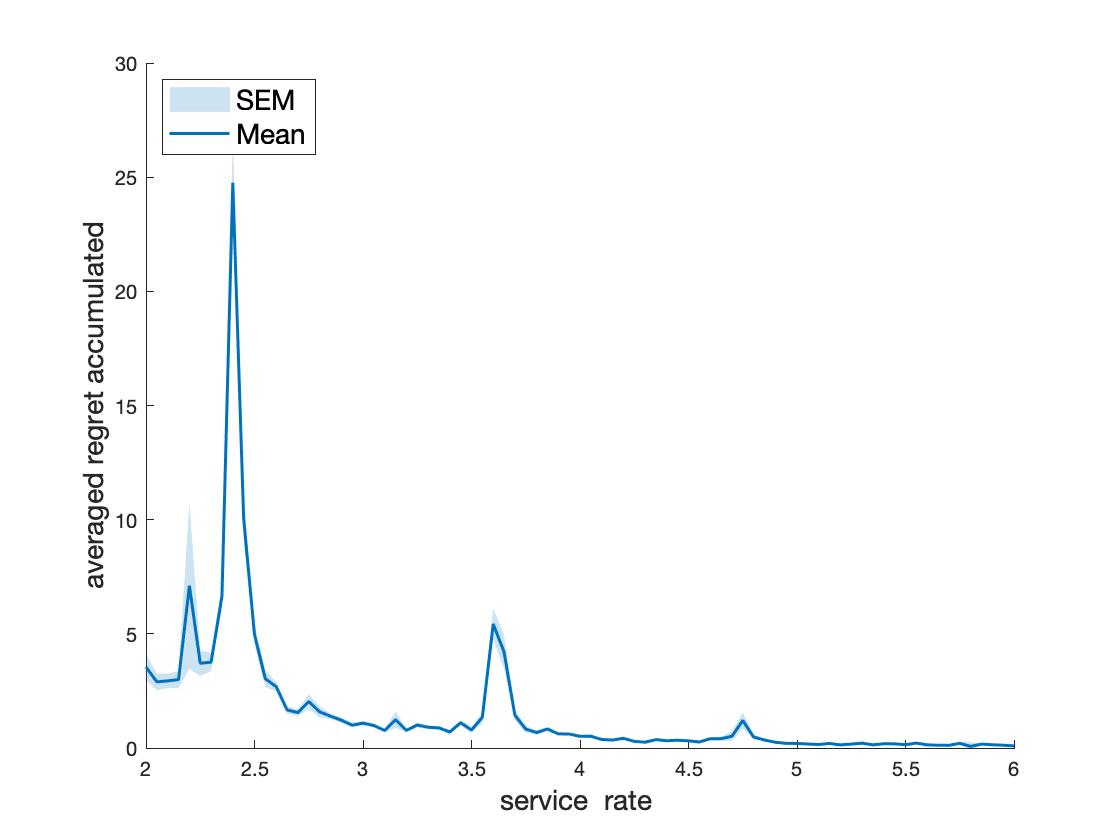}%@@@%{Fig6b.pdf}
			\caption{Average regret plot versus various $\mu$'s when $\lambda =1 $ }
			\label{fig10}
		\end{subfigure}
		\caption{Regret plot for various arrival and service rates. We set $C = R = 1$, $\mathbb{E}[B^j] = \ln(j)/j$, $K^*(j)\sim \ln(j)$ and $\alpha_j = j$.}% \VS{Other parameters?}}
		\label{fig_as}

	\end{figure}
 \Added{\paragraph{\bfseries Expected regret with different values of $\mu$ and $\lambda$:}
 Figure~\ref{fig_as} plots the  average regret accumulated when seeing more than $3*10^5$ arriving customers when fixing one of the pair of arrival and service rates while varying the other. The regret values are averaged over $600$ simulations. From the plot, we observe that when the arrival rate is fixed, as the service rate increases, in general, the regret decreases. However, the decrease is not strict and instead is non-monotonic, where the large cusps are usually around the parameter choices that have non-unique optimal thresholds. When the service rate is fixed, as the arrival rate increases, the regret follows a similar increasing/decreasing trend. 
}
\begin{figure}[htbp]
        \centering
		%\begin{subfigure}{.48\textwidth}
			%\centering
			%\includegraphics[width=\linewidth]{Jun_newUCB_t0compare.jpg}
			%\caption{Average regret plot. }
			%\label{fig_difAlg_t0_1}
		%\end{subfigure}%
		\begin{subfigure}{.48\textwidth}
			\centering
			\includegraphics[width=\linewidth]{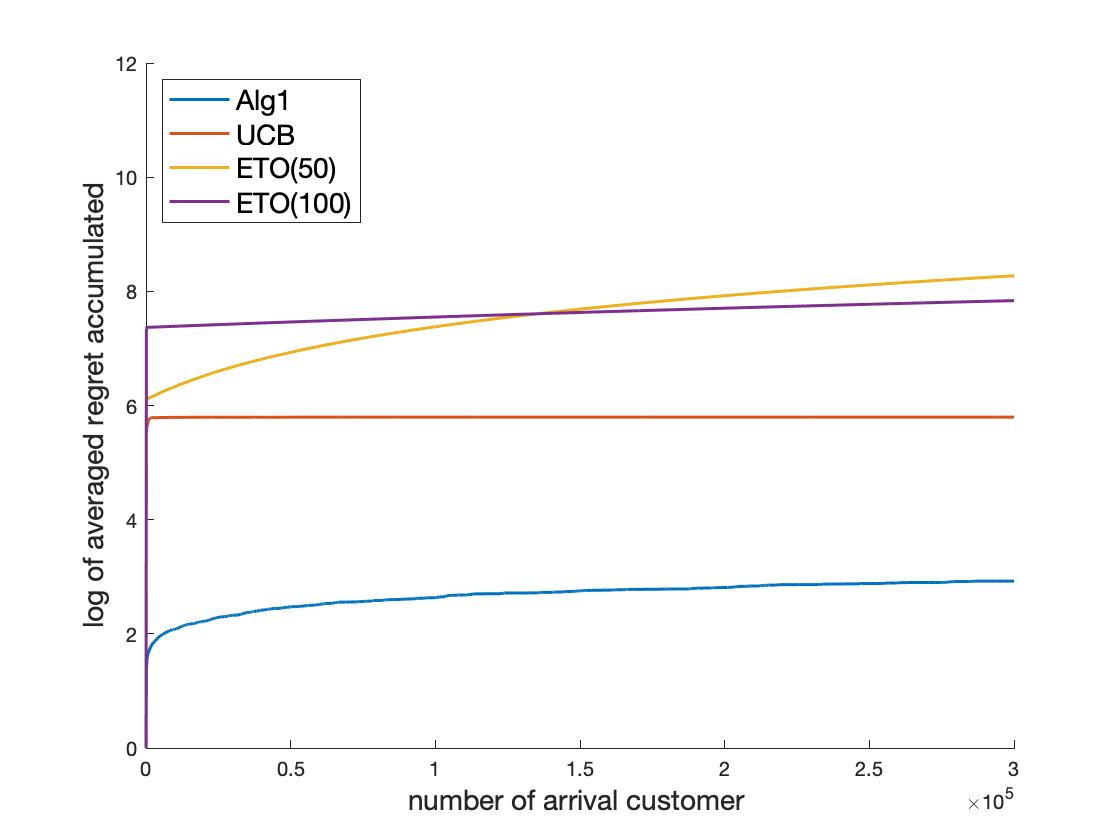}%@@@%{Fig7a.pdf}
			\caption{Log average regret plot when $\mu = 0.8$, $\lambda = 1$, $R = 1$ and $\bar{K} = 0$.}
			\label{fig_difAlg_t0_2}
		\end{subfigure}
  \begin{subfigure}{.48\textwidth}
		\centering
			\includegraphics[width=\linewidth]{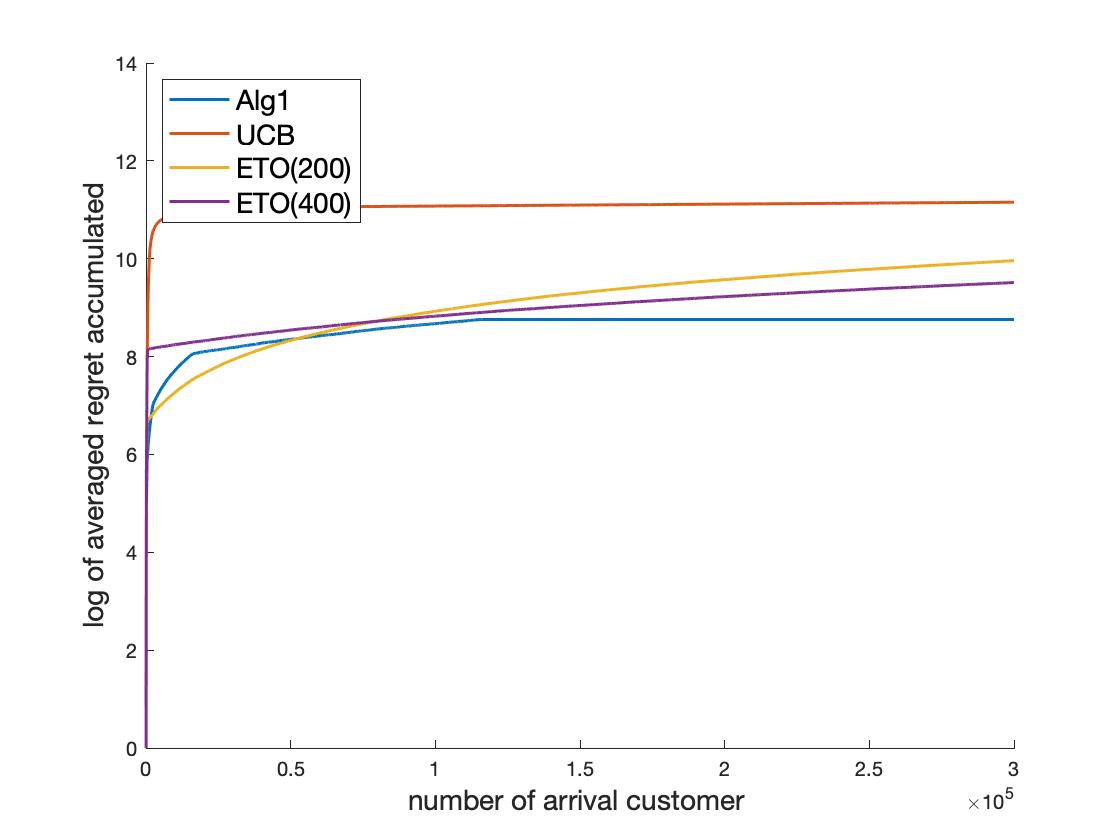}%@@@%{Fig7b.pdf}
			\caption{Log average regret plot when $\mu = 3$, $\lambda = 3.5$, $R = 21$ and $\bar{K} = 8$. }
			\label{fig_difAlg_t8_2}
		\end{subfigure}
		\caption{Log of regret accumulated when using different algorithms when the optimal threshold is unique. Alg1 is the learning algorithm proposed in Algorithm~\ref{alg:Alg3}. We set $C  = 1$, $\mathbb{E}[B^j] = \ln(j)/j$ , $K^*(j)\sim \ln(j)$ and $\alpha_j = j$. ETO($M$) is the Estimate-Then-Optimize algorithm that always accepts the first $M$ customers. UCB is the  Upper Confidence Bound algorithm.%\VS{What are $R$ and $C$? Other parameters? Also try ETO with different number of initial customers accepted. Also try ETO with different number of initial customers accepted.}\YZ{added, plots have been updated with error bar.}
        }
		\label{fig_difAlg_1}

	\end{figure}

 %\begin{figure}[htbp]
%		\centering
%		\begin{subfigure}{.48\textwidth}
%			\centering
			%\includegraphics[width=\linewidth]{Jun_newUCB_t8compare.jpg}
			%\caption{Average regret plot. }
			%\label{fig_difAlg_t8_1}
		%\end{subfigure}%
		%\begin{subfigure}{.48\textwidth}
	%		\centering
			%\includegraphics[width=\linewidth]{Jun_newUCB_t8compare_les.jpg}
			%\caption{Average regret plot comparing Alg1 with UCB. }
			%\label{fig_difAlg_t8_2}
		%\end{subfigure}
		%\caption{Regret accumulated when $\mu = 3$, $\lambda = 3.5$, $R = 21$ and $\bar{K} = 8$. Alg1 is the learning algorithm proposed in Algorithm~\ref{alg:Alg3}. We set $C = 1$, $\mathbb{E}[B^j] = \ln(j)/j$ , $K^*(j)\sim \ln(j)$ and $\alpha_j = j$.  ETO($M$) is the Estimate-Then-Optimize algorithm that always accepts the first $M$ customers. UCB is the  Upper Confidence Bound algorithm.%\VS{What is $C$? Other parameters? Can you run this longer to check whether UCB has finite regret?}\VS{Also, ETO with different number of initial customers for estimate.}\YZ{added, plots are updated with more arrivals.}
        %}
		%\label{fig_difAlg_2}
	%\end{figure}

 \begin{figure}[htbp]
        \centering
		%\begin{subfigure}{.48\textwidth}
	%		\centering
			%\includegraphics[width=\linewidth]{Jun_NU_t0_newUCB_comp.jpg}
			%\caption{Average regret plot. }
			%\label{fig_difAlg_NUt0_1}
		%\end{subfigure}%
		\begin{subfigure}{.48\textwidth}
			\centering
			\includegraphics[width=\linewidth]{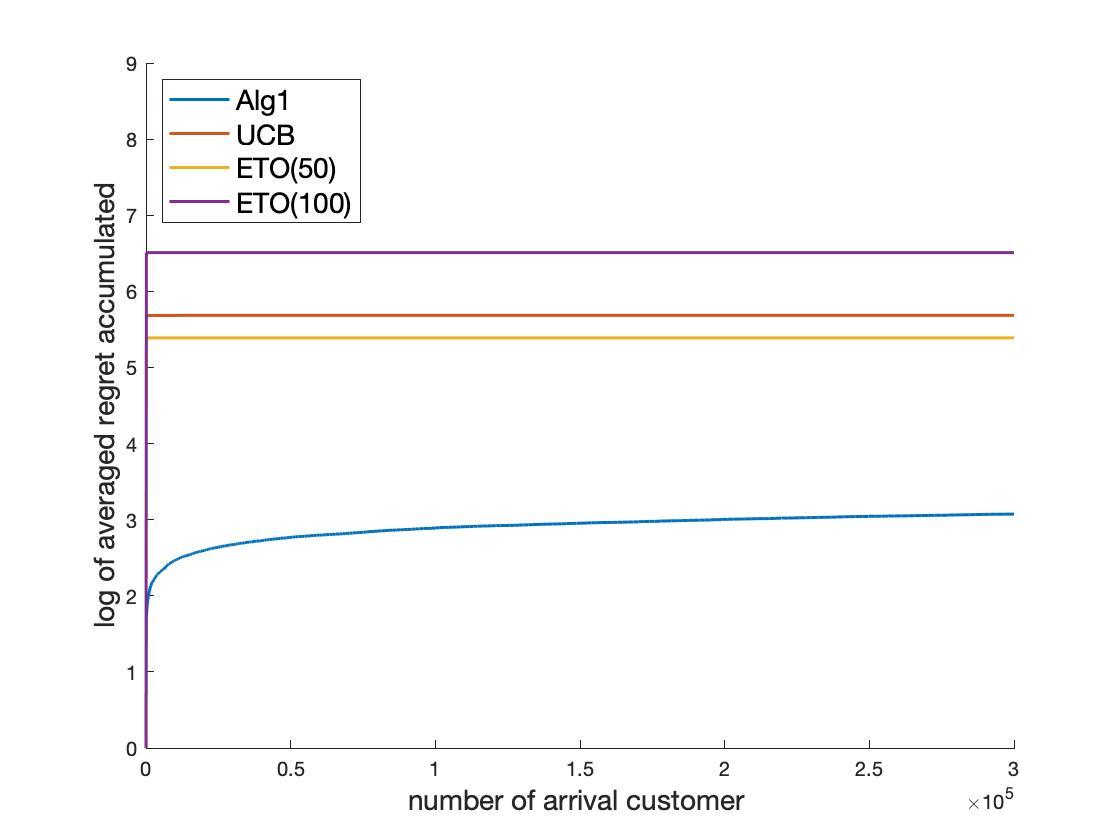}%@@@%{Fig8a.pdf}
			\caption{Log averaged regret plot when $\mu = 1$ and $\lambda = 1$. Both $\bar{K} = 1$ and $\bar{K}-1 = 0$ are optimal thresholds.  }
			\label{fig_difAlg_NUt0_2}
		\end{subfigure}
  \begin{subfigure}{.48\textwidth}
			\centering
			\includegraphics[width=\linewidth]{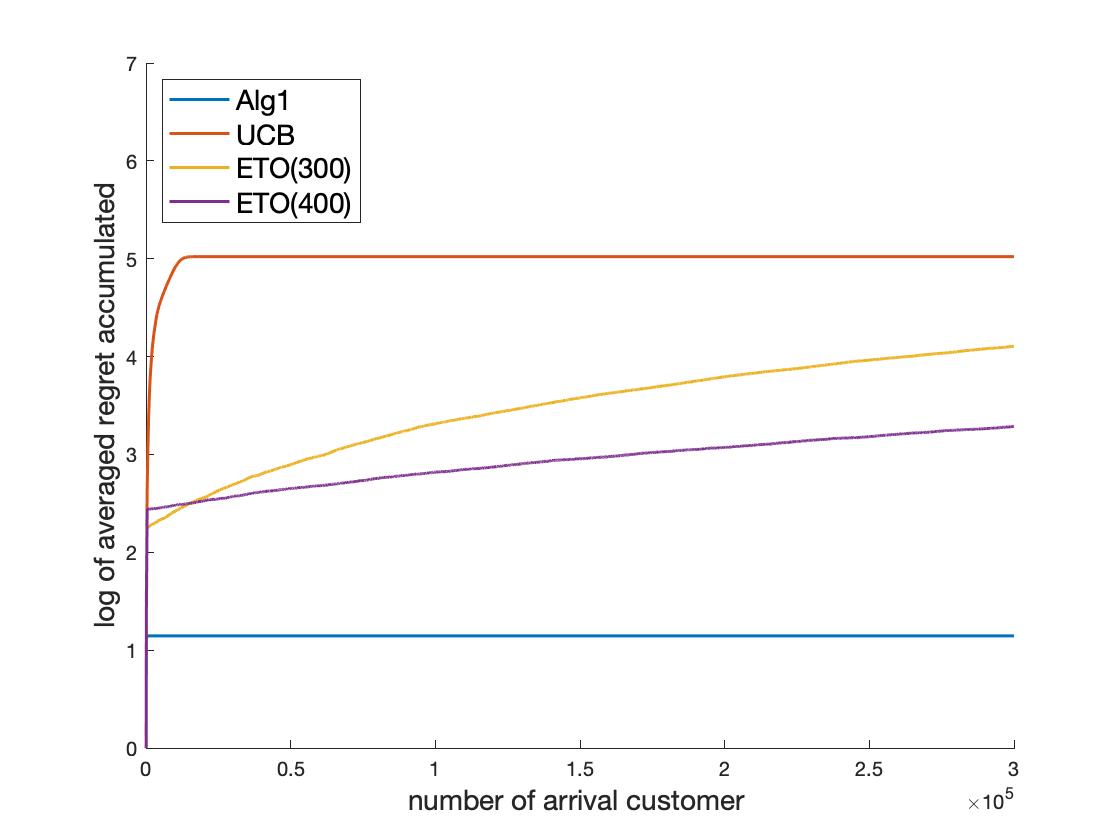}%@@@%{Fig8b.pdf}
			\caption{Log of average regret plot when $\mu = 2$, $\lambda = 1$ and $R = 129/32$. Both $\bar{K} = 5$ and $\bar{K}-1 = 4$ are optimal thresholds.}
			\label{fig_difAlg_NUt5_2}
		\end{subfigure}
		\caption{Log of regret accumulated when using different algorithms when the optimal thresholds are not unique. Alg1 is the learning algorithm proposed in Algorithm~\ref{alg:Alg3}. We set $C = 1$, $\mathbb{E}[B^j] = \ln(j)/j$ , $K^*(j)\sim \ln(j)$ and $\alpha_j = j$. ETO($M$) is the Estimate-Then-Optimize algorithm that always accepts the first $M$ customers. UCB is the  Upper Confidence Bound algorithm. %\VS{The regret of UCB ins (b) seems to be growing much slower than Alg1. Could you run for longer?} 
        }
  %\YZ{it turns out that in the case of both 0 and 1 are optimal thresholds, UCB performs better than our algorithm, since it would end up using threshold 1 (because of the UCB choice) yet out algorithm would not end up always using threshold 1, and then phase 1 may happen so the regret accumulates. I have added this to the discussion as well. }
		\label{fig_difAlg_3}

	\end{figure}

 %\begin{figure}[htbp]
%		\centering
%		\begin{subfigure}{.48\textwidth}
%			\centering
			%\includegraphics[width=\linewidth]{Jun_NU_t4_newUCB_comp.jpg}
			%\caption{Average regret plot. }
			%\label{fig_difAlg_NUt5_1}
		%\end{subfigure}%
		%\begin{subfigure}{.48\textwidth}
			%\centering
			%\includegraphics[width=\linewidth]{Jun_NU_t4_newUCB_comp_les.jpg}
			%\caption{Average regret plot comparing Alg1 with UCB. }
			%\label{fig_difAlg_NUt5_2}
		%\end{subfigure}
		%\caption{Regret accumulated when $\mu = 2$, $\lambda = 1$ and $R = 129/32$. Both $\bar{K} = 5$ and $\bar{K}-1 = 4$ are optimal thresholds. Alg1 is the learning algorithm proposed in Algorithm~\ref{alg:Alg3}. We set $C = 1$, $B^j = \ln(j)/j$ , $K^*(j)\sim \ln(j)$ and $\alpha_j = j$.  ETO($M$) is the Estimate-Then-Optimize algorithm that always accepts the first $M$ customers. UCB is the  Upper Confidence Bound algorithm. \VS{Here you will need to run UCB much longer to see linear regret since $\mu=2 > C/R=32/129$.}}
		%\label{fig_difAlg_4}
	%\end{figure}
 
\Added{
\paragraph{\bfseries Comparison with benchmark algorithms:}
 We also compared the finite time performance of our proposed Algorithm~\ref{alg:Alg3} with a few benchmark algorithms. In Figure~\ref{fig_difAlg_1} and Figure~\ref{fig_difAlg_3} we compared Algorithm~\ref{alg:Alg3}  with the Estimate-Then-Optimize (ETO) algorithm and the Upper Confidence Bound (UCB) algorithm when there are more than $3*10^5$ arrival customers and the regrets are averaged over 2000 simulations. We use ETO($M$) to denote the ETO algorithm which always accepts the first $M$ customers. We use the UCB algorithm described in \cite[Section 7.1]{banditalg} but with UCB bias subtracted from the estimated average service time.  Figure~\ref{fig_difAlg_t0_2} plots the log of average regret for the case when $\mu = 0.8$, $\lambda = 1$ and the optimal threshold is 0. Figure~\ref{fig_difAlg_t8_2} plots the log of average regret for the case when $\mu = 3$, $\lambda = 3.5$ and the optimal threshold is 8.  For the parameters used in these two plots, the optimal threshold is unique.
 Figure~\ref{fig_difAlg_NUt0_2} plots the average regret for the case when $\mu = 1$, $\lambda = 1$ and the optimal thresholds are $\{1,0\}$.
 Figure~\ref{fig_difAlg_NUt5_2} plots the average regret for the case when $\mu = 2$, $\lambda = 1$ and the optimal thresholds are $\{5,4\}$. For the parameter choices in Figure~\ref{fig_difAlg_3}, the optimal threshold is not unique. The regret values in these two plots are computed with respect to the alternating genie-aided system which would change the threshold used between $\{\bar{K}, \bar{K}-1\}$ according to the threshold used by Algorithm\ref{alg:Alg3}, ETO or UCB. 
 }
 \begin{figure}[htbp]
		\centering
    \includegraphics[width=0.5\linewidth]{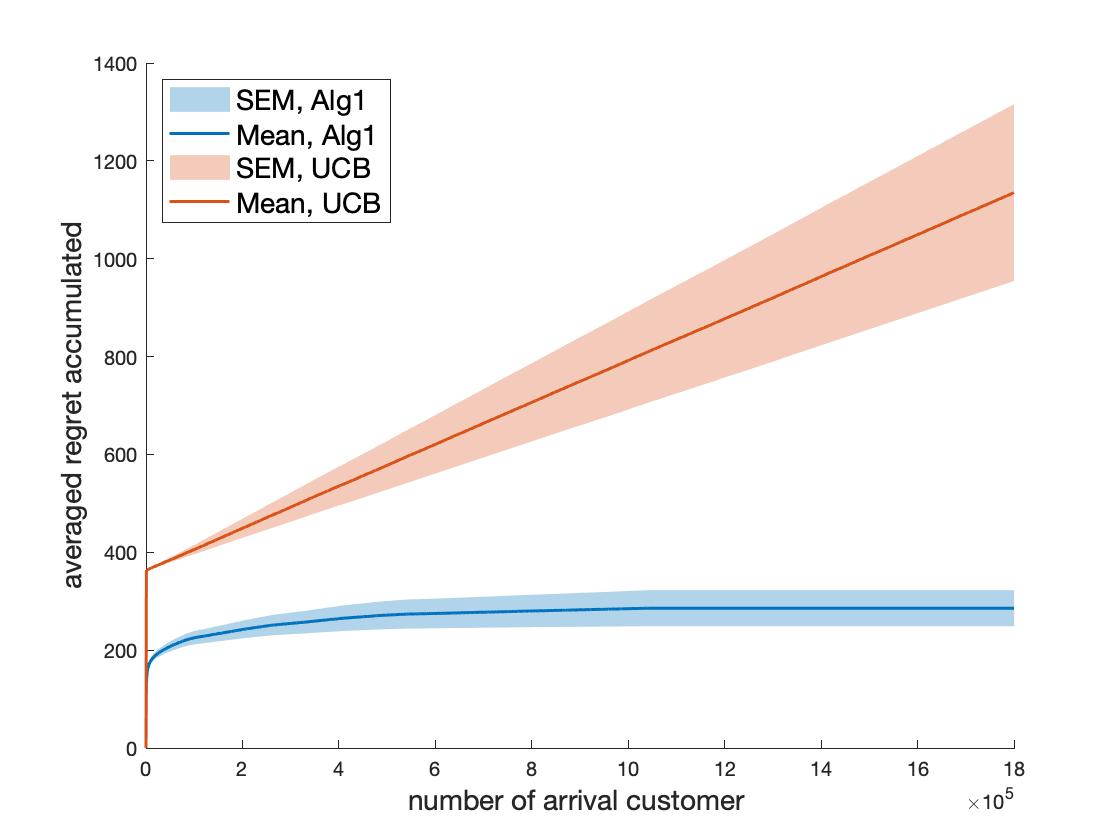}%@@@%{Fig9.pdf}
		\caption{ Regret accumulated when $\mu = 1.1$, $\lambda = 1$, $C = R = 1$ and $\bar{K} = 1$. Alg1 is the learning algorithm proposed in Algorithm~\ref{alg:Alg3}. We set $l_1 = l_2 = 30$, $B^j = \ln(j)/j$ , $K^*(j)\sim \ln(j)$ and $\alpha_j = j$. UCB is the Upper Confidence Bound algorithm. }
		\label{fig_difAlg_5}
	\end{figure}
 
 \Added{
 The order of the regret accumulated by Algorithm~\ref{alg:Alg3} and UCB are similar in  Figure \ref{fig_difAlg_t8_2} and \ref{fig_difAlg_NUt5_2}. However, in  Figure~\ref{fig_difAlg_t0_2} and \ref{fig_difAlg_NUt0_2} where $0$ is an optimal threshold, UCB achieves constant regret yet Algorithm~\ref{alg:Alg3} achieves a sublinear regret. It is likely that the regret accumulated by Algorithm~\ref{alg:Alg3} would slowly increase as the number of arrivals increases and eventually becomes larger than the regret of the UCB algorithm. Our algorithm may choose to use threshold 0 and then a phase 1 may be enforced and regret accumulates because of this. In Figure~\ref{fig_difAlg_5}, we compared the finite time performance of our proposed algorithm with UCB when $\mu = 1.1$ and $\lambda = 1$ with 2000 simulations and more than $10^6$ arrival customers. In this case, 1 is the unique optimal threshold. As we can observe from Figure~\ref{fig_difAlg_5}, the regret of UCB increases in a (approximately) linear fashion, while our proposed algorithm is able to achieve constant regret. In fact, we can argue the following for UCB-based dispatching (under the simpler setting of the arrival rate being known): 
 \begin{enumerate}[leftmargin=*]
 \item When the optimal threshold(s) is positive, then some bad initial service time samples can result in the estimated threshold being $0$. This bad event happens with positive probability for all $\mu> \tfrac{C}{R}$ (the probability decreases to $0$ as $\mu\rightarrow\infty$). Whenever this bad event occurs, then the UCB-based dispatching algorithm stops dispatching customers, obtains no new service time samples, and incurs linear regret. 
 \item When $0$ is an optimal threshold, then the corresponding bad event of estimating the threshold as positive is more benign. This holds as dispatching more customers only results in more service time samples, which then help to correct inaccurate estimates. Hence, we expect to achieve a constant or slowly growing (sub-linear) regret. 
 \end{enumerate}
 Note that the explanation above supports the conjecture in Remark~\ref{rem:conjecture} since the worst-case (over parameters) regret of UCB is expected to be linear in $N$. Moreover, since UCB needs to compute the estimated threshold at every arrival, it requires more computation when compared to Algorithm~\ref{alg:Alg3}.  
 }
 
 \begin{figure}[htbp]
        \centering
		
		\begin{subfigure}{.48\textwidth}
			\centering
			\includegraphics[width=\linewidth]{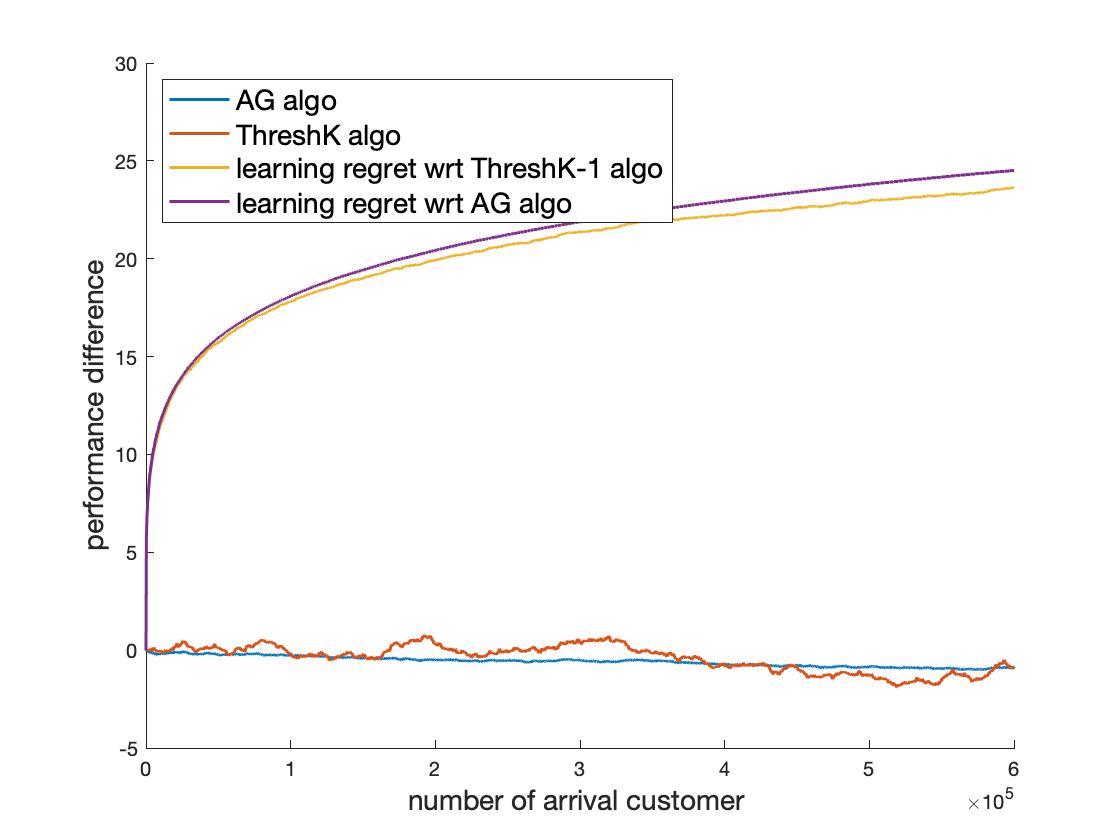}%@@@%{Fig10a.pdf}
			\caption{$\mu = \lambda = C = R = 1$. Both $\bar{K} = 1$ and $\bar{K}-1 = 0$ are optimal thresholds.}
			\label{com_genie_t0}
		\end{subfigure}
  \begin{subfigure}{.48\textwidth}
			\centering
			\includegraphics[width=\linewidth]{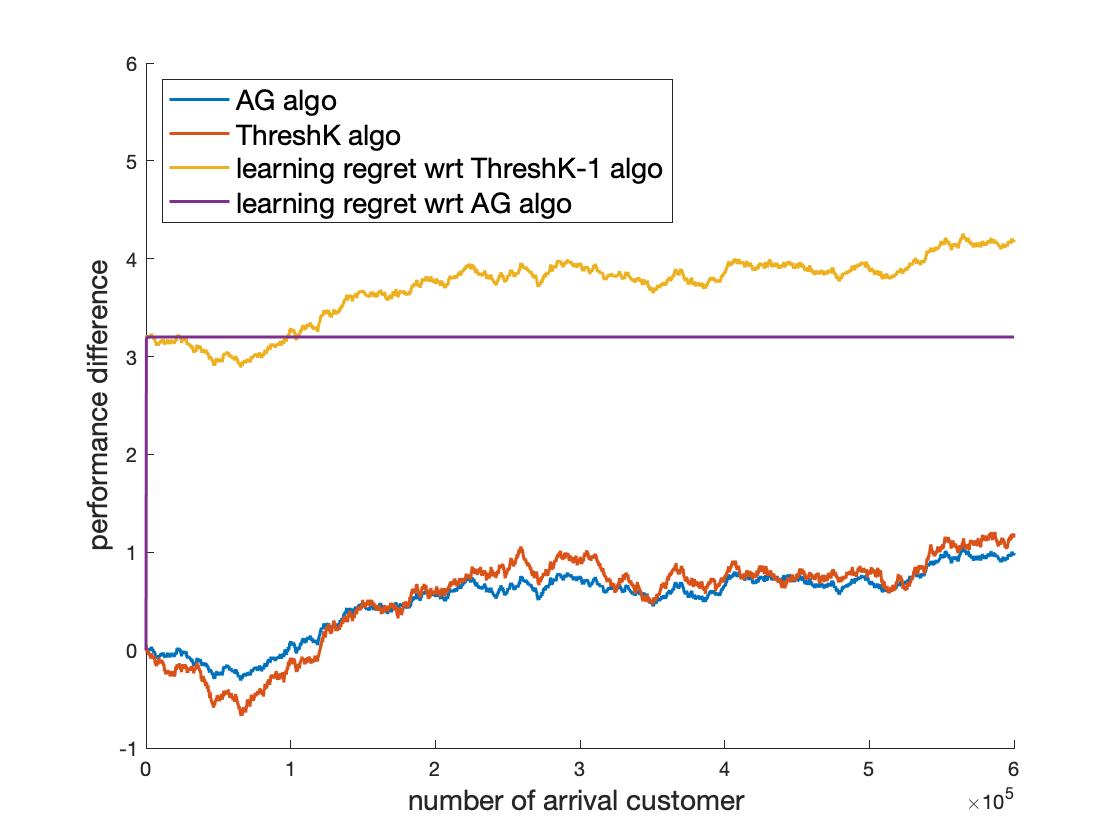}%@@@%{Fig10b.pdf}
			\caption{$\mu = 2$, $\lambda = 1$ and $R = 129/32$. Both $\bar{K} = 5$ and $\bar{K}-1 = 4$ are optimal thresholds.}
			\label{com_genie_t1}
		\end{subfigure}
		\caption{ Performance difference between the alternating genie, the genie algorithm using threshold $\bar{K}$ and the genie algorithm using threshold $\bar{K}-1$. The accumulated  net gain of   the genie algorithm using threshold $\bar{K}-1$  is scaled to be 0. 
        }
  
		\label{fig_comp_genie}
	\end{figure}
\Added{\paragraph{\bfseries Comparison of different genie-aided algorithms:}Figure~\ref{fig_comp_genie} compares the accumulated net gain between the alternating genie-aided algorithm (``AG algo" in the legend) coupled with Algorithm~\ref{alg:Alg3} and the genie-aided algorithms using threshold $\bar{K}$ (``ThreshK algo" in the legend) or $\bar{K}-1$ (``ThreshK-1 algo" in the legend) when optimal thresholds are not unique; the accumulated net gain of the genie-aided algorithm using threshold $\bar{K}-1$ are scaled to be 0.  Figure~\ref{fig_comp_genie} plots the difference between the net gain obtained by the alternating genie-aided system and the genie-aided system using static threshold $\bar{K} -1$, and the difference of the net gain between two genie-aided systems using static threshold $\bar{K}$ and $\bar{K}-1$ over two sets of parameters. We also include the regret accumulated by the learning algorithm compared with the genie-aided algorithm using threshold K-1. The performances of the algorithms are averaged over $18000$ simulations. As we can observe from the plots, the regret accumulated by the learning algorithm (with respect to either the alternating genie-aided system or the genie-aided system using threshold $K-1$) dominates the performance difference between the alternating genie-aided system and the genie-aided system using threshold $K-1$ , and the performing difference between the genie-aided system using threshold $K$ and the genie-aided system using threshold $K-1$. This is more evidence in favor of Remark~\ref{rem:optimal}. 
 }
 \section{Conclusions}\label{section:Conclusion}
	In this paper, we considered a social welfare maximizing problem, which was first proposed and studied %by Naor 
    in \cite{Naor}. We studied the learning problem of finding the proper threshold admission policy when the service and arrival rates are unknown. We proposed a learning algorithm that consists of batches where each batch has an optional exploration phase with a fixed length and an exploitation phase.  When the optimal policy is unique, we showed that our learning algorithm achieves an $O(1)$ regret whenever the optimal threshold is non-zero, and achieves an $O(\ln^{1+ \epsilon}(N))$ regret when the optimal threshold is zero, where $N$ denotes the total number of arrival customers to the systems.  When the optimal policy is not unique, we specified a particular optimal policy to compare with, and proved that similar regret bounds hold for our learning algorithm.

	 In our analysis, we assumed Poisson arrivals and exponentially distributed services with fixed arrival and service rate. We would like to adapt our algorithm to more general arrival processes and service-time distributions like the models in  \cite{Lippman76} and \cite{Johansen80}, so that a small regret is obtained in these more general settings too, \Added{such as generalization to optimal admission control in an $M/G/1$ queue with our information structure. %\VS{with our information structure}. %The analytical optimal strategy for this problem is still unknown and may be time-varying as suggested in \cite{mg1queue}, and the learning algorithm may not be easily adapted.
    This problem has received attention---see \cite{mg1queue}---under a different information structure where only the queue-length is observed by arrivals. Under this setting, the analytical optimal strategy for this problem is still unknown and may be time-varying; see \cite{mg1queue} for details. However, the problem may be tractable with our information structure as the Markov state---number in service and service time elapsed of customer currently being served---is observable and MDP theory could be applied. %\VS{Change to the following: This problem has received attention---see \cite{mg1queue}---under a different information structure where only the queue-length is observed by arrivals. Under this setting, the analytical optimal strategy for this problem is still unknown and may be time-varying; see \cite{mg1queue} for details.}\YZ{changed.}
     } %\VS{The information structure in \cite{mg1queue} is not the same as our paper! There admission policy can only depend on queue-length. In our paper, you also know the following---time in service of customer being served, and so the distribution of residual service time---in addition to completed service times. This problem is not solved but will be a more complex MDP with state being two-dimensional---number in service and service time elapsed of customer currently being served.} 
     Another possible direction is to consider a single queue with a buffer but with multiple servers like the model in \cite{Knudsen}.  Again, the aim would be to adapt our current learning algorithm to this setting as well, whilst achieving low regret. \Added{Finally, we conjectured that the order of the regret accumulated for the worst case choice of parameters would grow at least as $\Omega(\ln(N))$; see Remark \ref{rem:conjecture}. %\VS{See change to $\Omega(\ln(N))$.}\YZ{Thank you.} 
     Proving (or disproving) this conjecture is yet another problem for future work.
  }  

\noindent{\bfseries Acknowledgement:} We thank the anonymous AE and the referees for their insightful comments, which helped us improve our paper.
 
\bibliographystyle{informs2014}
\bibliography{ref4}

\end{document}